\documentclass[12pt]{article}
\pdfoutput=1
\usepackage{authblk}
\usepackage{caption}
\usepackage{subcaption}
\usepackage[bookmarks=false]{hyperref}

\usepackage{amsmath}
\usepackage{amsfonts}
\usepackage{amssymb}
\usepackage{amsthm}
\usepackage{mathrsfs}
\usepackage{mathtext}
\usepackage{mathtools}
\usepackage[mathscr]{eucal}
\usepackage{geometry}
\numberwithin{equation}{section}
\usepackage{a4wide}

\usepackage[pdftex]{color,graphicx}
\usepackage[english]{babel}
\usepackage{graphics}
\usepackage{enumerate}
\usepackage{enumitem}
\usepackage{array}
\usepackage{verbatim}
\usepackage{hyperref}
\usepackage{caption}
\usepackage{tikz-cd}
\usepackage{float}
\usepackage[capitalise]{cleveref}
\usepackage{graphicx}
\usepackage[export]{adjustbox}
\usepackage{csquotes}

\usepackage{pdflscape}
\usepackage{rotating}
\usepackage{graphicx}

\usepackage[all,knot]{xy}
\usepackage{tikz}
\tikzcdset{scale cd/.style={every label/.append style={scale=#1},
    cells={nodes={scale=#1}}}}
\usetikzlibrary{arrows}
\usetikzlibrary{shapes.geometric}
\usetikzlibrary{shapes,backgrounds}
\tikzset{every picture/.style={line width=0.75pt}}

\theoremstyle{definition}

\newtheorem{theo}{Theorem}[section]

\newtheorem{defi}[theo]{Definition}
\newtheorem{rema}[theo]{Remark}
\newtheorem{prop}[theo]{Proposition}
\newtheorem{coro}[theo]{Corollary}
\newtheorem{lemma}[theo]{Lemma}
\newtheorem{exam}[theo]{Example}

\newtheorem{conj}[theo]{Conjecture}
\newtheorem{conv}[theo]{Convention}
\newtheorem{cons}[theo]{Construction}

\newtheorem*{theorem**}{Theorem\theoremnum}
\newenvironment{theorem*}[1][]{%
  \edef\theoremnum{\if\relax\detokenize{#1}\relax\else~#1\fi}
  \begin{theorem**}
}{%
  \end{theorem**}
}  
  
\usepackage{url}
\makeatletter
\newcommand{\address}[1]{\gdef\@address{#1}}
\newcommand{\email}[1]{\gdef\@email{\url{#1}}}
\newcommand{\@endstuff}{\par\vspace{\baselineskip}\noindent\small
\begin{tabular}{@{}l}\scshape\@address\\\textit{E-mail address:} \@email\end{tabular}}
\AtEndDocument{\@endstuff}
\makeatother

\title{\textbf{Coloring Trivalent Graphs: A Defect TFT Approach}}
\author{Amit Kumar}
\address{Department of Mathematics, \\ Louisiana State University, Baton Rouge LA, USA.}
\email{akuma25@lsu.edu}
\date{}

\begin{document}

\maketitle

\begin{abstract}

   We show that the combinatorial matter of graph coloring is, in fact, quantum in the sense of satisfying the sum over all the possible intermediate state properties of a path integral. In our case, the topological field theory (TFT) with defects gives meaning to it. This TFT has the property that when evaluated on a planar trivalent graph, it provides the number of Tait-Coloring of it. Defects can be considered as a generalization of groups. With the Klein-four group as a 1-defect condition, we reinterpret graph coloring as sections of a certain bundle, distinguishing a coloring (global-sections) from a coloring process (local-sections.) These constructions also lead to an interpretation of the word problem, for a finitely presented group, as a cobordism problem and a generalization of (trivial) bundles at the level of higher categories.
   
\end{abstract}

\tableofcontents

\section{Introduction}
\subsection{Motivation and Main Result} \label{sec:motivation}

Given an un-directed graph $\Gamma$, and a discrete set $X$, an \textit{admissible edge coloring} of $\Gamma$, with values in $X$, is an assignment such that:

\begin{enumerate}

    \item Every edge of $\Gamma$ gets assigned an elements of $X$,
    \item Each edge, meeting at a vertex, gets assigned different elements of $X$.

\end{enumerate}

Note, we do not require $\Gamma$ to be embedded in some surface $\Sigma$. Therefore, an admissible coloring can be thought of as an assignment to the edges of an abstract graph, which is given in terms of a set of vertices, $V(\Gamma)$, a set of edges, $E(\Gamma)$ and a set of incidence relations, i.e., a map $\mathcal{I} $ from the set of edges, $E(\Gamma)$ to the set of unordered pair of vertices of $\Gamma$, that maps every edge to its end-points. The function $\mathcal{I}$ gives rise to another function $\mathcal{V}:V(\Gamma) \to \mathbb{Z}_{+}$, whose value at a vertex $v \in V(\Gamma)$ is called the valency of $v$. A graph is called $n$-regular if all vertices have valency $n$.

A graph is called planar if it can be embedded in $\mathbb{S}^2$, and it is trivalent (or cubic) if the valency at every vertex is three. An admissible edge coloring of a planar graph by a set of three elements $X$ is called a \textit{Tait-coloring} or a $3$-\textit{edge-coloring}, and the number of such distinct admissible colorings is called the \textit{number of Tait coloring.}

The term \textit{admissible edge coloring} is inspired by an analogous concept in ~ \cite{khovanov2021foam}, and the definition of Tait coloring as well as the definition of number of Tait coloring for a planar trivalent graph is in the spirit of Penrose from the '70s in ~ \cite{penrose1971applications} and recent works like ~ \cite{baldridge2018cohomology} and ~ \cite{ baldridge2023topological}.

\begin{figure}[h]
    \centering
    \includegraphics[width=0.98\linewidth]{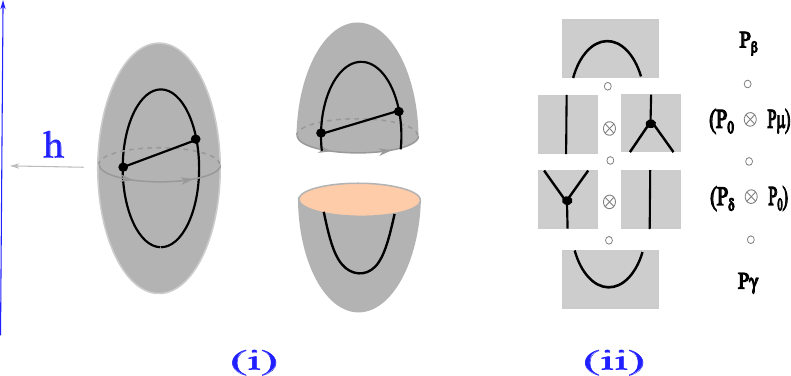}
    \caption{Diagram (i) shows that after choosing a height function, the coloring of a planar trivalent graph can be thought of as a sum over intermediate states. Diagram (ii) shows that it can be further decomposed, each piece can be admissibly colored separately and glued back with the hope of getting a coloring on the original graph.}
    \label{fig:quantum_decom}
\end{figure}

Next, we note that given a plane trivalent graph $\Gamma$, i.e. a graph $\Gamma$ embedded in $\mathbb{S}^2$, the number of Tait-coloring of $\Gamma$ can be obtained by decomposing a graph into two pieces as in \cref{fig:quantum_decom}, (I), coloring each piece admissibly, and gluing them back. If the color on the edges matches at the two boundary circles, keep it, discard it otherwise. This is one of the key properties of path integrals as discussed in [~\cite{hori2003mirror}, Section 8.3 ], [~\cite{Carqueville_2018}, Section 2.1(v)] or [~\cite{freed1993lectures}, Lecture 1], namely, \textit{summing over all possible field configurations on the glued boundary.}
In the context of \cref{fig:quantum_decom}, (i), after choosing a height function, the possible field configuration on the boundary circle is given by all possible assignments of colors to the endpoint of edges of the theta graph; the sum evaluates to zero if the color on the boundary does not match. Thus, a natural question arises: Does there exist a functorial quantum field theory whose partition function is the number of Tait-coloring of $\Gamma$? More precisely, does there exist a symmetric monoidal category $\mathcal{C}$, and a functor $ \mathcal{Z}: \mathcal{C} \to \text{Vect}_{\mathbb{K}} $ such that $\mathcal{Z}(\mathbb{S}^2, \Gamma)$ is a scalar, which is equal to the number of ways $\Gamma$ can be Tait-colored? Of course, this requires $(\mathbb{S}^2, \Gamma) \in \mathcal{C}$ in some sense.

It turns out that the category $\mathcal{C}$ is the category $\text{Bord}^{\textit{def, cw}}_2(\mathcal{D}_{+}^{\mathbf{3}})$ with objects as a disjoint union of circles with marked points, together with a germ of a collar, and morphisms between two objects as decorated stratified surfaces (see \cref{Bord_def} for a picture) bounding them. First, in \cref{sum_over_IS} we prove the above-mentioned gluing property of the Tait coloring of a planar trivalent graph. Next, after reinterpreting the concept of coloring as sections of a certain bundle, we proved the desired result:

\begin{theorem*}[\ref{main-2}]

Let $\Gamma$ be a trivalent graph embedded in $\mathbb{S}^2$. Consider the surface with defect $(\mathbb{S}^2, {\Gamma})$ in $\text{Mor}(\textit{Bord}_2^{\textit{def,cw}}(\mathcal{D}_{+}^{\mathbf{3}}))(\emptyset, \emptyset)$. The action of the functor $\chi^{cw}$ on $(\mathbb{S}^2, \Gamma)$ is the assignment
    \begin{equation*}
        \begin{split}
            \chi^{cw}(\mathbb{S}^2, \Gamma) &: \mathbb{C} \longrightarrow \mathbb{C}\\
            & \lambda \mapsto \#\text{Tait}(\Gamma) \lambda
        \end{split}
    \end{equation*}
    
  In other words, the number $\chi^{cw}(\mathbb{S}^2, \Gamma)(1)$ is the number of Tait-coloring of the planar trivalent graph $\Gamma$. 

\end{theorem*}

The following section outlines how this paper's main result(s) has been achieved.

\subsection{Tait's correspondence and Defects } \label{sec:intro_Tait}
The biggest inspiration comes from the work of Tait ~\cite{tait1880note}, who gave a way to vertex four-color a planar graph, given the admissible $3$-edge coloring of its dual trivalent graph. Given a planar trivalent graph $\Gamma$, Tait's correspondence provides a way to color the regions, i.e., the components of $\mathbb{S}^2 \setminus \Gamma$ by the four elements of $K_4 = \langle a, b, c \mid a^2, b^2, c^2, abc \rangle$ given an admissible $3$-edge coloring of $\Gamma$ by $X = \{a, b, c\}$, the non-trivial elements of $K_4$. \cref{fig:Tait_sheaf}, (i), is a demonstration of this procedure. We have chosen to color the unbounded region by $1$. However, one can arbitrarily select a region, and color, i.e., an element of $K_4$ and to color a neighboring region, multiply by the assignment on the edges of $\Gamma$ when crossing it (cf ~\cite{baldridge2023topological}).

\begin{figure}[h]
    \centering
    \includegraphics[width=0.98\linewidth]{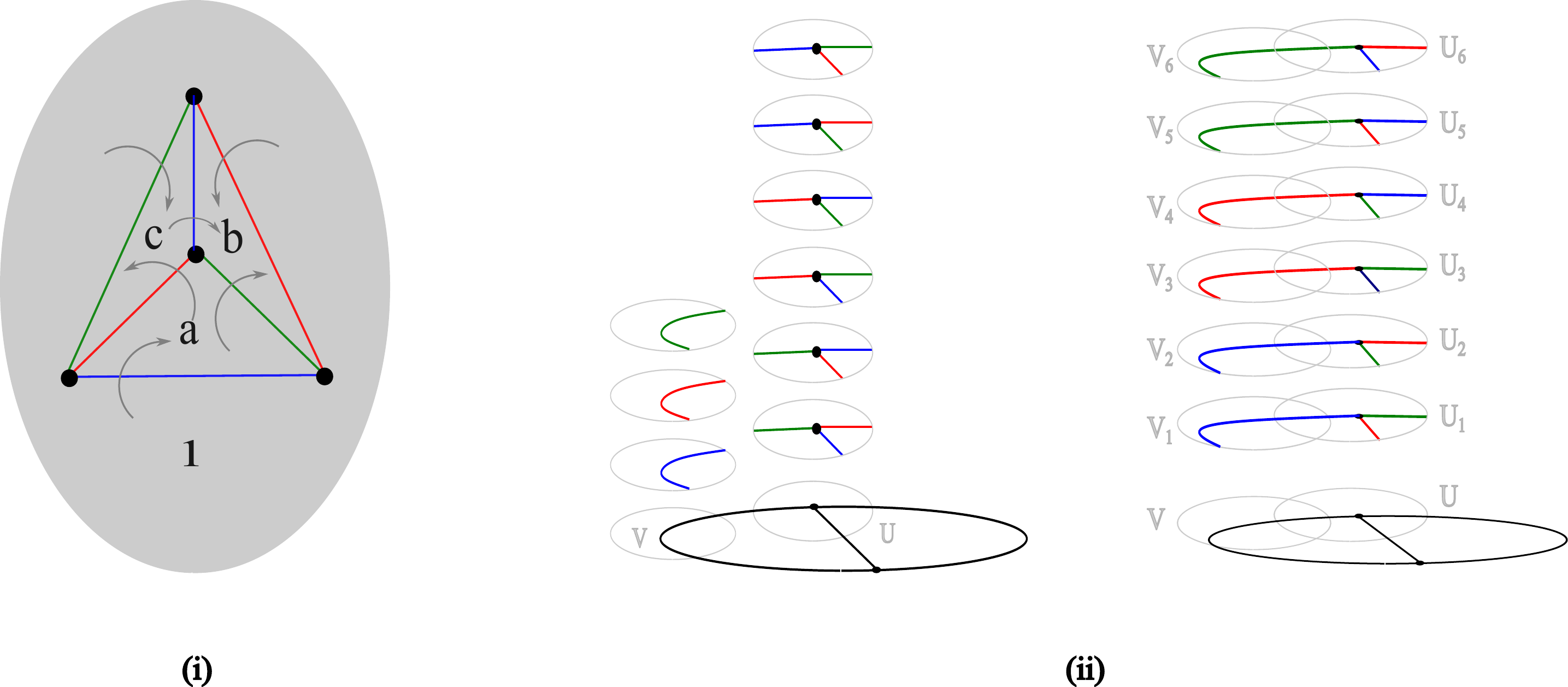}
    \caption{The blue color stands for $a$, red for $b$ and green for $c$ in Klein-four group. Diagram (i)  shows an example of Tait's correspondence. Diagram (ii) shows that the coloring of a graph embedded in a surface satisfies the sheaf property: only those assignments survive that agree on the overlap.}
    \label{fig:Tait_sheaf}
\end{figure}

It turns out that what Tait is describing is an example from the theory of defects [~\cite{carqueville2023topological}, Introduction, Paragraph 2] when the $1$-strata are labeled by non-identity elements of $K_4$ [cf ~\cite{carqueville20203}, Section 4.3]. This term, \textit{defects}, has its origin in physics and their presence reflects the phenomenon that the symmetry may be broken in different ways in different regions of the system [~\cite{lancaster2014quantum}, Chapter 26, Chapter 29]. A topological defect generalizes the idea of a symmetry group [~\cite{cordova2022snowmass}, ~\cite{schafer2024ictp}, ~\cite{freed2022topological}], and defects, topological or not, find their way into many areas of mathematics (and physics) [~\cite{bah2022panorama}, Section 2.9], [~\cite{freed2021quantum}]. The functorial field theory that we seek is an example of a \textit{topological field theory with defects}, which we discuss next.

\subsection{Topological Field Theory with Defects}

Now, we can state that the index, \textit{def}, in the notation $\text{Bord}^{\textit{def}}_2(\mathcal{D}_{+}^{\mathbf{3}})$ stands for defects, and the functor $\mathcal{Z}:\text{Bord}^{\textit{def, cw}}_2(\mathcal{D}_{+}^{\mathbf{3}}) \to \text{Vect}_{\mathbb{K}}$ is a \textit{field theory with defects}, which we need to be topological since the Tait-coloring is independent of planar isotopes. Referring to ~\cite{carqueville2019orbifolds}, a topological field theory with defects can be understood in terms of Extended TFTs, which is a higher functor ($n$-functor, $n>1$) into a higher symmetric monoidal target category. In this sense, a defect TFT can be thought of as enlarging the domain category instead, to accommodate all the enriched (higher categorical) structures in it while keeping the target category to be the symmetric monoidal $1$-category $\text{Vect}_{\mathbb{K}}$.

Therefore, in \cref{Bord}, we began by defining this domain category. We deviated from the mainstream literature to present the symmetric monoidal category $\text{Bord}^{\textit{def}}_n(\mathcal{D})$, for a set of defect conditions $\mathcal{D}$ (cf \cref{condns}), as the category of bordism between \textit{manifolds with defects}, i.e., the objects are a disjoint union of $(n-1)$-dimensional manifold with defects, each with a germ of a collar, and the set of morphism between two $(n-1)$-dimensional manifold with defects $\hat{N_1}$ and $\hat{N_2}$ is given by a $n$-dimensional manifolds with defects $\hat{M}$ with boundaries $(-\hat{N_1}) \sqcup \hat{N_2}$, where $(-\hat{N})$ stands for the opposite orientation on $N$. We drew inspiration from the observation that the edge coloring of a graph satisfies the sheaf property (\cref{fig:Tait_sheaf}, (ii)) and built the local models by drawing the string diagram for defect conditions. (Compare this approach with the definition of a pre-foam in [~\cite{khovanov2021foam}, Definition 2.1].) However, we have restricted ourselves to $n=2$ in this manuscript, and after developing their theory, we define a lattice TFT called, \textit{trivial surrounding theory} in \cref{sec:TFT} and construct a special trivial surrounding theory $\chi^{cw}$ in \cref{chi-cw} that provides the required partition function needed for \cref{main-2}. Since $\chi^{cw}$ is a lattice topological field theory, it uses a decomposition as in \cref{fig:quantum_decom}, (ii) to calculate the correlator, i.e., the value of $\chi^{cw}$ on a morphism. With the $\mathbb{C}$-vector space, spanned by $X = \{a, b, c\}$, which we also denote $X$ by the abuse of notation, the calculation in \cref{chi-cw} renders the following:

\[ 
\begin{matrix}
    
\chi^{cw}(P_{\beta}):
\begin{cases}
a \otimes a   \mapsto 1 \\
b \otimes b  \mapsto 1 \\
c \otimes c  \mapsto 1 \\
0 \hspace{2mm} \text{otherwise} & 
\end{cases}    

&

&

\chi^{cw}(P_{\mu}):
 \begin{cases}
 a \otimes b, b \otimes a  \mapsto c \\
 a \otimes c, c \otimes a  \mapsto b \\
 b \otimes c, c \otimes b  \mapsto a \\
 0 \hspace{5mm} \text{otherwise}
\end{cases}

\\
&

&

&
\\

\chi^{cw}(P_{\delta}):
\begin{cases}
 a & \mapsto  b \otimes c + c \otimes b \\
 b &  \mapsto  a \otimes c + c \otimes a \\
 c & \mapsto  a \otimes b + b \otimes a
\end{cases}

&

&

\chi^{cw}(P_{\gamma}): 1 \to a \otimes a + b \otimes b + c \otimes c

\end{matrix}
\]

Using these, together with $\chi^{cw}(P_0) = 1$, on \cref{fig:quantum_decom},(ii), we obtain,

\begin{multline} \label{partition_theta}
    1 \mapsto a \otimes a + b \otimes b + c \otimes c \mapsto a \otimes (b \otimes c + c \otimes b) + b \otimes (a \otimes c + c \otimes a) + c \otimes (a \otimes b + b \otimes a) \\
   a \otimes b \otimes c + a \otimes c \otimes b + b \otimes a \otimes c + b \otimes c \otimes a + c \otimes a \otimes b + c \otimes b \otimes a \mapsto 2 c \otimes c + 2 b \otimes b + 2 a \otimes a \mapsto 6 ,  
\end{multline}

which is equal to the number of Tait-coloring for the theta graph. However, we need to prove this in general, and this is where all the new mathematics lies. We discuss this next.

\subsection{Key ideas and tools} \label{sec:key_tools}

We begin by noticing that the correlators that the TFT $\chi^{cw}$ calculates should be invariant under planar isotopy if it is meant to calculate the number of Tait-coloring as mentioned in the previous section. Therefore, the first key step is to see a planar graph as a string diagram for a pivotal 2-category, which is a $2$-category with adjoints [cf ~\cite{carqueville2016lecture}, Section 2.2, ~\cite{barrett2024graycategoriesdualsdiagrams}, Section 2.1-2.3, Section 3.1-3.2]. In that case, \cref{prop-deco} proves that every planar trivalent graph can be viewed as assembled from simpler pieces as in \cref{fig:quantum_decom}, (ii). We emphasize that the collar plays a bigger role than just facilitating the gluing since there are cross-sections of a surface with defects that do not admit a collar. Such cross-sections are not objects in the category $\text{Bord}^{\textit{def}}_2(\mathcal{D})$ and this has to be taken care of when decomposing a graph.

Next, referring to \cref{partition_theta}, it is unclear how $a, b, c$ are connected with edge-coloring. Taking ideas from Tait's correspondence as discussed in \cref{sec:intro_Tait} we redefined the notion of coloring in \cref{sec:graph-coloring}. This is accomplished in two steps: First, building on the idea developed in ~\cite{casals2023legendrian}, given a finitely presented group $G$ with the presentation $P_G$, we define the category $\text{Bord}^{\textit{def}}_2(\mathcal{P_G})$ in \cref{sec:groups}. Next, use the presentation $ P_{K_4} \coloneqq \langle a, b, c \mid a^2 = b^2 = c^2 = abc = 1 \rangle$ for the Klein-four group to build the category $\text{Bord}^{\textit{def}}_2(\mathcal{P}_{K_4})$. The morphism set of this category can be thought of as accommodating all trivalent graph $\Gamma$, embedded in some surface $\Sigma$, which can be $3$-edge colored.

Once we have a category whose morphism set comprises all $3$-edge colored graphs, we can define a $3$-edge coloring of a trivalent graph in terms of lifts to this set. However, such a lift may not exist for the entire graph. \cref{big-1} gives an obstruction for planar graphs and also validates Tait's correspondence. In fact, \cref{big-1} generalizes the classical result that a planar trivalent graph with a bridge is not $3$-edge colorable. The generalized statement and the proof, both, require the group property of $K_4$.

Since, we can not start with a trivalent graph, embedded in some surface, and hope to find a lift in $\text{Bord}^{\textit{def, cw}}_2(\mathcal{P}_{K_4})$, we should look for a lift only locally and ask when such local lifts give rise to a global lift. In layman's language, this distinguishes a coloring from a coloring process. A coloring process is a local assignment of sections and a coloring is a global section. A coloring process leads to a coloring precisely when a global section extends all the local sections. This has been elaborated in \cref{sec:coloring_process} and the hidden abstract structure in the definition of a coloring process has been outlined in  \cref{sec:additional}.

\subsection{Word problem and other connections}

We see that the validity of \cref{big-1} rests on associating to a circle with defects in $\text{Bord}^{\textit{def, cw}}_2(\mathcal{P}_{K_4})$ an element of $K_4$. Moreover, it is an immediate corollary of \cref{big-1} that a single circle with defects labeled with a single point labeled $c$ can not be cobordant to a circle with two marked points both labeled $a$, as long as this cobordism has genus $0$ (see \cref{exam:planar_word-problem} and \cref{fig:planar_word-problem}). This leads to an interpretation of the word problem as a cobordism (with defects) problem via the category $\text{Bord}^{\textit{def}}_2(\mathcal{P_G})$, for a finitely presented group $G$ with the presentation $P_G$. This is discussed in \cref{sec:word-problem_theory}, and \cref{sec:word_problem_future} along with many other future directions that stem from this work in \cref{sec:future_dirn}.

\subsection{Acknowledgement}

First, the author would like to thank his advisor Scott Baldridge for patiently and passionately listening to the raw ideas and for his constant support. Next, the author is indebted to Ingo Runkel for his powerful suggestions, and Nils Carqueville for many fruitful discussions. Finally, the author wants to thank David Truemann, Eric Zaslow, Roger Casals, Mikhail Khovanov, and Louis-Hadrien Robert whose works have been an inspiration for this project. In particular, the author thanks Mikhail Khovanov and Anton Zeitlin for many discussions in the early stages of this project. The author also thanks Kevin Schreve, Jerome Hoffman, and Daniel Cohen for many informal but valuable discussions about this project.


\section{Category of 2-dimensional bordism with defects} \label{Bord}
The goal of this section is to review the category of smooth bordism with defects. We essentially follow ~\cite{davydov2011field}, but give new definitions and modify some old ones to set the ground for work in the subsequent sections. A general theory was developed in ~\cite{carqueville20203}, Section-2 and in ~\cite{carqueville2019orbifolds}. We only deviate from these in some terms, conventions, and notations. For instance, by $\textit{Bord}_2^{\textit{def}}$ we mean the category $\textit{Bord}_2^{\textit{def, top}}$, but omit the word 'top' as we are dealing with \textit{topological defects} throughout this manuscript. Similarly, we use symbols that keep track of various incidences among strata.

The category $\textit{Bord}_2^{\textit{def}}(\mathcal{D})$ constitutes objects and morphisms that are \textit{stratified spaces} with each stratum labeled with elements of sets called \textit{defect conditions}. We explore each of these concepts in the following subsections.

\subsection{Defect Conditions}


\begin{defi} \label{adm}
Given an $n$-dimensional oriented manifold $M$, and a finite collection $\mathfrak{S} = \{M_0, \dots , M_n \}$ of submanifolds of $M$, we say that $\mathfrak{S}$ is an \textit{admissible decomposition} of $M$ if the following conditions hold:
\begin{enumerate}
    \item (\textbf{covering}) $M = \bigcup_{i=0}^{n}M_i$,
    \item (\textbf{decomposition}) $\text{dim}(M_i) = i$, $M_i \cap M_j = \emptyset$ for $i \neq j$, and the orientation of $M_n$ is induced by the orientation of $M$, and
    \item (\textbf{admissibility}) each partial union $\bigcup_{i = 0}^{k}M_i$ is a closed subset of $M$ for every $k \leq n$.
\end{enumerate}
\end{defi}

\begin{exam} \label{exam-deco}
    \begin{enumerate}
         \item Let $M$ be $n$-dimensional, the collection $\mathfrak{S}_0 = \{M\}$, which means $M_k = M$ if $k = n$, and $M_k = \emptyset$ otherwise, is an admissible decomposition of $M$.
        \item Let $U = \mathbb{S}^1$ denote the unit circle in the complex plane. Form the collection $\mathfrak{U}$ with $U_0 = \{p,q, r, s\}$, where $p = 1, \hspace{1mm} q = i, \hspace{1mm} r = -1,\hspace{1mm} s = - i$, and $$U_1 = \mathbb{S}^1 \setminus M_0 = \{(p,q), (q,r), (r,s), (s,p)\}.$$ Then, $\mathfrak{U}$ is an admissible decomposition of $\mathbb{S}^1$. See \cref{0-1-decom}.

        \begin{figure}
            \centering
            \includegraphics[width=0.98 \linewidth]{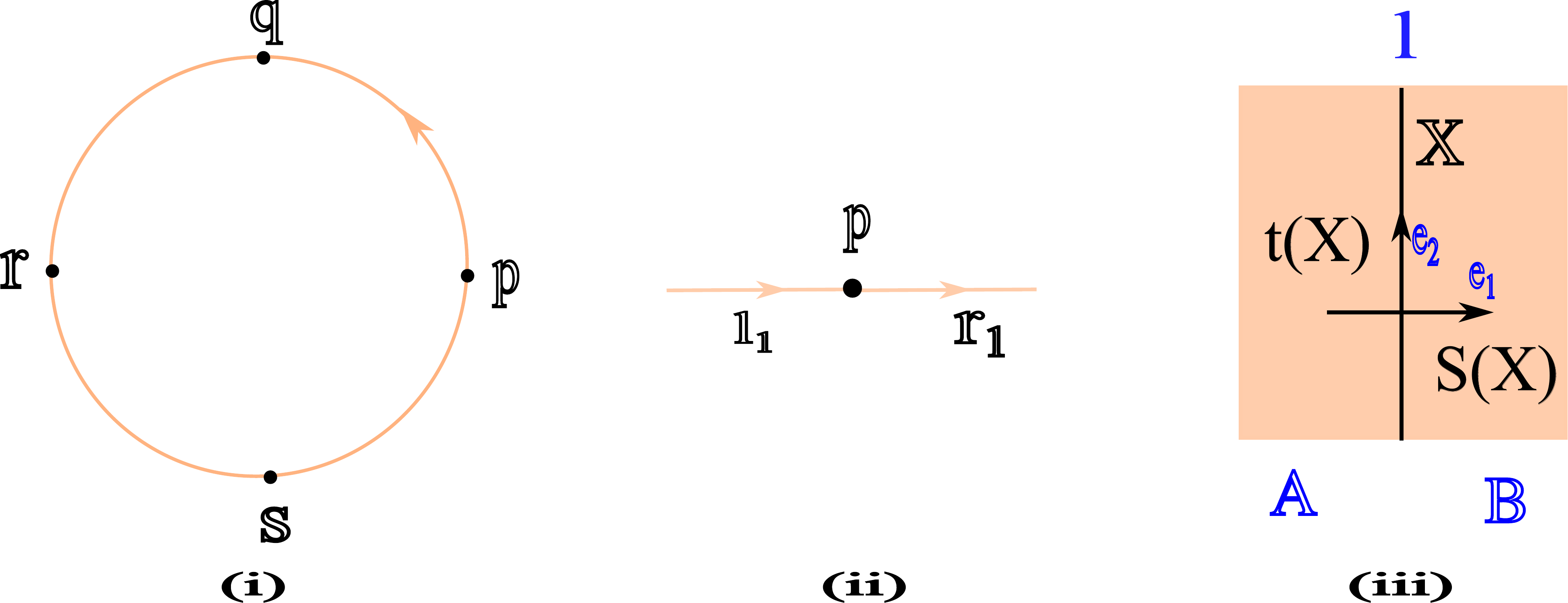}
            \caption{Diagram (i) shows an admissible decomposition $\mathfrak{U} = \{U_0, U_1\}$ of $\mathbb{S}^1$, where $U_0 = \{p,q, r, s\}$ and $U_1 = \mathbb{S}^1 \setminus U_0$. Diagram (ii) illustrates the condition \cref{defcirc} (4), and diagram (iii) illustrates the condition \cref{surfdef} (4.1). We have interpreted the sign 
            $\epsilon = +1$ to an upward pointed arrow.}
            \label{0-1-decom}
        \end{figure}

        \item (\textbf{Non-example.}) Let $\psi : \mathbb{S}^2 \to \mathbb{R}^2$ be the stereographic projection. The collection $\mathfrak{S}_1 = \{M_0, M_1, M_2\}$, where $$M_0 = \{\psi^{-1}(0,0)\}, M_1 = \{\psi^{-1}(\{(x, 0) \mid x \in \mathbb{R}, x \neq 0 \})\}, $$ and $M_2 = \mathbb{S}^2 \setminus (M_0 \cup M_1)$ is not an admissible decomposition. $M_0$ is the south pole, $M_1$ is the great circle containing the south pole but missing the north pole $\infty$. We see that $M_0 \cup M_1$ is not closed in $\mathbb{S}^2$ as the north-pole, which is in the closure of $M_0 \cup M_1$, is missing.
        \item (Turning (3) into an example.) However, the collection $\mathfrak{S}_2 = \{M'_0, M'_1, M'_2\}$, where $M'_0 = \{\psi^{-1}(0,0), \infty \}, M'_1 = M_1$ and $M'_2 = \mathbb{S}^2 \setminus (M'_0 \cup M'_1)$, is admissible.
        \item Given a link $L \subset \mathbb{S}^3$, form the collection $M_0 = M_2 = \emptyset$, $M_1 = L$, and  $M_3 = \mathbb{S}^3 \setminus (M_0 \cup M_1 \cup M_2)$. This is an admissible decomposition of $\mathbb{S}^3.$
        \item Let $\Sigma$ be a surface and $\Gamma$ an embedded graph in $\Sigma$ with vertex set $V(\Gamma)$ and edge set $E(\Gamma)$ such that $V(\Gamma) \cap \partial \Sigma = \emptyset$ and each edge of $\Gamma$ is transversal to $\partial \Sigma$, if they meet. The collection $\Sigma_0 = V(\Gamma), \Sigma_1 = E(\Gamma)$ and $\Sigma_2 = \Sigma \setminus (\Sigma_0 \cup \Sigma_1)$ is an admissible decomposition of $\Sigma$.
    \end{enumerate}
\end{exam}

\begin{rema}
  \begin{enumerate}
    \item A consequence of admissibility condition is that $\Bar{M}_k \setminus M_k$ is contained in the union $\bigcup_{i=0}^{k-1}M_i$ of lower dimensional piece. Where $\Bar{M}_k$ is the closure of $M_k$ in $M$.
    \item Let $\mathfrak{M}^k$ denotes the partial union $ \bigcup_{i = 0}^{k}M_i$, then there is a filtration by closed subspaces 
    \begin{equation} \label{eqn:filtration}
        M = \mathfrak{M}^n \supset \mathfrak{M}^{n-1} \supset \dots \supset \mathfrak{M}^0 \supset \mathfrak{M}^{-1}= \emptyset
    \end{equation}
    
    with $\mathfrak{M}^k \setminus \mathfrak{M}^{k-1} = M_k$. Therefore, an admissible decomposition canonically leads to \textit{stratification} of $M$. We will refer to the components of $M_k$ as the $k$-\textit{dimensional strata} of $M$. We refer to the reader to ~\cite{friedman2017singular}, (2.2) for the definition of a filtered space and stratification. 
    \item We stress that ~\cite{carqueville20203} does not require filtration to be given by closed subspaces, but it does not matter to us. Similarly, they also assume frontier and finiteness conditions, which in our case is the consequence of \cref{adm}. However, we do require filtration to be locally cone-like, which we define next.

  \end{enumerate}
\end{rema}

\begin{defi}[\cite{friedman2017singular}, \cite{maxim2019intersection}] \label{def:locally_CS}
With the notation as in \cref{eqn:filtration}, a filtered space $M$ is \textit{locally cone-like}, if for all $i, 0 \leq i \leq n$  and for each $x \in M_i$ there exists a neighborhood $U$ of $x$ in $M_i$, a neighborhood $U_x$ of $x$ in $M$, a compact $n-i-1$ stratified space $L$ with a filtration $$ L = \mathfrak{L}^{n-i-1} \supset \dots \mathfrak{L}^0 \supset \mathfrak{L}^{-1} = \emptyset$$ and a homeomorphism $ h: U \times \mathring{c}L \to U_x $ such that $h(U \times \mathring{c}\mathfrak{L}^k) = \mathfrak{M}^{i+k+1} \cap U_x$. Here, $\mathring{c}L = L \times [0,1)/L\times \{0\}$ is the open-cone on $L$.

\end{defi}

\begin{conv} \label{conv:locally_CS}
\begin{enumerate}
   \item Following \cite{friedman2017singular}, we will refer to $U_x$ in \cref{def:locally_CS} as a \textit{distinguished neighborhood} of $x$.
   \item For the rest of this manuscript we will work with $n=2$, i.e., the top dimensional strata in \cref{adm} are manifolds of dimension $2$.
   \item We assume that each $0$-dimensional strata is incident to at least two $1$-dimensional strata.
\end{enumerate}
\end{conv}


Given a set $D_i$, let $\Bar{D_i}$ be the set of formal inverses of elements in $D_i$, and \linebreak $X_i = D_i \cup \Bar{D_i}$. For example if $D_1 = \{x, y, z\}$ then $\Bar{D}_1 = \{x^{-1}, y^{-1}, z^{-1}\}$, and $X_1 = \{x, y, z, x^{-1}, y^{-1}, z^{-1}\}$. With this convention, define:


\begin{defi} \label{condns}

    A $2$-\textit{defect condition} or $2$-\textit{defect data} is a tuple  $$(D_2, D_1, D_0, \psi_{\{0,1\}}, \psi_{\{1,2\}})$$ consisting of sets $D_2, D_1, D_0$ and maps 
    $$ \psi_{\{1,2\}} : X_1 \to D_2 \times D_2 \hspace{5mm} \text{and} \hspace{5mm} \psi_{\{0,1\}}: X_0 \to \sqcup_{m=0}^{\infty}((X_1)^m /C_m) $$
    where, $(X_1)^m$ mean the $m$-fold Cartesian product of the set $X_1$ and $C_m$ is the group of cyclic permutations that acts on the $m$-tuples in $(X_1)^m$. These maps are subject to the following two \textit{orientation consistency} conditions:
    \begin{enumerate}
        \item If $\psi_{\{1,2\}}(x^{\epsilon}) = (\alpha, \beta)$ , then $\psi_{\{1,2\}}(x^{-\epsilon}) = (\beta, \alpha)$, and
        \item if $\psi_{\{0,1\}}(u^{\epsilon}) = [(x_1^{\epsilon_1}, \dots , x_m^{\epsilon_m})]$ , then $\psi_{0,1}(u^{-\epsilon}) = [(x_m^{-\epsilon_m}, \dots , x_1^{-\epsilon_1})]$
    \end{enumerate}
     Usually, in the literature (cf, ~\cite{davydov2011field}, ~\cite{carqueville20203}) the map $\psi_{1,2}$ is given in terms of two maps $s, t : X_1 \to D_2$ such that for $x^{\epsilon} \in X_1 , \psi_{\{1,2\}}(x^{\epsilon}) = (t(x^{\epsilon}), s(x^{\epsilon}))$. Using maps $s, t$ the tuples in $(X_1)^m$ are $m$ cyclically composable domain walls, i.e. if $(x_1^{\epsilon_1}, \dots , x_m^{\epsilon_m})$ is such a tuple, then $t(x_1) = s(x_n)$ and $t(x_{i+1}) = s(x_i)$.

\end{defi}
Here, we have digressed from ~\cite{davydov2011field} and ~\cite{carqueville20203} in terms of notations, the reason for indexing $\{0,1\}$ or $\{1,2\}$ will become apparent in \cref{surfdef} and the rest will become clear in \cref{sec:groups}.

\begin{rema}
Because of orientation consistency conditions, \cref{condns} (1) and (2), $\psi_{\{1,2\}}$ and $\psi_{\{0,1\}}$ can be uniquely determined by their values on $D_1$ and $D_0$ respectively.    
\end{rema}


\begin{cons} \label{local-junctions}
    We use the defect conditions to produce \textit{defect disks}, as shown in \cref{fig:defect_orient} (i) and (ii) by drawing a string diagram inside the oriented topological disk $\mathbb{D}^2$ (with the standard orientation) as follows:
    \begin{enumerate}
    
        \item If $\psi_{\{1,2\}}(x) = (\alpha, \beta)$, then we represent it as in \cref{fig:defect_orient} (i) with the orientation on the $1$-stratum chosen so that it gives the orientation on $\mathbb{D}^2$ as discussed in \cref{0-1-decom} (iii). The arrow is reversed if $x$ changes to $x^{-1}$ following the orientation consistency condition $\psi_{\{1,2\}}(x^{-1}) = (\beta, \alpha)$. So that $\alpha$ is always on the left of $x$ and $\beta$ is always on the right.
        \item If $\psi_{\{0,1\}}(u^{\epsilon}) = [(x_1^{\epsilon_1}, \dots, x_m^{\epsilon_m})]$, then place a single $0$-dimensional stratum at the center $O$ of $\mathbb{D}^2$ and following the positive induced orientation, which is anti-clockwise here, place the tuple $(x_1^{\epsilon_1}, \dots, x_m^{\epsilon_m})$ on the boundary circle and connect them to the center. For a point $P$ on the boundary circle that is marked $x^{\epsilon}$ the orientation on the $1$-stratum $OP$ is towards $O$ if $\epsilon = +1$ and away from $O$ if $\epsilon = -1$.
        \item There is an operation of involution $\ast$ that changes the orientation of $\mathbb{D}^2$. Note that this will also switch left and right, therefore, to be consistent with \cref{0-1-decom}, (iii), we need to reverse the direction of the arrows. Then, this facilitates gluing by inducing the same orientation at the boundary circle and its $0$-strata. The string diagrams for $\psi_{\{0,1\}}(u^{\epsilon})$ and $\psi_{\{0,1\}}(u^{- \epsilon})$ are connected via an involution to form \cref{fig:defect_orient} (iv) [cf, ~\cite{davydov2011field}, Section 2.3, 2.4.]. See also \cref{sec:word-problem_theory}.
        
    \end{enumerate}
       
\end{cons}

The proof of the following lemma is straightforward and is omitted.

\begin{lemma} \label{lemma:defect_disk-intersection}

With the definition of defect-disk as in \cref{local-junctions}, the intersection of two defect-disks is again a defect disk.
    
\end{lemma}

\begin{figure}
    \centering
    \includegraphics[width=0.98 \linewidth]{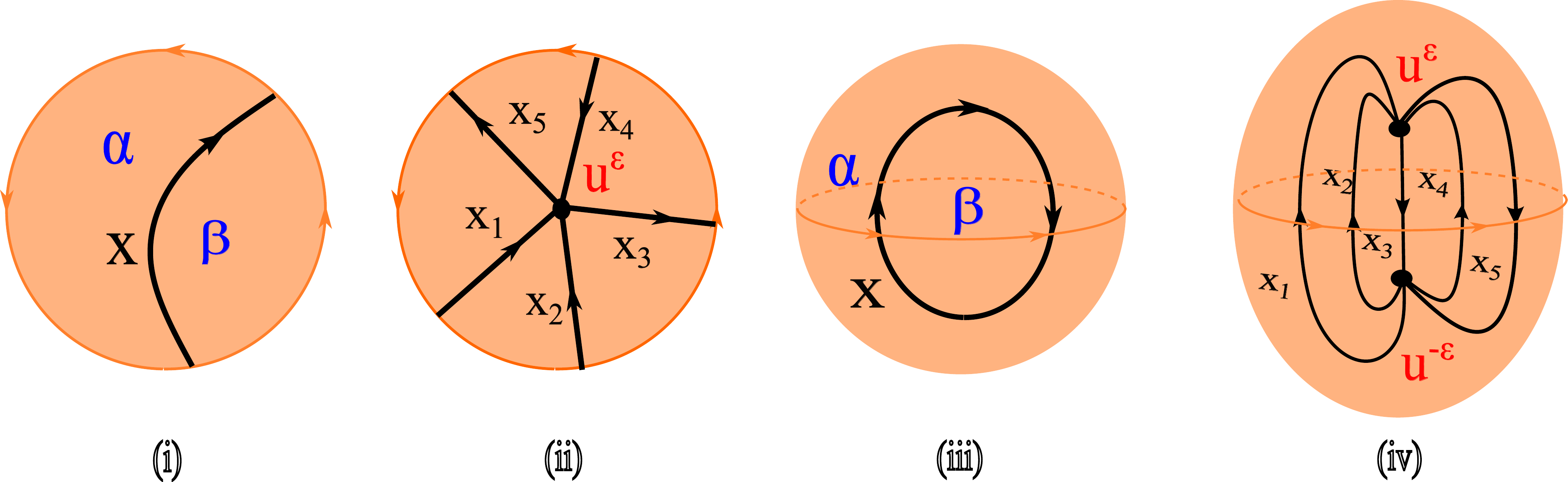}
    \caption{The map $\psi_{1,2}$ can be visualized as (i) where (in this case) it is: $x \mapsto (\alpha, \beta)$. (ii) represnts $\psi_{0,1} : u^{\epsilon} \mapsto [(x_1, x_2, x_3^{-1}, x_4, x_5^{-1})]$. (iii) (respectively (iv)) represents the orientation consistency condition for $\psi_{1,2}$ (respectively $\psi_{0,1}$).}
    \label{fig:defect_orient}
\end{figure}

\subsection{Objects}

Naively, an object in the category $\text{Bord}_2^{\text{def}}(\mathcal{D})$ is a circle with marked points and arcs. Points are marked with the elements of $D_1$, and arcs are marked with the elements of $D_2$. If we denote the circle by $\mathbb{S}^1$ then these points give rise to admissible decomposition of $\mathbb{S}^1$ as in ~\cref{exam-deco}(2). More generally, given sets of defect conditions $\mathcal{E} \coloneqq \{E_0 , E_1, \phi_{0,1}\}$, where $E_0 = D_1, E_1 = D_2$ and $\phi_{\{0,1\}} = \psi_{\{1,2\}}$ we define:

\begin{defi} \label{defcirc}
    A $1$-manifold with defects with defect conditions $\mathcal{E}$ is a tuple $(\mathbb{L}, \mathfrak{L}, d)$ where

\begin{enumerate}
  \item $ \mathbb{L}$ is a closed, connected, oriented $1$-manifold, possibly with boundary.
  \item $\mathfrak{L} $ is an admissible decomposition of $\mathbb{L}$. It consists of a set of points $L_0$ in the interior of  $ \mathbb{L}$ and its complement $L_1 = \mathbb{L} \setminus L_0$, forming a locally cone-like filtration \cref{def:locally_CS}.
  \item For $Y_i = E_i \sqcup \Bar{E}_i$ and $Y = Y_0 \sqcup Y_1$ , $d: \mathbb{L} \to Y $ such that
      \begin{itemize}
          \item $d(L_0) \subset Y_0$,  
          \item $d(L_1) \subset E_1$, and
          \item $d|_{\pi_{0}(L_i)}$ is constant              
    \end{itemize}
  \item The map $d: \mathbb{L} \to Y $ respects $\phi_{\{0,1\}}$. More precisely, if $p_0 \in L_0$ is such that  it is the common boundary of $l_1, r_1 \in L_1$ in a way so that $l_1$ ($r_1$) is oriented into (out of) $p_0$ , then $\phi_{\{0,1\}}(d(p)) = (d(l_1), d(r_1)) $ if the sign, $ \epsilon$, of $d(p)$ is +1. Refer to \cref{0-1-decom}, (ii) for illustration. 
  
\end{enumerate}
 
\end{defi}

\begin{exam}
    We upgrade \cref{exam-deco}, (2)] to a $1$-manifold with defects in ~\cref{1-man-def}, where $E_0 = \{x_1, x_2\}, E_1 = \{\alpha, \beta\}$. 

    \begin{figure}
        \centering
        \includegraphics[scale=0.3]{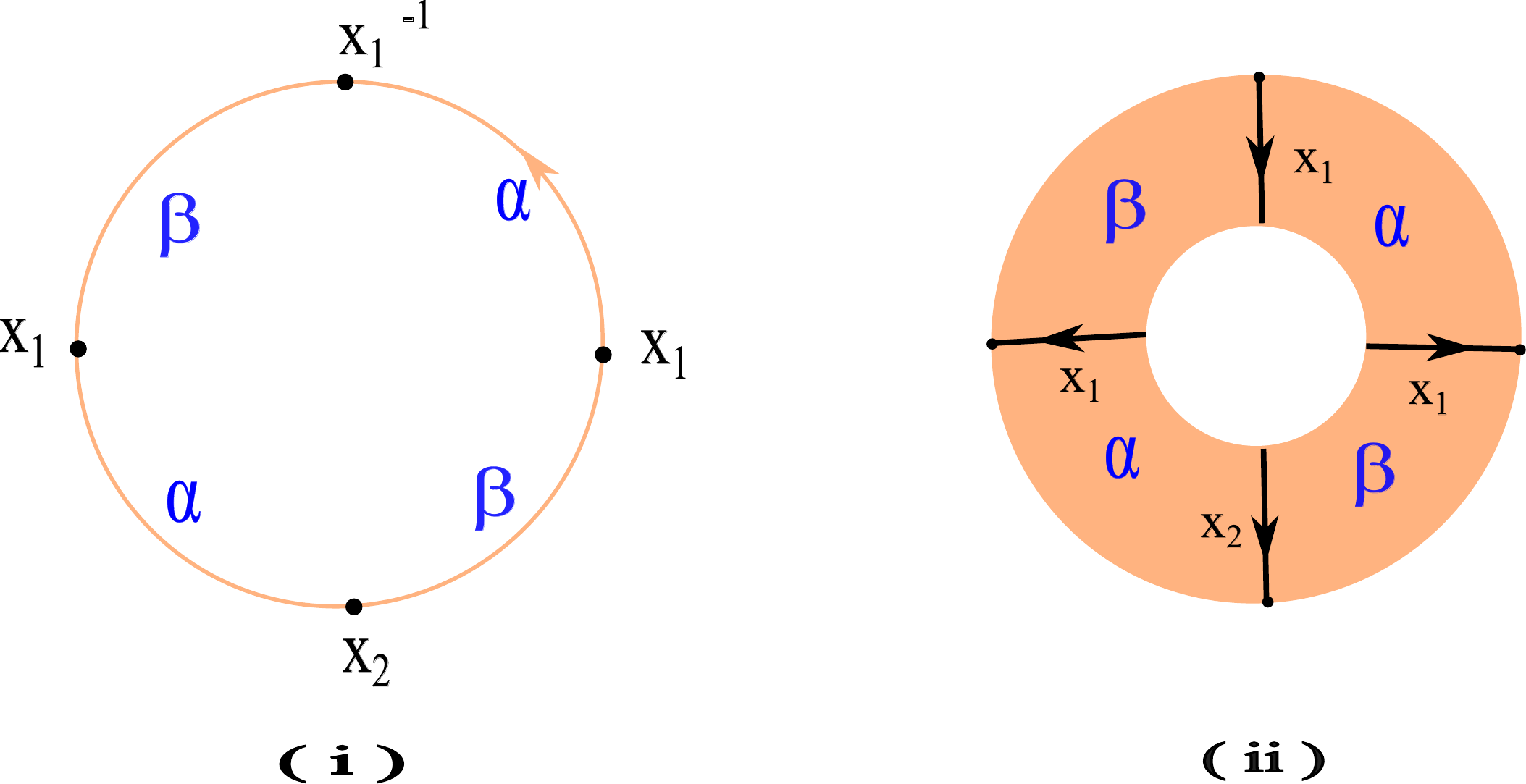}
        \caption{(i) shows a circle with defects. The map $d: \mathbb{S}^1 \to X$ takes $p \mapsto x_1 , q \mapsto x_1^{-1} , r \mapsto x_1 , s \mapsto x_2,(p,q) \mapsto \alpha , \hspace{2mm} (q,r) \mapsto \beta , \hspace{2mm} (r,s) \mapsto \alpha , \hspace{2mm} (s,p) \mapsto \beta$, and $\phi_{\{0,1\}} $ maps $ x_1$ to $ (\alpha, \beta)$ and $x_2$ to $(\beta, \alpha)$. (ii) shows a cylinder on this circle with defects. If we agree to call $p \times [0,1] (\hspace{1mm}\text{respectively} \hspace{1mm} q \times [0,1], r \times [0,1])$ as $p[1]$ (respectively $q[1], r[1]$), then there is no change in the map $d$ as $p$ (respectively $q, r$) lies in the unique component $p[1]$ (respectively $q[1], r[1] $). The same holds for the two-dimensional strata.}
        \label{1-man-def}
    \end{figure}
\end{exam}

\begin{rema}
    We call a closed $1$-manifold with defect, a \textit{circle with defects}. In this case, we further assume that the underlying circle comes with a \textit{distinguished point}, which we denote by $-1$. No $0$-dimensional stratum is allowed to pass through this point under isotopy preserving the defect structure. In other words, all the marked points lie in $ \mathbb{S}^1 \setminus \{-1\}$. See [~\cite{carqueville2016lecture}, Section 2.1] and ~\cite{davydov2011field}, Section 2.3] for details behind this convention.
\end{rema}


\subsection{Morphism} A morphism in the category $\text{Bord}_2^{\text{def}}(\mathcal{D})$ is given by an equivalence class of bordism between two circles with defects. We proceed to define this carefully. 

\begin{defi} \label{surfdef}
    Given a tuple of defect conditions $\mathcal{D} \coloneqq (D_2, D_1, D_0, \psi_{0,1}, \psi_{1,2})$ and associated sets $\{X_i\}$, a \textit{surface with defects} with defect conditions $\mathcal{D}$ consists of the following data.
    \begin{enumerate}
        \item An orientable surface $\Sigma$, possibly with boundary.
        \item An admissible decomposition $\mathfrak{S} \coloneqq \{\Sigma_2, \Sigma_1, \Sigma_0\}$ of $\Sigma$, with locally cone like filtration (cf \cref{def:locally_CS}), such that $\Sigma_0$ lies in the interior of $\Sigma$ and each member of $\Sigma_1$ is transversal to $\partial \Sigma$ if they meet.
        \item For $X = X_0 \sqcup X_1 \sqcup X_2$, a map $d : \Sigma \to X$ such that
          \begin{itemize}
             \item $d(\Sigma_i) \subset X_i$ for all $i \neq 2$,
             \item $d(\Sigma_2) \subset D_2$, and
             \item $d|_{\pi_{0}(\Sigma_i)}$ is constant.
          \end{itemize} 
        \item The map $d : \Sigma \to X$ respects the maps $\psi_{0,1}$ and $\psi_{0,2}$ in the following sense:
             \begin{enumerate}
             
                 \item Let $l \in \Sigma_1$ and $A,B \in \Sigma_2$ be such that $l$ is the common boundary of $A$ and $B$. If $U_x$ is a distinguished neighborhood (\cref{conv:locally_CS}) around a point $x$ in $l$ then the assignments $d(A), d(B)$ and $d(l)$ make $U_x$ into one of the defect disks of the kind \cref{fig:defect_orient} (i). The fact that $d$ is constant on components and \cref{lemma:defect_disk-intersection} makes it well-defined.        
                 \item Let $O \in \Sigma_0$ is such that it is incident by $l_1, \dots, l_m \in \Sigma_1$ and $V_O$ be a distinguished neighborhood (\cref{conv:locally_CS}) around $O$. The assignments $d(l_1), \dots, d(l_m)$ and $d(O)$ are required to make $V_O$ into one of the defect disks of the kind \cref{fig:defect_orient} (iii). Again, this is well-defined due to \cref{lemma:defect_disk-intersection}, and the fact that $d$ is constant on each component.
                 
            \end{enumerate}

    \end{enumerate}

    We will denote a surface with defect by a tuple $(\Sigma, \mathfrak{S}, \mathcal{D}, d) $ or just by $(\Sigma, \mathfrak{S}, d) $ when the defect condition $\mathcal{D}$ is clear from the context. We will refer to $\Sigma$ as the \textit{underlying surface}.
\end{defi}

\begin{exam}
 A defect data controls the map $d$. \cref{exam-surfdef} shows both an example and a non-example of a surface with defects. here, the underlying surface is a disk which is a surface with a boundary.

\end{exam}

\begin{figure}[h]
    \centering
    \includegraphics[scale=0.5]{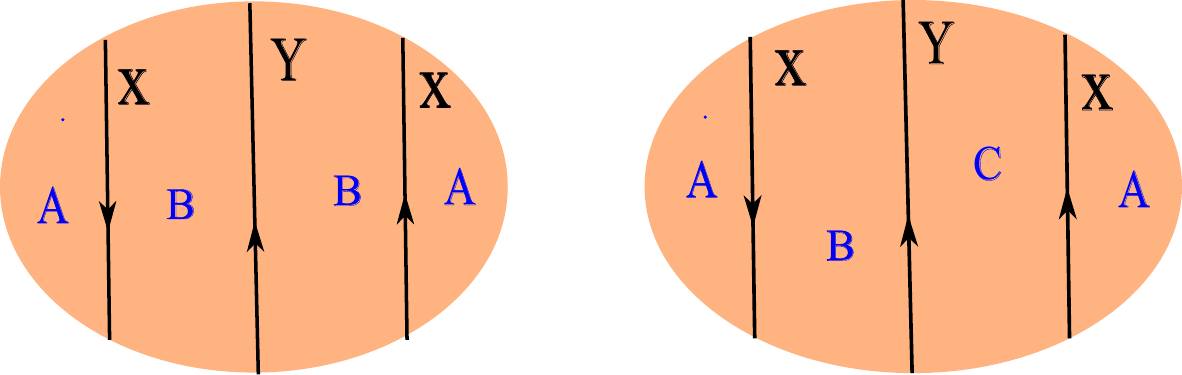}
    \caption{Shows an example of a surface (discs) with defect (left) and a non-example (right). One can check that the map $\psi_{\{1, 2\}}$ is not well defined for $X$ in the case of the diagram on the right. }
    \label{exam-surfdef}
\end{figure}

Next, we define the concept of isomorphism between two surfaces with defects with the same defect conditions $\mathcal{D}$.

\begin{defi} \label{iso}

    Given two surfaces with defects $ (U, \mathfrak{U}, d^U)  \hspace{1mm} \text{and} \hspace{1mm}  (V, \mathfrak{V}, d^V) $, an isomorphism between them is a map $f$ such that 
    \begin{enumerate}
        \item $f: U \to V$ is an orientation preserving diffeomorphism, together with the property that for all $i$, $U_i = f^{-1}(V_i)$. In other words, $f$ preserves orientation and admissible decomposition.
        \item $d^U = d^V \circ f$, or the following diagram commute: 
        $$\begin{tikzcd}
            U \arrow[rr, "f"] \arrow[dr, "d^U"] &  & V \arrow[dl, "d^V"]\\
             & X & 
        \end{tikzcd}$$
    \end{enumerate}
    
\end{defi}

\begin{rema}
    We do not aim to define a general morphism of surfaces with defects as it is not needed for this manuscript. However, it can be modeled on ~\cite{carqueville20203} Definition 2.4.
\end{rema}



Given a circle with defect $(\mathbb{S}^1, \mathfrak{C}, d)$ with defect conditions $\mathcal{E}$ as in \cref{defcirc}, there is a canonical way to put a surface with defect structure on the surface $\mathbb{S}^1 \times I$ for any interval $I$. We denote this surface with defects by $(\mathbb{S}^1 \times I, \mathfrak{C}[1], \mathcal{D}, d)$. The motivation behind such a notation comes from the observation that the admissible decomposition, defect conditions, and the incident map are given by just shifting the index by $+1$, while there is no change in the map $d$. Note that the resulting surface with defects does not have a $0$-dimensional stratum.

\begin{defi} \label{cylinder-collar}
Given a $1$-manifold with defect 
 $\hat{\mathbb{S}} \coloneqq (\mathbb{S}^1, \mathfrak{C}, d)$ with defect conditions $\mathcal{E}$ as in \cref{defcirc},
    \begin{enumerate}
        \item A \textit{open-cylinder} on $\hat{\mathbb{S}}$ is a surface with defect isomorphic to $$(\mathbb{S}^1 \times (0,1), \mathfrak{C}[1], \mathcal{D},  d).$$ 
        \item A \textit{collar} of $\hat{\mathbb{S}}$ is a surface with defects isomorphic to $$(\mathbb{S}^1 \times [0,1), \mathfrak{C}[1], \mathcal{D}, d).$$
    \end{enumerate}
\end{defi}

\cref{1-man-def} (ii) shows a cylinder on (i).

Finally, we define what it means for two circles with defects to be cobordant via a surface with defects. 

\begin{defi} \label{cobordism}
     Let $(U, \mathfrak{U}, d_{U}) $ and $ (V, \mathfrak{V}, d_{V})$ be two circles with defects with defect conditions $\mathcal{E}$ coming from the defect conditions $\mathcal{D}$ as in \cref{defcirc}. A \textit{bordism with defect} is a surface with defects $(\Sigma, \mathfrak{S}, d) $ with defect conditions $\mathcal{D}$  such that 
    \begin{enumerate}
    
        \item (Oriented bordism) $\partial \Sigma = (- U) \sqcup V$ where $-U$ denotes $U$ with the reverse orientation.
        \item (compatible with admissible decomposition) For each $j$, $ U_j \sqcup V_j \subset \partial \Sigma_{j+1}$.
        \item Let $\mathfrak{i}$ (respectively $\mathfrak{j}$) be the inclusion map for the inclusion of $U$ (respectively $V$) into $\partial \Sigma$. The maps $d_U$ and $d_V$ is related to $d$ via $d_U = \mathfrak{i}^{\ast}d$ and $d_V = \mathfrak{j}^{\ast}d$. These maps fit together in the following commutative diagram:
           $$\begin{tikzcd}
                & \Sigma \arrow[dd, "d"] & \\
               U \arrow[dr, "d_U"] \arrow[ur, "\mathfrak{i}"] &  & V \arrow[dl, "d_V"] \arrow[ul, "\mathfrak{j}"] \\
                & X & \\
           \end{tikzcd}$$

    \end{enumerate}

\end{defi}

\begin{conv}

We denote an oriented bordism $\Sigma$ with in-boundary $U$ and out-boundary $V$ by $\begin{tikzcd}[scale cd =0.5] & \Sigma \arrow[dr, "o"] & \\ U \arrow[ur, "i"] &  & V \end{tikzcd}$.

\end{conv}

\begin{exam}
    \cref{Bord_def} shows an example of a bordism from a disjoint union of two circles with defects to a single circle with defects. Here, we have chosen $D_2 = \{\alpha, \beta, \gamma, \delta\}, D_1 = \{x_1, x_2, x_3, x_4, x_5, x_6\}, D_0 = \{u\}; \psi_{1,2}: x_1 \mapsto (\alpha, \beta), x_2 \mapsto (\beta, \alpha), x_3 \mapsto (\alpha, \gamma), x_4 \mapsto (\alpha, \gamma), x_5 \mapsto (\gamma, \alpha), x_6 \mapsto (\delta, \alpha); \psi_{1,2}: u \mapsto [(x_2, x_5^{-1}, x_3^{-1}, x_1)] $.

    \begin{figure}
    \centering
    \includegraphics[scale=0.32]{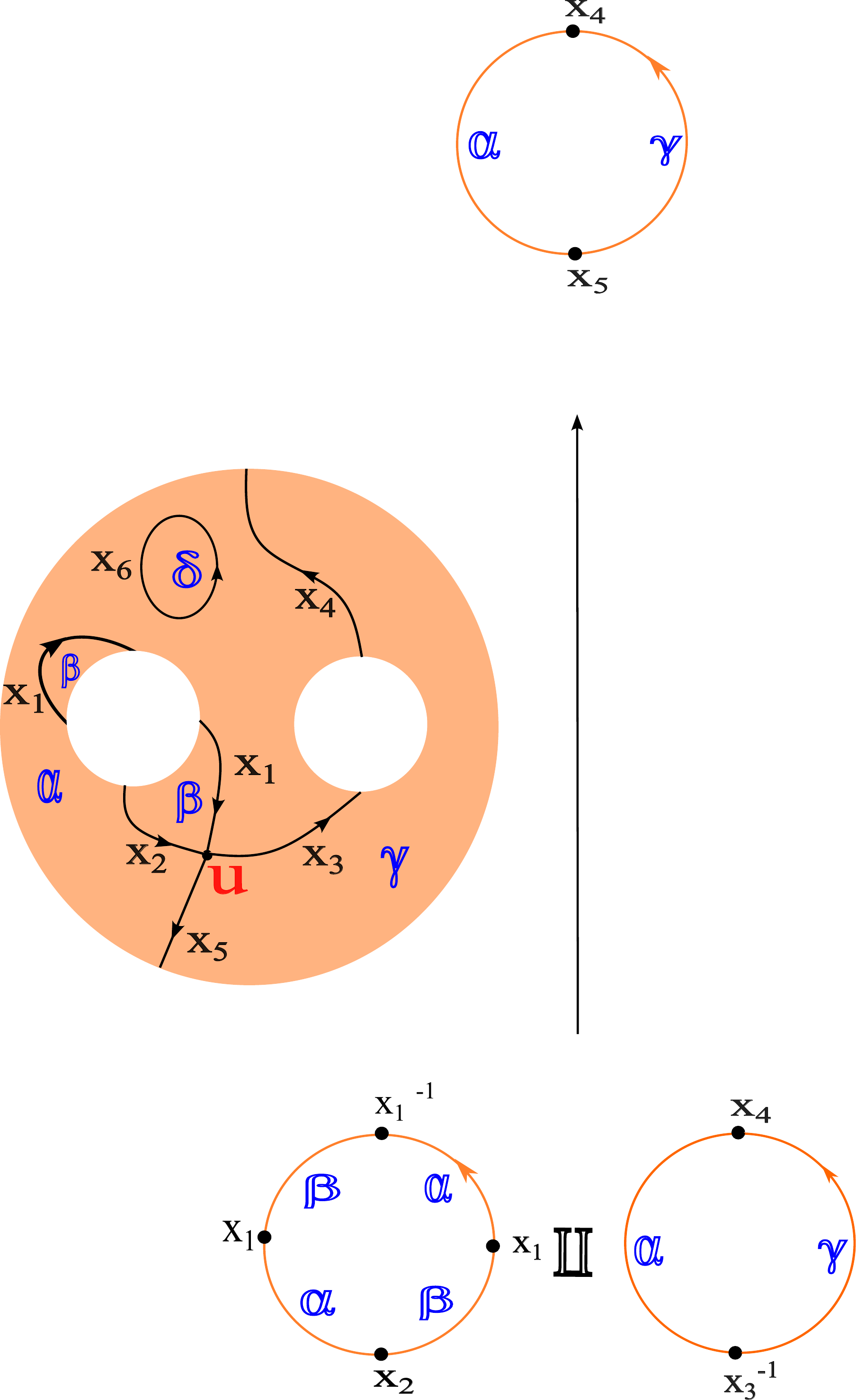}
    \caption{The pair of pants with defects is a bordism between two circles with defects on the left and one on the right. The sign convention on $1$-strata is taken positive if the arrow flows in the direction of time. Thus, an in-boundary gets a positive (respectively negative) sign if the arrow is out of (respectively into) it. The opposite is true for the out-boundary. }
    \label{Bord_def}
\end{figure}

\end{exam}

Finally, we collect all the constituent data of the category $\text{Bord}_2^{\text{def}}(\mathcal{D})$ at one place:

\begin{defi} \label{defn:def-cat}

The category $\textit{Bord}_2^{\textit{def}}(\mathcal{D})$ consists of the following data:

\begin{enumerate}

    \item The \textbf{defect conditions} $\mathcal{D}$, which constitutes sets $D_2, D_1, D_0$, and maps $\psi_{1,2}$ and $\psi_{0,1}$ defined in ~\cref{condns}.
    \item an \textbf{objects} is a disjoint union of circles with defects, defined in ~\ref{defcirc}, together with a germ of collars and distinguished points. As mentioned in [~\cite{freed2019lectures}, 1.2] think of these marked circles as coming with a germ of an embedding into a surface with defects. Further, the germ of collars makes sure that the gluing is well-behaved.
    \item Given two objects $U$ and $V$, a \textbf{morphism} between them is either a permutation of $\hat{\mathbb{S}}^1$ factors in the disjoint union or a surface with defect $\Sigma$ such that $\begin{tikzcd}[scale cd =0.5] & \Sigma \arrow[dr, "o"] & \\ U \arrow[ur, "i"] &  & V \end{tikzcd}$ that respects the condition on distinguished points. See \cref{generic_cross-section} and \cref{distinguished} below. We consider two such bordisms to be equivalent if there is a boundary-preserving isomorphism of the surface with defects. More precisely, a morphism between two objects is given by a defect bordism class between them.
    
\end{enumerate}
This gives $\textit{Bord}_2^{\textit{def}}(\mathcal{D})$ the structure of a symmetric monoidal category (see also \cite{davydov2011field}, Section 2.3.)
\end{defi}

\begin{exam}
    ~\cref{Bord_def} shows an example of a morphism from marked circles on the left to the marked circle on the right. Please note the convention about the sign of $1$-defect conditions.
\end{exam}

Given a surface with defects  $(\Sigma, \mathfrak{S}, \mathcal{D}, d)$, one would expect its cross-sections to be objects in the category $\textit{Bord}_2^{\textit{def}}(\mathcal{D})$, but this requires some care as an arbitrary cross-section may not have a germ of collar around it. The remedy is to consider only those cross sections that admit a collar around it. This is accomplished in two steps. First, we note the existence of a forgetful functor into the category $\textit{Bord}_2$, the symmetric monoidal category of ordinary $2$-dimensional oriented bordism. Next, we use Morse theory to define a \textit{generic cross-section}.

\begin{defi}
    For any set of defect conditions $\mathcal{D}$, there exists a forgetful functor $\mathfrak{D}: \textit{Bord}_2^{\textit{def}}(\mathcal{D}) \to \textit{Bord}_2$ defined by its actions on objects and morphisms as follows:
    \begin{enumerate}
        \item(On objects) $\mathfrak{D}((O, \mathfrak{O}, d_{O})) = O $, 
        \item(On morphism) $\mathfrak{D}( (\Sigma, \mathfrak{S}, \mathcal{D} , d)) = \Sigma$
    \end{enumerate}
That is, the functor $\mathfrak{D}$ maps a circle (surface) with defects to the underlying circle (surface) by forgetting all stratification and defects.
\end{defi}


Next, for a surface $\Sigma$ (or an underlying surface of a surface with defects) choose a height function $f: \Sigma \to \mathbb{R}$. Since $\Sigma$ is compact, we can assume that $f(\Sigma) \subset [0,1]$.

\begin{defi} \label{generic_cross-section}
    Let  $(\Sigma, \mathfrak{S}, \mathcal{D} , d)$ be a surface with defects, and $f: \Sigma \to [0,1]$ be a height function. For $t \in [0,1]$, we say that $f^{-1}(\{t\})$ is a \textit{generic cross-section} of the surface with defects $(\Sigma, \mathfrak{S}, \mathcal{D} , d)$ if there exists $\epsilon > 0$ such that $f^{-1}((t - \epsilon, t + \epsilon))$ is isomorphic to an open-cylinder or a collar of $f^{-1}(\{t\})$ (\cref{cylinder-collar}).
\end{defi}

We refer to [~\cite{audin2014morse}, Section 1.2] for the existence of the Morse function on a given surface $\Sigma$. The existence of a height function follows from the existence of Morse functions.

\begin{rema} \label{distinguished}
    To respect the condition of the distinguished point, we demand that a generic cross-section is a circle with defects with a distinguished point, and the totality of all such distinguished points is contained in a component of $\Sigma_2$. 
\end{rema}

We end this section by mentioning that we have kept ourselves limited to the two-dimensional defect case as it is best suited for our objective, but for $n = 3$, it is done in ~\cite{carqueville20203}, and in ~\cite{carqueville2019orbifolds} for general $n$. Another generalized picture is presented in `\cite{lurie2008classification}, Section-4.3. The other approach, as suggested by our definition of a surface with defects, comes from the introduction of \textit{constructible sheaf}. This latter approach has been developed in ~\cite{freed2022topological}, Section-2.4 and 2.5 in relation with topological symmetries of QFT.


\section{Topological field theories with defect} \label{sec:TFT}

Let $\mathbb{K}$ be a field and $\text{Vect}_{\mathbb{K}}$ denotes the symmetric monoidal category of $\mathbb{K}$ vector spaces. A \textit{topological field theory with defect} (or a TFT \textit{with defects}  is a symmetric monoidal functor $$ T: \text{Bord}_2^{\text{def}}(\mathcal{D}) \to \text{Vect}_{\mathbb{K}} . $$ We are interested in the category of  $\mathbb{K}$ vector spaces with a trace pairing, which restricts the target category of $T$ to $ \text{Vect}_F(\mathbb{K}) $, i.e., the category of finite dimensional  $\mathbb{K}$ vector spaces. One example of such a functor comes from lattice TFT constructions. We do not give this construction in detail but highlight only the essential ingredients and steps. A detail of this construction can be found in [~\cite{davydov2011field}, (3)], which is closest to the spirit of this manuscript. Other references are: [~\cite{frohlich2007duality} , ~\cite{fuchs2002tft}], together with an earlier work [~\cite{turaev1999homotopy}].

\subsection{Category of bordism with PLCW decomposition}

The most essential ingredient for lattice TFTs is the category $\text{Bord}_2^{\text{def, cw}}(\mathcal{D})$. It has the same objects and morphisms as the category $\text{Bord}_2^{\text{def}}(\mathcal{D})$ but they come equipped with an extra structure, namely a PLCW decomposition. We only collect the key feature of PLCW decomposition, and refer to ~\cite{kirillov2012piecewise} for details. The main feature of PLCW decomposition is that although it is more general than triangulation, it is less general than a CW decomposition. 

\begin{defi} [~\cite{kirillov2012piecewise}, Definition 3.1, 3.3]

Let $B^k$ be the $k$-dimensional ball, $\mathring{B}^k$ be its interior, and $S^{n-1} = \partial B^n$ be its boundary. A \textit{generalized k-cell} is a subset $C \subset \mathbb{R}^n$ together with a map $\phi: B^k \to C$, called a characteristic map, such that $\phi \mid_{\mathring{B}^k}$ is injective. 
    
\end{defi}

For a generalized cell $C$, denote $C = \mathring{C} \sqcup \dot{C}$ where $\mathring{C} = \phi(\mathring{B}^k)$ and $\dot{B}^k = \phi(\partial B^k)$.

\begin{defi} [~\cite{kirillov2012piecewise}, Definition 4.1] A generalized cell complex is a finite collection $K$ of generalized cells in $\mathbb{R}^n$ such that 

\begin{enumerate}
    \item for any distinct cells $A, B \in K, \mathring{A} \cap \mathring{B} = \emptyset$; and
    \item for any cell $C \in K, \dot{C}$ is a union of cells.
\end{enumerate}

The \textit{support} $|K|$ of such a complex $K$ is defined by $|K| = \bigcup_{C \in K}C$. It is a compact subspace of $\mathbb{R}^n$. Conversely, a compact subspace $X$ of $\mathbb{R}^n$ is said to admit a \textit{generalized cell decomposition} if $X = |K|$ for some generalized cell complex $K$.
    
\end{defi}

\begin{defi} [~\cite{kirillov2012piecewise}, Definition 4.3, 5.1]

A PLCW complex $K$ is a generalized cell complex, which is defined inductively in terms of skeletons
$$ K^0 \subset K^1 \subset K^2 \subset \dots , $$ where

\begin{itemize}
    \item \textbf{Base case:} $K^0$ is discrete.
    \item \textbf{Induction step:} Assuming $K^{n-1}$ is a PLCW complex, $K^n$ is built out of $K^{n-1}$ by attaching generalized $n$-cells. Explicitly, for any generalized $n$-cell $A \in K, A = \phi(B^n)$, there exists a generalized cell decomposition $L$ of $\partial B^n$ such that $\phi \mid_{\mathring{B}^n}$ is a cellular map of generalized complexes that is also injective when restricted to the interior of each cell in $L.$
\end{itemize}
    
\end{defi}

~\cref{useful} below, about the properties of PLCW complexes, is very useful for working purposes:

\begin{prop} [~\cite{kirillov2012piecewise}, Section 5] \label{useful}

 The following key properties of a PLCW complex $K$ is immediate:

\end{prop} 

\begin{enumerate}
    \item $ |K| = \bigsqcup_{C \in K} \mathring{C} $.
    \item If $A, B \in K$ are two cells, then $A \cap B$ is a union of cells of $K$.
    \item For any $k$-cell $C \in K$, $\dot{C}$ is a union of $(k-1)$ cells of $K$.
    \item Every PLCW complex is a CW complex.
\end{enumerate}

Given a compact $n$-dimensional manifold $M$, possibly with a non-empty boundary, we say that $M$ admits a PLCW decomposition into $k$-cells for $k = 0, \dots, n$ if $M = |K|$ for some PLCW complex $K$.
 
\begin{exam}
    A cell-decomposition of $\mathbb{S}^2$ into a single $0$-cell and a $2$-cell is a CW-decomposition but not a PLCW decomposition since ~\cref{useful} condition $(3)$ fails.
\end{exam}

One important characteristic of PLCW decomposition is the analog of Alexander's theorem, which states that any two PLCW decompositions of a compact finite-dimensional manifold are related by a sequence of elementary moves [~\cite{kirillov2012piecewise} Theorem 8.1.] This property is very important for us as it allows one to prove that the definition of lattice TFT does not depend on the choice of a cell decomposition [~\cite{davydov2011field}, Section 3.6.]

Next, we discuss the category of bordism equipped with a PLCW decomposition:

\begin{conv}
    We follow ~\cite{davydov2011field} for notations and conventions. In particular, by a \textit{cell-decomposition} of a manifold $M$, we mean a PLCW decomposition of $M$, and by a cell, we mean a generalized cell.
\end{conv}

\begin{conv} \label{conv:cells}
    For a space $M$ with a PLCW decomposition, denote the collection of cells by $C(M)$, and by $C_k(M)$ the collection of $k$-cells.
\end{conv}

\begin{figure}
     \centering
     \includegraphics[width=0.98\linewidth]{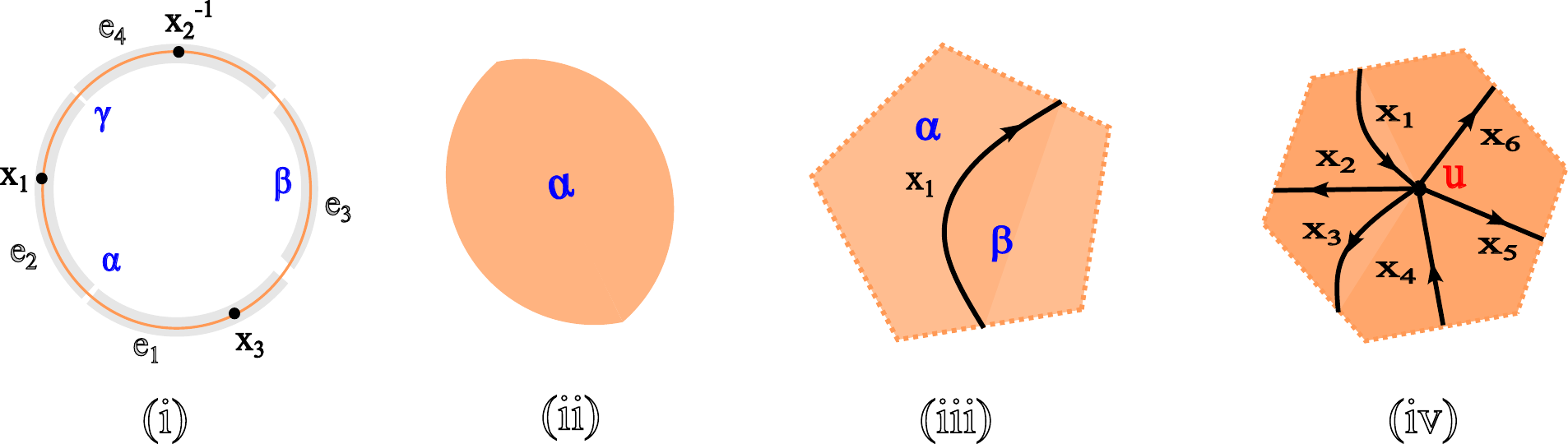}
     \caption{Diagram (i) shows an object in the category $\textit{Bord}_2^{\text{def, cw}}(\mathcal{D})$. Note that the $1$-cell $e_3$ does not contain any defect. Diagram (ii) is a generalized cell with no defect. We will refer to diagrams (ii), (iii), and (iv) as \textit{basic-gons} of type (i), (ii) and (iii) respectively. To read the defect conditions, one must interpret these basic-gons as defect disks and follow \cref{local-junctions}. For example, from diagram (iii) we read $\psi_{1,2}(x) = (\alpha, \beta)$ and from diagram (iv) we read $\psi_{0,1}(u) = [(x_1, x_2^{-1}, x_3^{-1}, x_4, x_5^{-1}, x_6^{-1})]$. Where we have oriented the surface in an anti-clockwise manner. One can again use \cref{local-junctions} (ii) to produce the basic-gon for $\psi_{\{0,1\}}(u^{-1})$ from the data of  $\psi_{\{0,1\}}(u)$. }
     \label{basic_gons}
 \end{figure}

\begin{defi} \label{PLCW}

The category $\textit{Bord}_2^{\text{def, cw}}(\mathcal{D})$ consists of the following data:

\begin{enumerate}

    \item The set of \textbf{defect conditions} $\mathcal{D}$ as defined in \cref{condns}.
    \item An \textbf{object} is given by a disjoint union of circles with defects $U$ together with a cell-decomposition $C(U)$, such that 
      \begin{itemize}
        \item each point of the set $U_0$ lies in a $1$-cell, and 
        \item each $1$-cell contains at most one such point (\cref{basic_gons}, (i)).
      \end{itemize}
    \item A \textbf{morphism} is a surface with defects equipped with a PLCW decomposition $C(\Sigma)$ that is homeomorphic to one of the configurations in \cref{basic_gons} (ii), (iii), or (iv)(after ignoring the labels.) More precisely,
        \begin{itemize}
            \item each $1$-dimensional stratum $\Sigma_1$ only intersects $1$-cells and $2$-cells, but not $0$-cells. Moreover, each $1$-cell intersects only one $1$-stratum of $\Sigma_1$ .
            \item Each $0$-dimensional stratum lies inside a $2$-cell, and each $2$-cell contains at most one such strata. It may only contain a star-shaped configuration of $1$-strata such that each edge of this cell is traversed by exactly one $1$-stratum.
            \item If a $2$-cell contains no $0$-stratum but only $1$-strata then it must be homeomorphic to the configuration shown in \cref{basic_gons}, (iii).
        \end{itemize}
    
\end{enumerate}

\end{defi}

\begin{conv}
    \cref{basic_gons} (ii), (iii), and (iv) are called \textit{basic-gons} (of Types I, II, and III respectively.) Surfaces with defects are assembled from them.
\end{conv}

\begin{rema}
    It is good to think of the basic-gons as \textit{cups} and \textit{caps}. One can check that, under orientation consistency conditions of \cref{condns}, caps transform to cups and vice-versa. See also \cref{sec:word-problem_theory}; in particular \cref{fig:cup-cap_detail}.
\end{rema}

There is a forgetful functor $F: \textit{Bord}_2^{\text{def, cw}}(\mathcal{D}) \to \textit{Bord}_2^{\text{def}}(\mathcal{D})$, which is full and surjective. The lattice TFT construction uses the PLCW decomposition to construct a symmetric monoidal functor $$T^{\text{CW}}:\textit{Bord}_2^{\text{def, cw}}(\mathcal{D}) \to \text{Vect}_F(\mathbb{K})$$ and then $F$ is used to show that $T^{\text{CW}}$ is independent of the cell-decomposition by showing the existence of a unique symmetric monoidal functor $T$ that makes the following diagram commute:

\begin{equation} \label{plcwcd}
 \begin{tikzcd}
    \textit{Bord}_2^{\text{def, cw}}(\mathcal{D}) \arrow[rr, "T^{\text{CW}}"] \arrow[dr, "F"] &  & \text{Vect}_F(\mathbb{K})\\
     & \textit{Bord}_2^{\text{def}}(\mathcal{D}) \arrow[ur, "\exists! \hspace{1mm} T"] & 
\end{tikzcd}
\end{equation}

We do not prove this here but refer to [~\cite{davydov2011field}], Section 3.6.

\subsection{Lattice TFT with defects}

In short, a lattice TFT assigns a Frobenius algebra $A_a$ to $2$-dimensional stratum labeled with defect $a$, a $(A_a - A_b)$-bimodule $X_x$ to the $1$-stratum labelled with $x$ such that $t(x) = a$ and $s(x) = b$, and a bimodule intertwiner to $0$-dimensioanl stratum. We refer to ~\cite{davydov2011field} Section 3.3 for an overall algebraic preliminaries, with supplements ~\cite{kockfrobenius} for Frobenius algebra, and ~\cite{mac2013categories} for bimodules. The following convention will be very handy for our purpose:

\begin{conv} \label{conv:tft}
    Let $A$ and $B$ be a unital, associative algebra over $\mathbb{K}$, and $X$ be a $\mathbb{K}$ vector space. We write $X$ for an $A-B$-bimodule $X$, and $X^{-1}$ for the $B-A$-bimodule $X^{\ast}$, where $X^{\ast}$ the dual of $X$. This way, we can denote a bimodule by $X^{\epsilon}$ where $\epsilon = \pm 1$.

    Later, we will choose $\mathbb{K} = \mathbb{C}$, but the theory works for any field $\mathbb{K}.$
\end{conv}

Recall that, if $A$ is an associative, unital algebra over $\mathbb{K}$, then for a right $A$-module $X$ and for a left $A$-module $Y$ , the tensor product $X \otimes_{A}Y$ is defined as the following cokernel:

\begin{equation} \label{eqn:bimod-1}
     X \otimes A \otimes Y \xrightarrow{l-r}  X \otimes Y \xrightarrow{p}  X \otimes_{A} Y ,
\end{equation}

where $l$ is the left-multiplication map given on pure tensors by $l(x \otimes a \otimes y) = x \otimes (ay)$ and extended by linearity. Similarly, $r$ is the right-multiplication map given by $r(x \otimes a \otimes y) = (xa) \otimes y$. Here, $x \in X, y \in Y, a \in A$ and $\otimes$ denote the tensor product of $\mathbb{K}$-vector space - $\otimes_{\mathbb{K}}$. When $X$ is an $A-A$-bimodule, the \textit{cyclic tensor product} is defined:

\begin{equation} \label{eqn:bimod-2}
     A \otimes X \xrightarrow{l-r}  X \xrightarrow{p} \circlearrowleft X ,
\end{equation}

where $l(a \otimes x) = ax$ and $r(a \otimes x = xa)$.

\begin{defi} \label{def:trivial}
A lattice TFT $T_0^{\textit{cw}}: \textit{Bord}_2^{\textit{def, cw}}(\mathcal{D}) \to \text{Vect}_F(\mathbb{K})$ is a \textit{trivial surrounding theory} if $D_2 $ is a singleton $ \{\ast\}$ and $T_0^{\textit{cw}}$ assigns $\mathbb{K}$ to $\ast$.
    
\end{defi}

Such a theory is characterized by the fact that the non-trivial part of the theory lies solely on the $1$-dimensional strata. It assigns a $\mathbb{K}$ vector space $X_x$ for each $x \in D_1$, which is naturally a $\mathbb{K}-\mathbb{K}$-bimodule. Note that it is enough to consider the assignment for $D_1$ as it can be extended on entire $X_1$ using the orientation consistency and \cref{conv:tft} via the rule: $T_0^{\textit{cw}}(x^{-1}) = T_0^{\textit{cw}}(x)^{-1}$. In other words if $T_0^{\textit{cw}}$ assigns $x$ a $\mathbb{K}$-vector space $X_x$ then it assigns $x^{-1}$ its dual - $X_x^{\ast}$.

The following lemma is very important as all the calculations we are going to do is based on it:

\begin{lemma} \label{prop:tensor}
    For a $\mathbb{K}$ vector space $X, Y$, the following identities hold:
        \begin{enumerate}
            \item $ X \otimes_{\mathbb{K}} Y \cong X \otimes Y$
            \item $\circlearrowleft_{\mathbb{K}} X \cong X$
        \end{enumerate}
\end{lemma}

Where the tensor product $\otimes_{\mathbb{K}}$ on the left of (1) is the tensor product in the sense of \cref{eqn:bimod-1}, and the tensor product $\otimes$ on the right is the tensor product of $\mathbb{K}$-vector space.

The proof of \cref{prop:tensor} is straightforward, and is omitted.

Below, we summarise the input data for a trivial surrounding theory before giving its lattice TFT construction. The reader should consult [~\cite{davydov2011field}, Section 3.4, Section 3.5] for a more general theory.

\begin{defi} \label{trivial_data}
A trivial surrounding theory assigns:

\begin{enumerate}

    \item The field $\mathbb{K}$ for $\ast \in D_2$, 
    \item a $\mathbb{K}$-vector space $X_x$ for each $x \in D_1$, extended to $X_1$ by the rule $T_0^{\textit{cw}}(x^{-1}) = T_0^{\textit{cw}}(x)^{-1}$.
    \item for $u \in D_0$ such that $\psi_{0,1}(u) = [(x_1^{\epsilon_1}, \dots , x_n^{\epsilon_n})]$ a linear map $\mu_u \in \text{Hom}_{\mathbb{K}}(X_{x_1}^{\epsilon_1}\otimes \dots \otimes X_{x_n}^{\epsilon_n}, \mathbb{K})$ with the property that $\mu_u$ is invariant under the induced action on $X_{x_1}^{\epsilon_1}\otimes \dots \otimes X_{x_n}^{\epsilon_n}$ of the action of the cyclic group $C_n$ on the tuple $(x_1^{\epsilon_1}, \dots , x_n^{\epsilon_n})$. We will denote the set of such maps by $\circlearrowleft_{\text{Inv}}\text{Hom}_{\mathbb{K}}(X_{x_1}^{\epsilon_1}\otimes \dots \otimes X_{x_n}^{\epsilon_n}, \mathbb{K})$.

\end{enumerate}

\end{defi}

\begin{rema}
    \cref{trivial_data}, (3) is a consequence of the definition of general lattice TFT data and \cref{prop:tensor}.
\end{rema}

\begin{exam} \label{Tait_coloring}
    Let $X \coloneqq \mathbb{C} \langle a, b, c \rangle$. The map $\mu : X \otimes X \otimes X \to \mathbb{C}$ defined on the bases by the rule
\[ \mu(x \otimes y \otimes z) =
  \begin{cases}
    1       & \quad \text{if } x, y, z \text{ are all different}\\
    0  & \quad \text{otherwise }
  \end{cases}
, \]  

and extended by linearity satisfies the condition [\cref{trivial_data}, (3)]. In fact, it is invariant under the transposition of factors.

\end{exam}

\begin{rema} \label{Frobenius}
We set $\mathbb{K} = \mathbb{C}$ for the rest of this manuscript, and emphasize that the assignment $\mathbb{C}$ to two-dimensional starta by a trivial surrounding theory should be viewed as a Frobenius algebra. Indeed the non-degenerate pairing $\beta : \mathbb{C} \otimes \mathbb{C} \to \mathbb{C} $ given by $ \beta(a \otimes b) = ab $ makes $\mathbb{C}$ a Frobenius algebra with the \textit{counit} $\mathbf{\epsilon}_{\mathbb{C}}$ as the identity $\mathbf{1}_{\mathbb{C}} : \mathbb{C} \to \mathbb{C}$. The copairing $\gamma: \mathbb{C} \to \mathbb{C} \otimes \mathbb{C} \cong \mathbb{C}$ is also the identity map.    
\end{rema}

\begin{cons} \label{trivial_theory}

Fix the defect condition $\mathcal{D}$ with $D_2 = \{\ast\}$. We proceed to explain the trivial surrounding theory as a  symmetric monoidal functor $$ T_0^{\textit{cw}}: \textit{Bord}_2^{\textit{def, cw}}(\mathcal{D}) \to \text{Vect}_F(\mathbb{C}) $$ defined on objects and morphisms using the PLCW decomposition as follows:

\begin{itemize}
    \item \textbf{On objects.} Let $U$ be an object in $\textit{Bord}_2^{\textit{def,cw}}(\mathcal{D})$ which is a single circle. By the definition of the category $\textit{Bord}_2^{\textit{def,cw}}(\mathcal{D})$, it comes equipped with a cell-decomposition as in \cref{basic_gons}-(i). Let $e \in C(U)$ be such a $1$-cell. We assign to it the vector space:
\begin{equation} \label{on_objects}
  R_e =
  \begin{cases}
    \mathbb{C}       & \quad \text{if} \hspace{2mm} e \hspace{2mm} \text{contains no} \hspace{2mm} 0\text{ - defect}\\
    X_x^{\epsilon}  & \quad \text{if } \hspace{2mm} e \hspace{2mm} \text{contains a} \hspace{2mm} 0\text{ - defect with label} \hspace{2mm} x^{\epsilon}
  \end{cases}.
\end{equation}
Then the action of $T_0^{\textit{cw}}$ on this single object $U$ is given by

\begin{equation}
T_0^{\textit{cw}}(U) = \bigotimes_{e \in C_1(U)}R_e
\end{equation}

Here we are using \cref{conv:cells}. For a general object $O = U_1 \sqcup \dots \sqcup U_n$, we extend the definition of $T_0^{\textit{cw}}$ by monoidal property, namely: $ T_0^{\textit{cw}}(O) = T_0^{\textit{cw}}(U_1) \otimes \dots \otimes T_0^{\textit{cw}}(U_n)$

\begin{figure}[ht]
    \centering
    \includegraphics[width=0.9\linewidth]{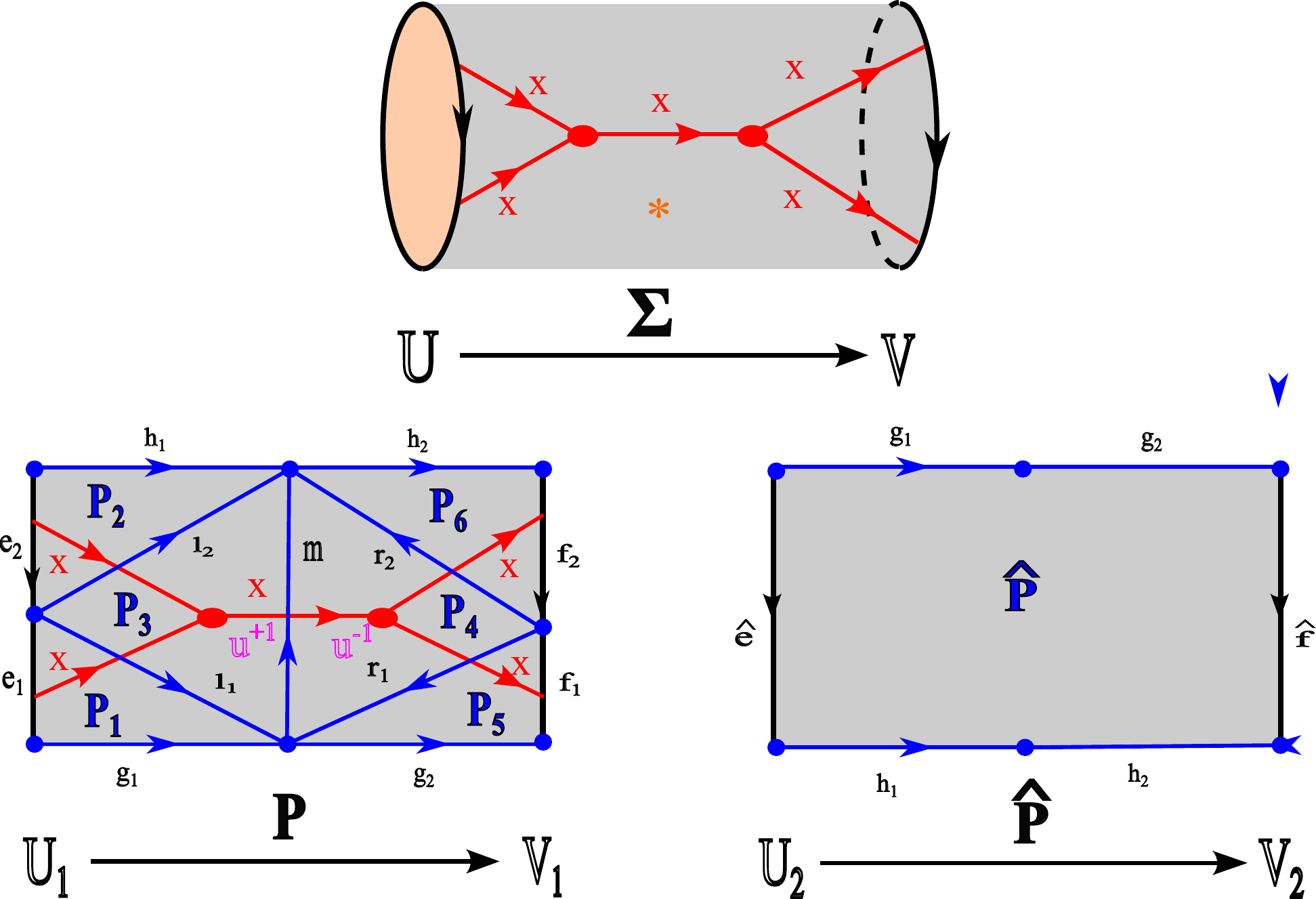}
    \caption{$U$ and $V$ both have two $0$-defects labelled $x$. The two bottom pictures show the PLCW decomposition using basic-gons. It has eight $0$-cells (marked as small blue-circles), $15$ $1$-cells, $ C_1(M) = \{e_1, e_2, \hat{e}, g_1, g_2, h_1, h_2, l_1, l_2, m, r_1, r_2, f_1, f_2, \hat{f}\}$, and seven $2$-cells, $C_2(M) = \{P_1, P_2, P_3, P_4, P_5, P_6, \hat{P}\}$.}
    \label{fig:pre-dumbbell}
\end{figure}

\end{itemize}

\begin{exam} \label{pre-dumbbell:obj}
We refer to \cref{fig:pre-dumbbell} and want to evaluate $T_0^{\textit{cw}}(U)$ and $T_0^{\textit{cw}}(V)$. We make a table below to achieve that:

\begin{center}
\begin{tabular}{ |c|c|c|c|c|c|c| } 
 \hline
 $\mathbf{e :}$ & $e_1$ & $e_2$ & $\hat{e}$ & $f_1$ & $f_2$ & $\hat{f}$ \\ 
 \hline
 $\mathbf{R_e :}$ & $X$ & $X$ & $\mathbb{C}$ & $X$ & $X$ & $\mathbb{C}$ \\ 
 \hline
\end{tabular}
\end{center}

From this, we get 
\begin{equation}
\begin{split}
T_0^{\textit{cw}}(U) & = R_{e_1}\otimes R_{e_2}\otimes R_{\hat{e}} \\
 & = X \otimes X \otimes \mathbb{C} \\
 & \cong X \otimes X
\end{split}
\end{equation}
A similar calculation shows that $T_0^{\textit{cw}}(V) \cong X \otimes X$.

\end{exam}

\begin{itemize}

\item \textbf{On morphism.} Let $\Sigma: U \to V$ be a bordism in $\textit{Bord}_2^{\textit{def,cw}}(\mathcal{D})$, the action of the functor $T_0^{\textit{cw}}$ on $(\Sigma: U \to V)$ is given by:    

 \begin{equation}
    \resizebox{0.9\hsize}{!}{%
    $T_0^{\textit{cw}}(\Sigma): T_0^{\textit{cw}}(U) \xrightarrow{\mathbf{1}_{ T_0^{\textit{cw}}(U)} \otimes \mathscr{P}(\Sigma)} T_0^{\textit{cw}}(U) \otimes Q(\Sigma) \otimes T_0^{\textit{cw}}(V) \xrightarrow{\mathscr{E}(\Sigma) \otimes \mathbf{1}_{ T_0^{\textit{cw}}(V)}} T_0^{\textit{cw}}(V)$%
        }.
   \end{equation}
    
\end{itemize}

We need to describe all these components: the vector space $Q(\Sigma)$, and the maps $\mathscr{P}(\Sigma)$, called the \textit{propagator}, and $\mathscr{E}(\Sigma)$, called the evaluation. By definition, $\Sigma$ is equipped with a PLCW decomposition. We use this fact to define the space $Q(\Sigma)$ and maps $\mathscr{P}(\Sigma)$, and finally use the basic-gons (see \cref{basic_gons}) to define the map $\mathscr{E}(\Sigma)$. 

We begin with vector space $Q(\Sigma)$. Let $\partial_{\text{in}}\Sigma$ be the part parameterized by $U$, the in-boundary of $\Sigma$,  and  $\partial_{\text{out}}\Sigma$ be the part parameterised by $V$, the out-boundary of $\Sigma$. For $P \in C_2(\Sigma)$, we consider triples of the form $(P, e, \mathfrak{O})$ where $e \in C_1(\Sigma)$ is a $1$-cell forming an edge of the polygon $P$, and $\mathfrak{O}$ is an orientation of $e$. We demand that the triple satisfy the condition that the orientation of $P$ comes from the orientation of $\Sigma$, which also orients $e$ as a part of $\partial P$. In other words, the pair $(e, \mathfrak{O})$ is part of $\partial P$ as an oriented edge. We will follow the outward normal first convention. Thus an edge $e$ gets $\mathfrak{O}$ as $+1$ if the orientation on $e$ with the outward normal first convention gives the orientation of $P$ and $-1$ otherwise. Next, to each such triple  $(P, e, \mathfrak{O})$ we assign a vector space: 

\begin{equation}
    Q_{(P, e, \mathfrak{O})} =
  \begin{cases}
    \mathbb{C}      & \quad \text{if } \hspace{2mm} (e, \mathfrak{O}) \hspace{2mm} \text{does not intersect}\hspace{2mm} \Sigma_1.\\
    X_x  & \quad \text{if } \hspace{2mm} (e, \mathfrak{O}) \hspace{2mm} \text{intersects} \hspace{2mm} \Sigma_1 \hspace{2mm} \text{at a defect with label}\hspace{2mm} x\\  & \quad \text{and is oriented into the polygon} \hspace{2mm} P, \hspace{2mm} \text{and}\\
    X_x^{\ast}   & \quad \text{if } \hspace{2mm} (e, \mathfrak{O}) \hspace{2mm} \text{intersects} \hspace{2mm} \Sigma_1 \hspace{2mm} \text{at a defect with label}\hspace{2mm} x \\
     & \quad \text{and is oriented out of the polygon}\hspace{2mm} P
  \end{cases}
\end{equation}

Finally we define the vector space $Q(\Sigma)$ by:
\begin{equation}
    Q(\Sigma) = \bigotimes_{(P, e, \mathfrak{O}), e \notin \partial_{\text{in}}\Sigma} Q_{(P, e, \mathfrak{O})}
\end{equation}

\begin{exam} \label{pre-dumbell-Q(M)}
We again refer to \cref{fig:pre-dumbbell} and want to calculate $Q(\Sigma)$ for the PLCW decomposition given there. We list the data in the following table.
\begin{center}
\begin{tabular}{ |c|c|c|c|c|c| } 
 \hline
 $\mathbf{(P, e, \mathfrak{O}):}$ & $(P_1, e_1, -)$ & $(P_1, g_1, +)$ & $(p_1, l_1, -)$ & $(P_2, e_2, -)$ & $(P_2, l_2, +)$   \\ 
 \hline
 $\mathbf{Q_{(P, e, \mathfrak{O})}:}$ & $X$ & $\mathbb{C}$ & $X^{\ast}$ & $X$ & $X^{\ast}$ \\ 
 \hline
\end{tabular}
\end{center}   
                  \vspace{2mm}
\begin{center}
\begin{tabular}{ |c|c|c|c|c|c| } 
 \hline
 $\mathbf{(P, e, \mathfrak{O}):}$ & $(P_2, h_1, -)$ & $(P_3, l_1, +)$ & $(P_3, l_2, -)$ & $(P_3, m, +)$ & $(P_4, m, -)$\\ 
 \hline
 $\mathbf{Q_{(P, e, \mathfrak{O})}:}$ & $\mathbb{C}$ & $X$ & $X$ & $X^{\ast}$ & $X$ \\ 
 \hline
\end{tabular}
\end{center}  
                   \vspace{2mm}
\begin{center}
\begin{tabular}{ |c|c|c|c|c|c| } 
 \hline
 $\mathbf{(P, e, \mathfrak{O}):}$ & $(P_4, r_1, -)$ & $(P_4, r_2, +)$ & $(p_5, g_2, +)$ & $(P_5, r_1, +)$ & $(P_5, f_1, +)$   \\ 
 \hline
 $\mathbf{Q_{(P, e, \mathfrak{O})}:}$ & $X^{\ast}$ & $X^{\ast}$ & $\mathbb{C}$ & $X$ & $X^{\ast}$  \\ 
 \hline
\end{tabular}
\end{center}  
                  \vspace{2mm}
\begin{center}
\begin{tabular}{ |c|c|c|c|c|c| } 
 \hline
 $\mathbf{(P, e, \mathfrak{O}):}$ & $(P_6, r_2, -)$ & $(P_6, h_2, -)$ & $(p_6, f_2, +)$ & $(\hat{P}, \hat{e}, -)$ & $(\hat{P}, \hat{f}, +)$   \\ 
 \hline
 $\mathbf{Q_{(P, e, \mathfrak{O})}:}$ & $X$ & $\mathbb{C}$ & $X^{\ast}$ & $\mathbb{C}$ & $\mathbb{C}$  \\ 
 \hline
\end{tabular}
\end{center} 
                   \vspace{2mm}
\begin{center}
\begin{tabular}{ |c|c|c|c|c| } 
 \hline
 $\mathbf{(P, e, \mathfrak{O}):}$ & $(\hat{P}, g_1, -)$ & $(\hat{P}, g_2, -)$ & $(\hat{P}, h_1, +)$ & $(\hat{P}, h_2, +)$  \\ 
 \hline
 $\mathbf{Q_{(P, e, \mathfrak{O})}:}$ & $\mathbb{C}$ & $\mathbb{C}$ & $\mathbb{C}$ & $\mathbb{C}$  \\ 
 \hline
\end{tabular}
\end{center} 
               \vspace{3mm}
Finally, we drop the contribution from edges $e_1, e_2$ and $\hat{e}$ (as they belong to $\partial_{\text{in}}\Sigma$) to write

\begin{align*}
    Q(\Sigma) = Q_{(P_1, g_1, +)} \otimes Q_{(P_1, l_1, -)} \otimes Q_{(P_2, l_2, +)} \otimes Q_{(P_2, h_1, -)} \otimes Q_{(P_3, l_1, +)} \\
    \otimes Q_{(P_3, l_2, -)} \otimes Q_{(P_3, m, +)} \otimes Q_{(P_4, m, -)} \otimes Q_{(P_4, r_1, -)} \otimes Q_{(P_4, r_2, +)} \\ \otimes Q_{(P_5, g_2, +)} \otimes Q_{(P_5, r_1, +)} \otimes Q_{(P_5,f_1, +)} \otimes Q_{(P_6, r_2, -)} \\ \otimes Q_{(P_6, h_2, -)} \otimes Q_{(P_6, f_2, +)} \otimes Q_{(\hat{P}, \hat{f}, +)} \otimes Q_{(\hat{P},g_1, -)} \\ \otimes Q_{(\hat{P}, g_2, -)} \otimes Q_{(\hat{P}, h_1, +)} \otimes Q_{(\hat{P}, h_2, +)}.
\end{align*}

This, in turn, gives

\begin{align*}
     Q(\Sigma) = \mathbb{C} \otimes X^{\ast} \otimes X^{\ast} \otimes \mathbb{C} \otimes X \otimes X
     \otimes X^{\ast} \otimes X \otimes X^{\ast} \otimes X^{\ast} \otimes \mathbb{C}
     \otimes X \\
     \otimes X^{\ast} \otimes X \otimes \mathbb{C} \otimes X^{\ast} \otimes \mathbb{C} \otimes \mathbb{C} \otimes \mathbb{C} \otimes \mathbb{C} \otimes \mathbb{C},
\end{align*}

where we have intentionally kept copies of $\mathbb{C}$ for now. We also write the edges for $\mathbb{C}$ and the vector space $X$. We will see in subsequent calculations (for evaluation) that it is important to keep track of the edges contributing to $Q(\Sigma)$. So, we conclude this example by just rewriting $Q(\Sigma)$ where the contributions of relevant edges have been marked correctly:

\begin{equation} \label{Q-labelled}
    \begin{aligned}
        Q(\Sigma) & = \mathbb{C}_{g_1} \otimes X^{\ast}_{l_1} \otimes X^{\ast}_{l_2} \otimes \mathbb{C}_{h_1} \otimes X_{l_1} \otimes X_{l_2}
     \otimes X^{\ast}_{m} \otimes X_{m} \\ 
     & \otimes X^{\ast}_{r_1} \otimes X^{\ast}_{r_2}  \otimes \mathbb{C}_{g_2} \otimes X_{r_1} \otimes X^{\ast}_{f_1} \otimes X_{r_2} \otimes \mathbb{C}_{h_2} \\ 
     & \otimes X^{\ast}_{f_2} \otimes \mathbb{C}_{\hat{f}} \otimes \mathbb{C}_{g_1} \otimes \mathbb{C}_{g_2} \otimes \mathbb{C}_{h_1} \otimes \mathbb{C}_{h_2}.
    \end{aligned}
\end{equation}
\end{exam}

Next, we turn to the propagator $\mathscr{P}(\Sigma) : \mathbb{C} \to Q(\Sigma) \otimes T_0^{cw}(V)$. We note that in \cref{Q-labelled}, each edge appears twice: one with $X$ and the other time with $X^{\ast}$. (Same holds for $\mathbb{C}$ but it has been identified with its dual.) This is not a coincidence as we are going to see below that the space $Q(\Sigma)$ has been assembled from the propagator $\mathscr{P}(\Sigma)$. Each edge $e \in C_1(\Sigma)$ appears twice since two polygons share it, and the induced orientation from one polygon is opposite to the induced orientation from the other as both polygons are given the same orientation as that of $\Sigma$. The construction of the map $\mathscr{P}(\Sigma)$ is done by defining on each edge $e$ and assembling it at the end. Let us denote the two triples involving the edge $e$ by $(P(e)_1, e, \mathfrak{O}_1)$ and $(P(e)_2, e, \mathfrak{O}_2)$. The notation means that $e$ appears as a common boundary of $P(e)_1, P(e)_2 \in C_2(\Sigma)$ with orientations $\mathfrak{O}_1$ and $\mathfrak{O}_2$ that are opposite to each other. We have two cases to consider.

\begin{itemize}
    \item The case where $e$ is an interior edge, that is, $e \notin  C_1(\Sigma) \cap \partial\Sigma$. In this case we define the linear the map $$ \mathscr{P}_e : \mathbb{C} \to Q_{(P(e)_1, e, \mathfrak{O}_1)} \otimes Q_{(P(e)_2, e, \mathfrak{O}_2)} $$ according to the following two sub-cases:
    \begin{enumerate}
        \item If $e$ does not intersect $\Sigma_1$ then both $ Q_{(P(e)_1, e, \mathfrak{O}_1)}$ and $Q_{(P(e)_2, e, \mathfrak{O}_2)}$ is $\mathbb{C}$. In this case, we take $\mathscr{P}_e = \gamma$ where $\gamma: \mathbb{C} \to \mathbb{C} \otimes \mathbb{C}$ is the copairing of the Frobenius algebra $\mathbb{C}$, which by \cref{Frobenius} is the identity map. Thus for $e \notin C_1(\Sigma) \cap \Sigma_1$, $\mathscr{P}_e : \mathbb{C} \to \mathbb{C}$ is the identity $\mathbf{1}_{\mathbb{C}}$.
        \item If $e$ does intersect $\Sigma_1$ then it does so in a defect labelled $x$. In this case one of the  $ Q_{(P(e)_1, e, \mathfrak{O}_1)}$ and $ Q_{(P(e)_2, e, \mathfrak{O}_2)}$ is $X_{x,e}$ and the other is $X_{x,e}^{\ast}$. If we choose to write $\mathscr{P}_e : \mathbb{C} \to X_{x,e} \otimes X_{x,e}^{\ast}$ then $\mathscr{P}_e$ is given by 
        \begin{equation}
            \mathscr{P}_e(\lambda) = \lambda \sum_{i}v_i \otimes v_i^{\ast}
        \end{equation}
        where $\{v_i\}$ is a basis of $X_{x,e}$ and $\{v_i^{\ast}\}$ be the corresponding dual basis of $X_{x,e}^{\ast}$, that is, $v_i^{\ast}(v_j) = \delta_{ij}$. Here we have denoted the vector space $X_x$ by $X_{x,e}$ to keep track of the edge.
        \end{enumerate}

    \item If $e$ is such that $e \in C_1(\Sigma) \cap \partial_{\text{out}}\Sigma$, then there is exactly one triple $(P, e, \mathfrak{O})$ that contains $e$. Let us call such a triple $(P(e), e, \mathfrak{o})$, and define
    \begin{equation}
        \mathscr{P}_e = Q_{P(e), e, \mathfrak{O}} \otimes R_{e}
    \end{equation}
    Note that if one of $Q_{P(e), e, \mathfrak{O}}$ and $R_e$ gets $X_{x,e}$, the other will get its dual $X_{x,e}$.
\end{itemize}

Altogether, the propagator $\mathscr{P}(\Sigma)$ is defined by
\begin{equation}
    \mathscr{P}(\Sigma) = \bigotimes_{e \in C_1(\Sigma), e \notin \partial_{\text{in}}\Sigma} \mathscr{P}_e .
\end{equation}

\begin{exam} \label{pre-dumbbell-gator}
    We again refer to \cref{fig:pre-dumbbell} and \cref{pre-dumbell-Q(M)} and compute the propagator $\mathscr{P}(\Sigma)$. To do that we list all the individual maps $\mathscr{P}_e$:

    \begin{equation}
        \begin{matrix}
            \mathscr{P}_{g_1}: \mathbb{C} \to \mathbb{C}_{g_1} \otimes \mathbb{C}_{g_1}& \mathscr{P}_{l_1}: \mathbb{C} \to X_{l_1} \otimes X_{l_1}^{\ast} & \mathscr{P}_{l_2}: \mathbb{C} \to X_{l_2}\otimes X_{l_2}^{\ast}\\
            \mathscr{P}_{h_1}: \mathbb{C} \to \mathbb{C}_{h_1} \otimes \mathbb{C}_{h_1} & \mathscr{P}_{m}: \mathbb{C} \to X_m \otimes X_m^{\ast} & \mathscr{P}_{g_2}: \mathbb{C} \to \mathbb{C}_{g_2} \otimes \mathbb{C}_{g_2}\\ 
            \mathscr{P}_{r_1}: \mathbb{C} \to X_{r_1}\otimes X_{r_1}^{\ast} & \mathscr{P}_{r_2}: \mathbb{C} \to X_{r_2}\otimes X_{r_2}^{\ast} & \mathscr{P}_{h_2}: \mathbb{C} \to \mathbb{C}_{h_2} \otimes \mathbb{C}_{h_2} \\
            \mathscr{P}_{f_1}: \mathbb{C} \to X_{f_1}\otimes X_{f_1}^{\ast} & \mathscr{P}_{f_2}: X_{f_2}\otimes X_{f_2}^{\ast} & \mathscr{P}_{\hat{f}}: \mathbb{C} \to \mathbb{C}_{\hat{f}} \otimes \mathbb{C}_{\hat{f}}
        \end{matrix}
    \end{equation}

We see, after arranging \cref{Q-labelled} that $\mathscr{P}(\Sigma): \mathbb{C} \to Q(\Sigma)\otimes T_0^{\textit{cw}}(V)$.

\end{exam}
Finally, we define the evaluation map $\mathscr{E}(\Sigma): T_0^{\textit{cw}}(U) \otimes Q(\Sigma) \to \mathbb{C}$. By adjoining $T_0^{\textit{cw}}(U)$ to $Q(\Sigma)$, we have gathered all the $Q_{(P, e, \mathfrak{o})}$ from the table in \cref{pre-dumbell-Q(M)} that we dropped when writing the expression of $Q(\Sigma)$, that is, $$ T_0^{\textit{cw}}(U) \otimes Q(\Sigma) = \bigotimes_{P \in C_2(M), (e, \mathfrak{O}) \in \partial P} Q_{(P, e, \mathfrak{O})}$$

For each polygon $P \in C_2(M)$, define a $\mathbb{C}$-linear map 
\begin{equation}
    \mathscr{E}_P : \bigotimes_{(e, \mathfrak{O})\in P}Q_{(P, e, \mathfrak{O})} \to \mathbb{C}
\end{equation}

as follows, depending on the kind of defects $P$ contains:
\begin{enumerate}
    \item $P$ does not intersect $\Sigma_0$ or $\Sigma_1$. In this case we get $\mathscr{E}: \otimes^{n}\mathbb{C} \to \mathbb{C}$ which is given by $\mathscr{E}_P(c_1 \otimes \dots \otimes c_n) = \mathbf{\epsilon}_{\mathbb{C}}(c_1 \dots c_m)$ where $\mathbf{\epsilon}_{\mathbb{C}}$ is the counit from \cref{Frobenius}, which is the identity. So, under the identification of the space $\otimes^n\mathbb{C}$ with $\mathbb{C}$, $\mathscr{E}_P: \mathbb{C} \to \mathbb{C}$ is simply the identity map $\mathbf{1}_{\mathbb{C}}$.
    \item $P$ intersects $\Sigma_1$ but not $\Sigma_0$. By the definition of $\textit{Bord}_2^{\textit{def, cw}}(\mathcal{D})$, it must resemble \cref{basic_gons}(iii), that is, there is only one such component of $\Sigma_1$. Let it be $x$. There is one oriented edge where $x$ leaves $P$. If this (oriented) edge is denoted by $(e_1, \mathfrak{O}_1)$, then $Q_{(P, e_1, \mathfrak{O}_1)}$ equals $X_x^{\ast}$. Starting from this edge we traverse the edges of $\partial P$ in an anti-clockwise manner. The linear map $\mathscr{E}_P$ then takes the form $$ \mathscr{E}_P: X_x^{\ast} \otimes (\otimes^{n_1}\mathbb{C}) \otimes X_x \otimes (\otimes^{n_2}\mathbb{C}) \to \mathbb{C}. $$
    Set 
  \begin{equation} \label{E-type-2}
  \begin{aligned}
      \mathscr{E}_P(v_x^{\ast} \otimes c_1 \otimes \dots \otimes c_{n_1}\otimes w_x \otimes c_{n_1+1} \otimes \dots \otimes c_{n_1+n_2} )\\
        = v_x^{\ast}((c_1 \dots c_{n_1})w_x(c_{n_1+1}\dots c_{n_1 + n_2})) .   
    \end{aligned}
    \end{equation}
    By linearity, it becomes $(c_1 \dots c_{n_1})v_x^{\ast}(w_x)(c_{n_1+1}\dots c_{n_1 + n_2})$, which also shows that in case of a trivial surrounding theory, we could have picked up any edge of $P$.

    \item Finally, suppose that $P$ does contains a component of $\Sigma_0$. Then by the definition of $\textit{Bord}_2^{\textit{def, cw}}(\mathcal{D})$, it must look like \cref{basic_gons} (iv). Explicitly, each oriented edge $(e_i, \mathfrak{O}_i) \in \partial P$ intersects $\Sigma_1$. Choose an arbitrary edge $(e_1, \mathfrak{O}_1)$ and order the remaining edges in an anti-clockwise manner. Let $u^{\epsilon} \in X_0$ be the label at the only element of $\Sigma_0 \cap P$, and $\psi_{0,1}(u^{\epsilon}) = [(x_1^{\epsilon_1}, \dots, x_n^{\epsilon_n})]$. If $\epsilon = +1$ then the TFT $T_0^{\textit{cw}}$ assigns $u^{+1}$ an element $\mu_u \in \circlearrowleft_{\text{Inv}}\text{Hom}_{\mathbb{K}}(X_{x_1}^{\epsilon_1}\otimes \dots \otimes X_{x_n}^{\epsilon_n}, \mathbb{K})$. Set $\mathscr{E}_P(v_1 \otimes \dots \otimes v_n) = \mu_u(v_1 \otimes \dots \otimes v_n)$. Note that this is independent of the choice of $(e_1, \mathfrak{O}_1)$ since $\circlearrowleft_{\text{Inv}}\text{Hom}_{\mathbb{K}}(X_{x_1}^{\epsilon_1}\otimes \dots \otimes X_{x_n}^{\epsilon_n}, \mathbb{K})$ is defined that way. See \cref{trivial_data} for detail.

    If $\epsilon = -1$, then repeat the same argument with the class $[(x_n^{-\epsilon_n}, \dots, x_1^{-\epsilon_1})] $.   
        
\end{enumerate}

\begin{exam} \label{exam:pre-dumbbell-eval}
    We continue to evaluate the TFT assigned to \cref{fig:pre-dumbbell}. Moving ahead of \cref{pre-dumbbell-gator}, write $\mathscr{E}_P$ for polygons in $C_2(\Sigma)$. The only polygon of type-(1) is $\hat{P}$, for which

    \begin{equation}
        \begin{matrix}
         \mathscr{E}_{\hat{P}}: & Q_{(\hat{P}, \hat{e}, - )}\otimes  Q_{(\hat{P}, h_1, +)}\otimes  Q_{(\hat{P},h_2, +)}\otimes  Q_{(\hat{P}, g_1,-)}\otimes  Q_{(\hat{P}, g_2, -)}\otimes  Q_{(\hat{P}, \hat{f}, +)} \to \mathbb{C} \\ 
         & \mathbb{C}_{\hat{e}}\otimes  \mathbb{C}_{h_1}\otimes  \mathbb{C}_{h_2}\otimes  \mathbb{C}_{g_1}\otimes  \mathbb{C}_{g_2}\otimes  \mathbb{C}_{\hat{f}} \longrightarrow \mathbb{C} \\
         & \lambda_{\hat{e}}\otimes \lambda_{h_1}\otimes \lambda_{h_2}\otimes \lambda_{g_1}\otimes \lambda_{g_2}\otimes \lambda_{\hat{f}}  \longmapsto \lambda_{\hat{e}}\lambda_{h_1}\lambda_{h_2}\lambda_{g_1} \lambda_{g_2}\lambda_{\hat{f}}
         \end{matrix}.
    \end{equation}
        
     $P_1, P_2, P_5$ and $P_6$ is of type-(2):

    \begin{equation}
        \begin{matrix}
         \mathscr{E}_{P_1}: & Q_{(P_1, l_1, - )}\otimes  Q_{(P_1, e_1, -)}\otimes  Q_{(P_1,g_1, +)} \to \mathbb{C} \\ 
         & X_{l_1}^{\ast}\otimes  X_{e_1}\otimes \mathbb{C}_{g_1}  \longrightarrow \mathbb{C} \\
         & v^{\ast} \otimes w \otimes \lambda_{g_1}  \longmapsto \lambda_{g_1} v^{\ast}(w)
       \end{matrix}.
    \end{equation}            
        
Similarly for $P_2, P_5$ and $P_6$. In this case, all of these maps are the same except for the indexing. Finally, $P_3$ and $P_4$ is of type-(3). Since $\psi_{0,1}(u) = [(x, x, x^{-1})]$, $T_0^{\textit{cw}}$ assigns $u$ some $\mu_1 \in \circlearrowleft_{\text{Inv}}\text{Hom}_{\mathbb{C}}(X_{l_1}\otimes X_{l_2}\otimes X_{m}^{\ast}, \mathbb{C})$, where all of $X_{m}, X_{l_1}$ and $X_{l_2}$ are $X_x$.

    \begin{equation}
        \begin{matrix}
         \mathscr{E}_{P_3}: & Q_{(P_3, l_1, +)}\otimes  Q_{(P_3, l_2, -)}\otimes  Q_{(P_3,m , +)} \to \mathbb{C} \\ 
         & X_{l_1}\otimes  X_{l_2}\otimes X_{m}^{ast}  \longrightarrow \mathbb{C} \\
         & v_1 \otimes v_2 \otimes v_3  \longmapsto \mu_1 (v_1 \otimes v_2 \otimes v_3)
         \end{matrix}.
    \end{equation}  

Next we turn to $P_4$, the TFT $T_0^{\textit{cw}}$ assigns $u^{-1}$ an element $\mu_2 \in \circlearrowleft_{\text{Inv}}\text{Hom}_{\mathbb{C}}(X_{r_1}^{\ast}\otimes X_{m}\otimes X_{r_2}^{\ast}, \mathbb{C})$. Again, all of $X_{m}, X_{r_1}$ and $X_{r_2}$ are $X_x$.
    \begin{equation}
        \begin{matrix}
         \mathscr{E}_{P_4}: & Q_{(P_4, r_1, -)}\otimes  Q_{(P_4, m, -)}\otimes  Q_{(P_4,r_2 ,+)} \to \mathbb{C} \\ 
         & X_{r_1}^{\ast}\otimes  X_{m}\otimes X_{r_2}^{ast}  \longrightarrow \mathbb{C} \\
         & v_1 \otimes v_2 \otimes v_3  \longmapsto \mu_2 (v_1 \otimes v_2 \otimes v_3)
         \end{matrix}
    \end{equation}  

We conclude this example by pointing out that the map $\mu$ from \cref{Tait_coloring} can take the place of both $\mu_1$ and $\mu_2$. This map will be useful in the next sections.

\end{exam}

This finalizes the construction of a lattice TFT with a trivial surrounding theory. All these conditions can be deduced from the lattice TFT construction in ~\cite{davydov2011field} with the special case when the trivial Frobenius algebra, namely $\mathbb{K}$ is assigned to all the two-dimensional strata. Therefore the trivial surrounding theory enjoys all the facilities of its more general counterpart, i.e., lattice TFT. In particular, the trivial surrounding theory is independent of the choice of a PLCW decomposition. Two such cell decompositions are forced to look similar in the vicinity of a defect by \cref{basic_gons}, but they can always be refined, and altered in many ways. We refer to ~\cite{davydov2011field} and ~\cite{kirillov2012piecewise} for details. Alternatively, one could define the map $\mathscr{P}$ and $\mathscr{E}$ by declaring identity on regions with no lower-dimensional defects (that is without resorting to the Frobenius algebra property of $\mathbb{C}$) and using just the property of vector space and then proving the independence on cell-decomposition from scratch, following Section 7 and 8 of ~\cite{kirillov2012piecewise}. However, we will not take this approach here.

\end{cons}

\subsection{Some useful results}

Now, we state and prove results that will simplify the calculation in the case of a trivial surrounding theory. All the proofs in this section rely on two facts. First, that the TFT is independent of a PLCW decomposition. Second, a trivial surrounding theory assigns $\mathbb{C}$ to two-dimensional strata.

\begin{figure}[ht]
    \centering
    \includegraphics[width=0.9\linewidth]{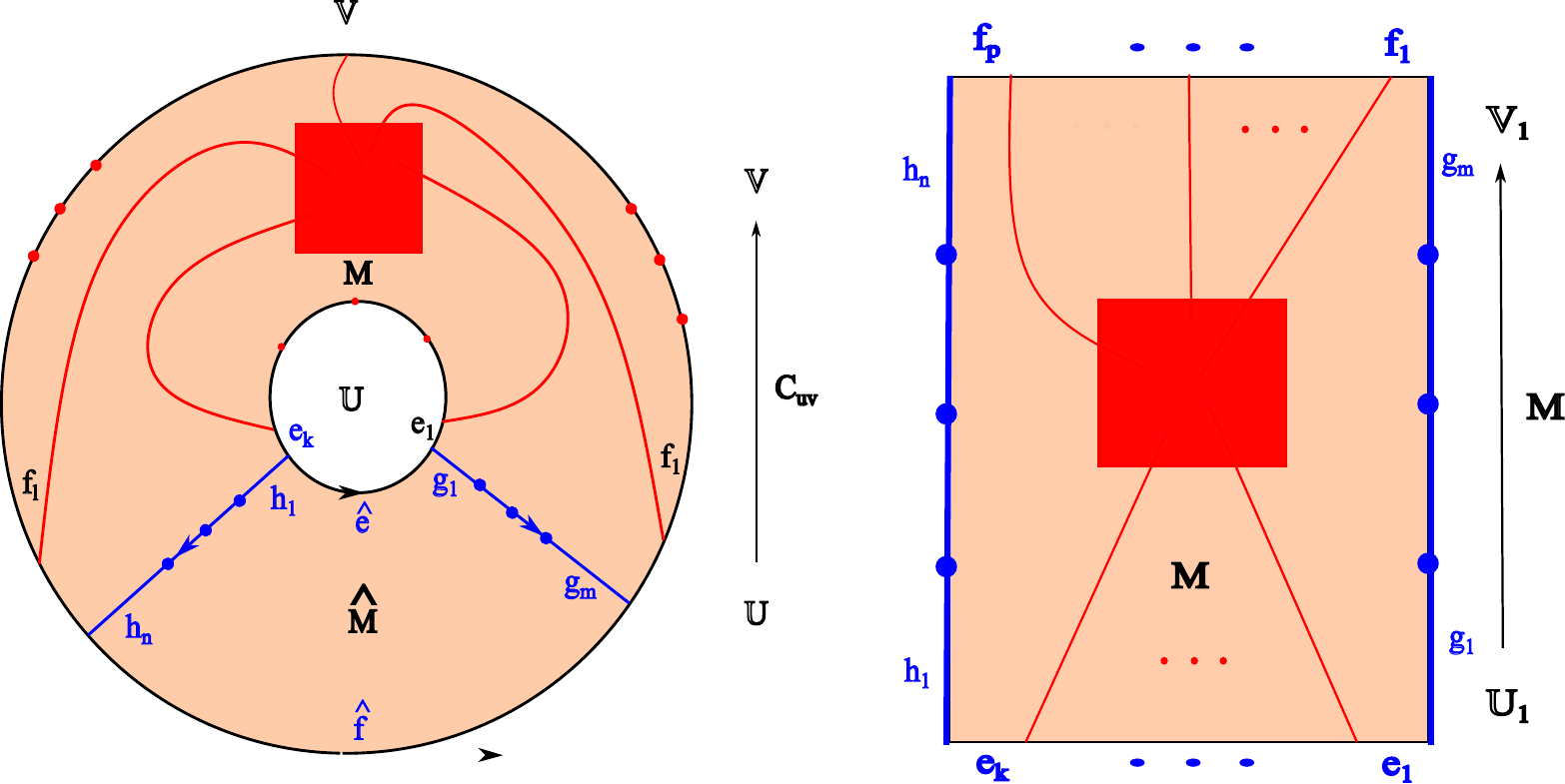}
    \caption{We have the cylinder as a morphism on the left, and on the right is the polygon containing all the defects of $C_{uv}$ $C_1(U_1) = \{e_1, \dots, e_k\}$ which forms $C_1(U)$ together with $\hat{e}$. Similarly, $C_1(V_1) = \{f_1, \dots, f_p\}$ which with $\hat{f}$ forms $C_1(V)$. The red dots \textcolor{red}{$(\dots)$} depict the presence of $1$-dimensional defects in that region, while the blue dots \textcolor{blue}{$(\dots)$} depict the presence of more cells. The cells $\{h_1, \dots, h_n\}$ and $\{g_1, \dots, g_m\}$ are the $1$-cells of $C_1(C_{uv})$ that forms the two common boundaries of $M$ and $\hat{M}$}
    \label{fig:TFT_Prop-1}
\end{figure}

Let $C_{uv}: U \to V$ be the morphism from $U$ to $V$ such that the underlying surface with boundary is a cylinder. Since it has to respect the distinguished point, there exists a PLCW decomposition decomposing $C_{uv}$ into two polygons $M$ and $\hat{M}$, where $\hat{M}$ is such that it contains no $0$ or $1$ dimensional strata. (See \cref{fig:TFT_Prop-1} for a visual demonstration.) Under this decomposition, $U$ (respectively $V$) decomposes as $U_1 \cup U_2$ (respectively $V_1 \cup V_2$) such that $U_1$ (respectively $V_1$) is the restriction of $U$ (respectively $V$) to $M$, and  $U_2$ (respectively $V_2$) is the restriction of $U$ (respectively $V$) to $\hat{M}$. Let $T_0^{\textit{cw}}(U_1) = \otimes_{e \in C_1(U_1)}R_e$ and $T_0^{\textit{cw}}(V_1) = \otimes_{f \in C_1(V_1)}R_f$ and we define $\mathscr{P}(M)$ by keeping contributions from $M$ only, that is, if $e \in C_1(\partial M) \cap C_1(\partial \hat{M})$, we define a truncated propagator $\mathscr{P}'_e: \mathbb{C} \to Q_{(P(e), e, \mathfrak{O})}$ where $P(e) \in C_2(M)$ and define
\begin{equation} \label{P-E-polygons}
    \begin{aligned}
        \mathscr{P}(M) \coloneqq (\bigotimes_{e \notin \hat{M}, e \notin \partial_{\text{in}}C_{uv}, e \in C_1(C_{uv})}\mathscr{P}_e) \otimes_{e \in C_1(\partial \hat{M})} \mathscr{P}'_e & &,& & \mathscr{E}(M) \coloneqq \bigotimes_{P \in C_2(C_{uv}), P \neq \hat{M} } \mathscr{E}_P 
    \end{aligned}
\end{equation}

 and $Q(M)$ be the restriction of $Q(C_{uv})$ to the codomain of $\mathscr{P}(M)$.

The following theorem says that under such conditions, the calculation can be done on the planar polygon $M$.

\begin{prop} \label{Prop-main}
    
    If $T_0^{\textit{cw}}$ is a trivial surrounding theory, then $T_0^{\textit{cw}}(C_{uv}): T_0^{\textit{cw}}(U) \to T_0^{\textit{cw}}(V)$ equals $T_0^{\textit{cw}}(M): T_0^{\textit{cw}}(U_1) \to T_0^{\textit{cw}}(V_1)$, where $T_0^{\textit{cw}}(M)$ is defined by the composite:
    \begin{equation}
        T_0^{\textit{cw}}(U_1) \xrightarrow{\mathbf{1}_{ T_0^{\textit{cw}}(U_1)} \otimes \mathscr{P}(M)} T_0^{\textit{cw}}(U_1) \otimes Q(M) \otimes T_0^{\textit{cw}}(V_1) \xrightarrow{\mathscr{E}(M) \otimes \mathbf{1}_{ T_0^{\textit{cw}}(V_1)}} T_0^{\textit{cw}}(V_1)  
    \end{equation}
    
\end{prop}

\begin{proof}
    Let's denote the single edge covering $U_2$ (respectively $V_2$) by $\hat{e}$ (respectively $\hat{f}$). The map $T^{\textit{cw}}(C_{uv}): T^{\textit{cw}}(U) \to T^{\textit{cw}}(V)$ is given by 

   \begin{equation} \label{proof_TFT_Prop-1}
    \resizebox{0.9\hsize}{!}{%
    $T^{\textit{cw}}(C_{uv}): T^{\textit{cw}}(U) \xrightarrow{\mathbf{1}_{ T^{\textit{cw}}(U)} \otimes \mathscr{P}(C_{uv})} T^{\textit{cw}}(U) \otimes Q(C_{uv}) \otimes T^{\textit{cw}}(V) \xrightarrow{\mathscr{E}(C_{uv}) \otimes \mathbf{1}_{ T^{\textit{cw}}(V)}} T^{\textit{cw}}(V)$%
        }
   \end{equation}
   
The key step is to write $$\mathscr{E}(C_{uv}) = (\bigotimes_{P \in C_2(C_{uv}), P \neq \hat{M}}\mathscr{E}_p) \otimes \mathscr{E}_{\hat{M}}$$ and arrange the middle term in the \cref{proof_TFT_Prop-1} according to it. We get $T_0^{\textit{cw}}(U) = T_0^{\textit{cw}}(U_1) \otimes R_{\hat{e}}$ and $T_0^{\textit{cw}}(V) = T_0^{\textit{cw}}(V_1) \otimes R_{\hat{f}}$. The propagators $\mathscr{P}_{g_i}: \mathbb{C} \to \mathbb{C}_{g_i} \otimes \mathbb{C}_{g_i}$ has the form $1 \mapsto 1_{g_i} \otimes 1_{g_i}$ and similarly for $h_i$ for every $i$. Let $P_{g_i}$ (respectively $P_{h_i}$) be the polygons containing $g_i$ (respectively $h_i$). One of the two factors from $\mathscr{P}_{g_i}(1)$ (respectively $\mathscr{P}_{h_i}(1)$) goes with $\mathscr{E}_{P_{g_i}}$ (respectively $\mathscr{E}_{P_{h_i}}$) and the other with $\mathscr{E}_{\hat{P}}$ reducing \cref{proof_TFT_Prop-1} to:

\begin{equation}
    \begin{aligned}
        T_0^{\textit{cw}}(U_1)\otimes \mathbb{C}_{\hat{e}} \xrightarrow{\mathbf{1}_{ T_0^{\textit{cw}}(U_1)} \otimes\mathscr{P}(M)} T_0^{\textit{cw}}(U_1) \otimes Q(M) \otimes \mathbb{C}_{\hat{e}} \otimes (\otimes_{g_i}\mathbb{C}_{g_i})\otimes (\otimes_{h_i} \mathbb{C}_{h_i}) \\
 \otimes T_0^{\textit{cw}}(V_1) \otimes \mathbb{C}_{\hat{f}} \xrightarrow{\mathscr{E}(M) \otimes \mathbf{1}_{ T_0^{\textit{cw}}(V_1)}} T_0^{\textit{cw}}(V_1) \otimes \mathbb{C}_{\hat{f}}        \end{aligned} .
\end{equation}

This gives the desired result since for a $\mathbb{C}$-vector space $X$, $X \otimes \mathbb{C} \cong \mathbb{C}$.

\end{proof}

\begin{theo} \label{Thm:simplest}
The calculation of a trivial surrounding theory $T_0^{cw}(C_{uv})$ as in \cref{Prop-main} can be done on the polygon of the kind $M$ in \cref{fig:TFT_Prop-1} as shown below:
    \begin{figure}[ht]
        \centering
        \includegraphics[width=0.5\linewidth]{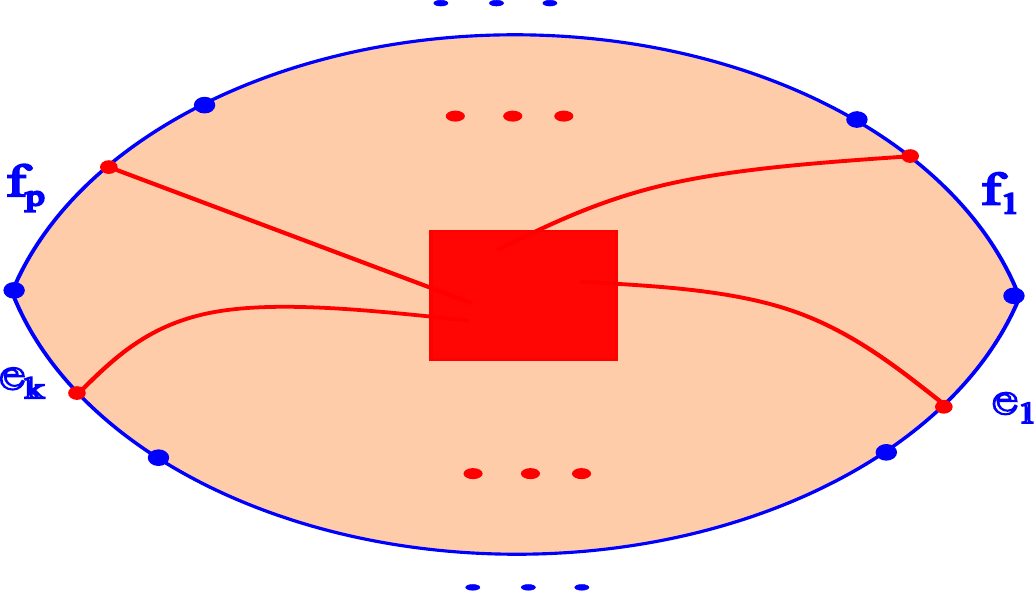}
        \caption{  The red dots \textcolor{red}{$(\dots)$} depict the presence of $1$-dimensional defects in that region, while the blue dots \textcolor{blue}{$(\dots)$} depict the presence of more cells.  }
        \label{fig:simplest}
    \end{figure}
     
\end{theo}

\begin{proof}
    This is a simple application of the fact that the TFT is trivial surrounding, i.e. it assigns $\mathbb{C}$ to the $2$-dimensional strata. First, note that \cref{fig:simplest} is obtained from $M$ in \cref{fig:TFT_Prop-1} by identifying all the zero cells on the non-object sides $(g_1, \dots, g_m) $ and $( h_1, \dots, h_n)$ to (two distinct) single $0$-cells. Then, since $\mathscr{P}$ assigns $\mathbb{C}$ to these edges and both $v^{\ast}$ and $\mu_v$ are $\mathbb{C}$-linear, one can take the contribution from these edges out, for instance, \cref{E-type-2} can be written as $(c_1 \dots c_{n_1}) V_x^{\ast}(w_x)(c_{n_1+1}\dots c_{n_1 + n_2})$, which is same as if the polygon had only two sides, one where the defect enters and the other where the defect leaves, since $c_i$ are scalars. More precisely, if $g_1, \dots, g_m$ are edges ($1$-cells) of one of the basic-gons with a $0$-defect as in \cref{fig:pre-dumbbell}, then each cell containing $g_i$ and $g_{i+1}$ is of the kind discussed in \cref{E-type-2} and thus can be identified, as their contributions are scalars that can be pulled out when writing the expression for $\mathscr{E}$. Repeating this process identifies each $g_i$ and the same holds for $h_i$.
\end{proof}


To state the next proposition we need to use the fact that the category of $2$-defect TFT is equivalent to a Pivotal 2-category. We do not make this construction explicit here and refer to ~\cite{carqueville2016lecture} (2.2) and (2.3). In what follows, only the key features of this construction is highlighted. 

For a defect TFT $T: \textit{Bord}_2^{\textit{def, cw}}(\mathcal{D}) \to \text{Vect}_{F}(\mathbb{K})$, the data of the $2$-category $\mathcal{B}_T$ consists of:
\begin{itemize}
    \item (Level 0) the class $\text{Obj}(\mathcal{B}_T) = D_2$. The string diagram is shown below:
    \begin{equation} \label{fig:level-0}
        \centering
        \includegraphics[scale=0.5]{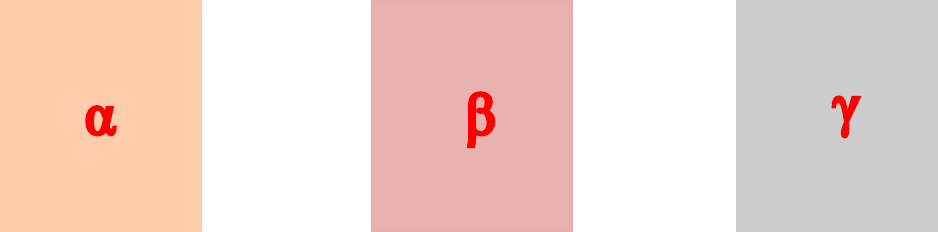}
    \end{equation}
    \item (Level 1) Given two objects $\alpha, \beta \in D_2$, a $\mathbb{K}$-linear category $\mathcal{B}_T(\alpha, \beta)$ whose objects are $1$-morphisms $X: \alpha \to \beta$ and as a category it is a free category generated by the pre-category given by maps $s, t : D_1 \to D_2$. (See \cref{condns}). As a string diagram:

    \begin{equation} \label{fig:level-1}
        \centering
        \includegraphics[scale=0.5]{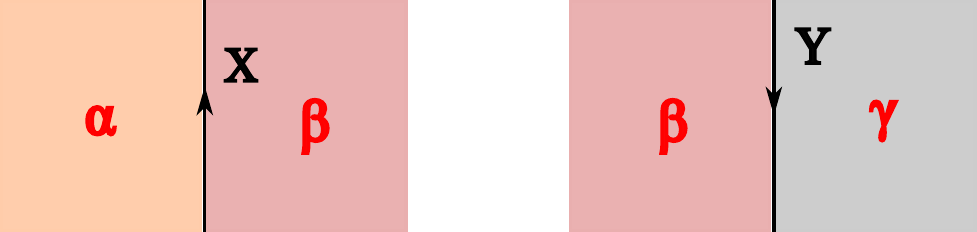}
    \end{equation}
    
    The horizontal composition of two $1$-morphisms is given by a $\mathbb{K}$-linear map $$ \mathcal{B}_T(\beta, \gamma) \otimes \mathcal{B}_T(\alpha, \beta) \to \mathcal{B}_T(\alpha, \gamma)$$ and is also called \textit{fusion}. It can be represented by the string diagram:
    \begin{equation} \label{fusion}
    \includegraphics[scale=0.4,valign=c]{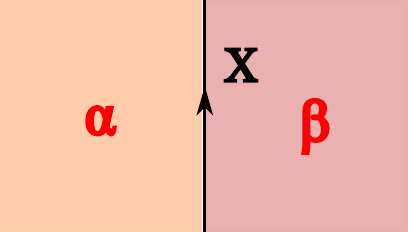}
  \otimes \vcenter{\hbox{\includegraphics[scale=0.4]{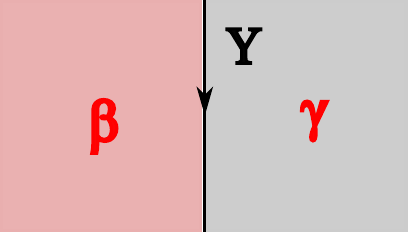}}} = \includegraphics[scale=0.4,valign=c]{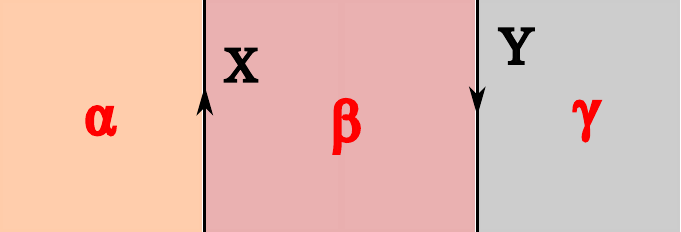}     
    \end{equation}

    The output in \cref{fusion} is usually written as $X \otimes Y^{-1} : \alpha \to \gamma$.

    \item (Level 2) The $\mathbb{K}$-linear space of $2$-morphism between $X \coloneqq (x_1^{\epsilon_1}, \dots x_n^{\epsilon_n}): \alpha \to \beta$ and $Y \coloneqq (y_1^{\nu_1}, \dots, y_m^{\nu_m}): \alpha \to \beta$, $\text{Hom}(X, Y)$ is given by $T(Y \otimes X^{-1})$:
    \begin{equation} \label{vertex-hole}
    Hom(X, Y) = T \Bigg(\includegraphics[scale=0.4,valign=c]{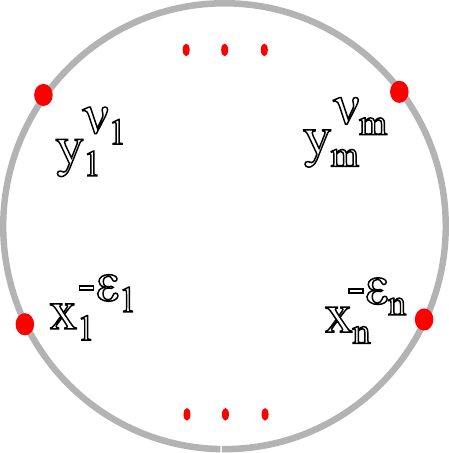} \Bigg)    
    \end{equation}  

    This corresponds to the local operators inserted at defect junctions. We refer to [~\cite{carqueville2016lecture} (2.17)] and [~\cite{davydov2011field}, (3.65)] for details on how to build the set $D_0$ from this data. For now, we only mention that one should think of this space as consisting of discs, with at most one labeled $0$-stratum, available to fill in. To see this, first note that Diagram (ii) in \cref{fig:defect_orient} can be thought as a cup $\hat{\mathbb{D}}$ and the TFT $T$ evaluates on it to $\mathbb{C} \to T(\partial \hat{\mathbb{D}})$, i.e. it assigns $\hat{\mathbb{D}}$ to a vector in the space $  T(\partial \hat{\mathbb{D}}) $. The case with \ref{vertex-hole} is similar: a disk with the boundary as this circle, when thought as a cup will assign a vector in the space $ \text{Hom}(X, Y) $. As an example, when $Y = X$ the identity $2$-morphism $\mathbf{1}_{X}$, which is an element of $\text{Hom}(X, X)$ is given by (i) below. On the other hand  (ii) shows the \textit{vertical composition} $v \circ u$ of two $2$-morphisms  $u \in \text{Hom}(X, Y)$ and $v \in \text{Hom}(Y, Z)$ . This is why we limited ourselves to disks with at most one vertex. Although, any disk with the boundary $Y \otimes X^{-1}$, when thought of as a cup, gets assigned a vector in $\text{Hom}(X, X)$ by $T$, it can always be written as vertical compositions of disks with single labeled $0$-stratum.
    \begin{equation*}
        \centering
        \includegraphics[scale=0.3]{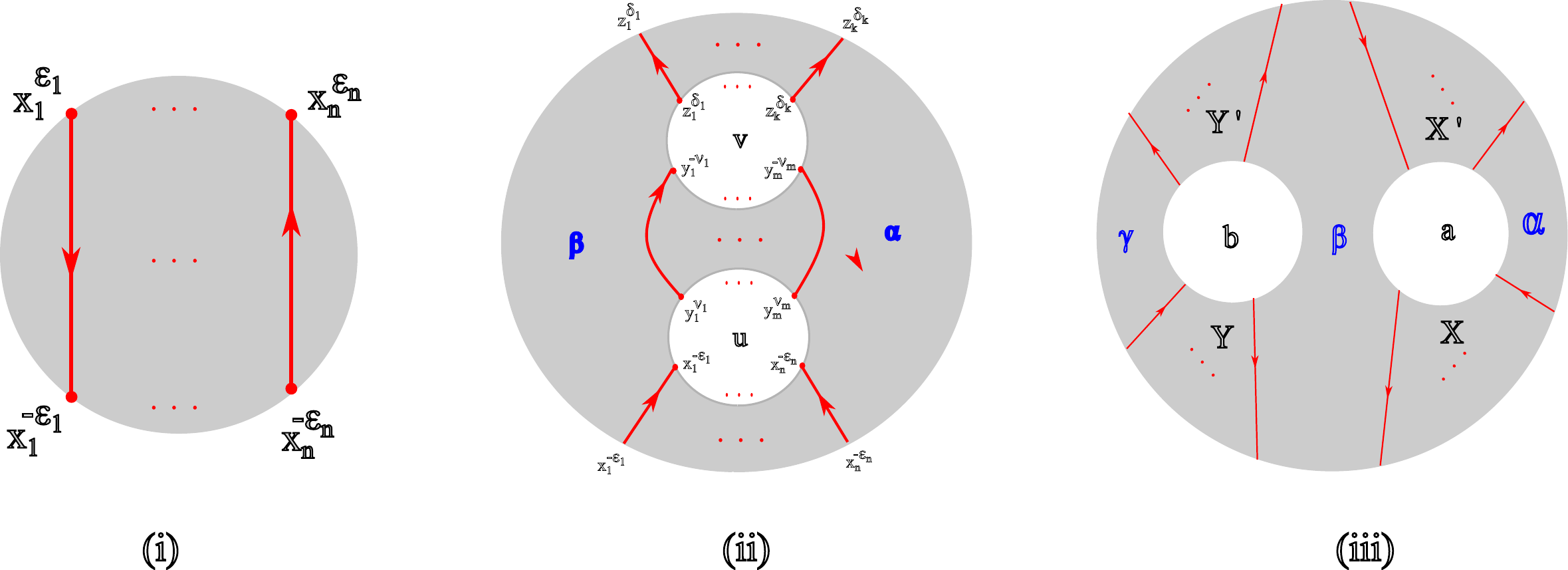}
    \end{equation*}
    While diagram (iii) shows the \textit{horizontal composition} $b \otimes a$ of $b \in \text{Hom}(Y, Y')$ and $a \in \text{Hom}(X, X')$.
\end{itemize}

\begin{rema}
    A consequence of the functoriality of fusion $\otimes$ is that the horizontal and vertical composition satisfies the \textit{interchange law}: For $\psi \in \text{Hom}(Y, Y'), \phi \in \text{Hom}(X, X')$
    
    \begin{equation} \label{interchange-law}
        \psi \otimes \phi = (\psi \otimes \mathbf{1}_{X}) \circ (\mathbf{1}_{Y'} \otimes \phi) = (\mathbf{1}_{Y}\otimes \phi) \circ (\psi \otimes \mathbf{1}_{X'})
    \end{equation}

    \begin{equation*}
        \centering
        \includegraphics[scale=0.3]{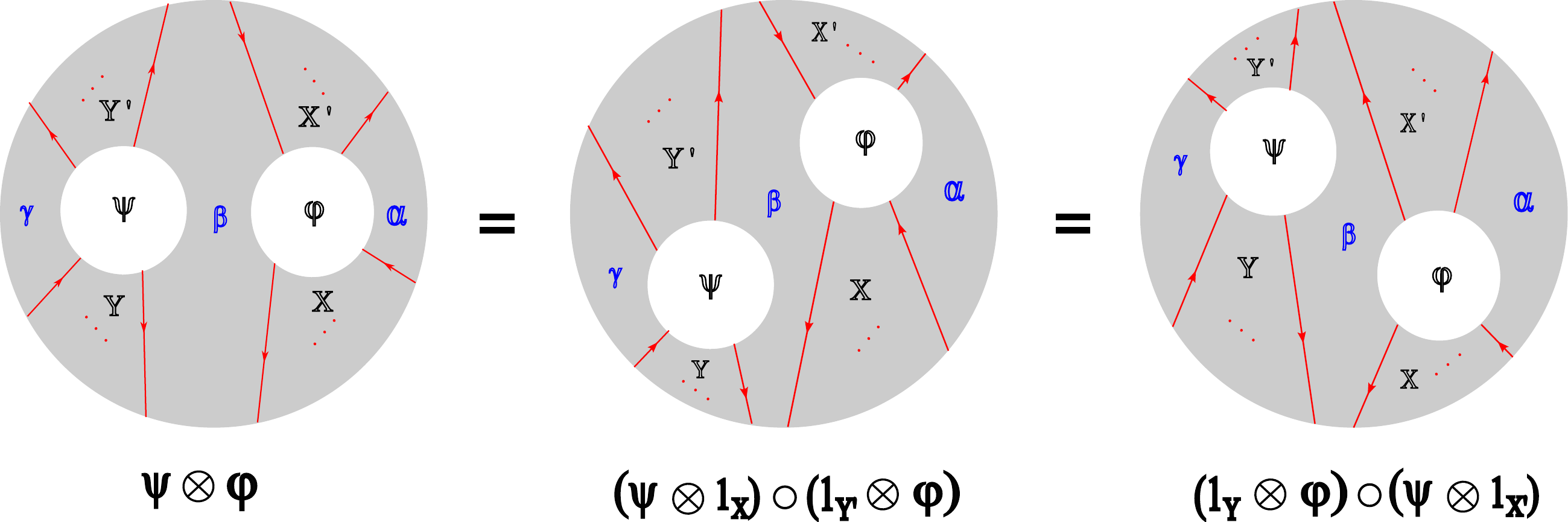}
    \end{equation*}
\end{rema}

Returning to the case of lattice TFT and trivial surrounding theory, the following proposition can now be stated:

\begin{figure}
    \centering
    \includegraphics[scale=0.35]{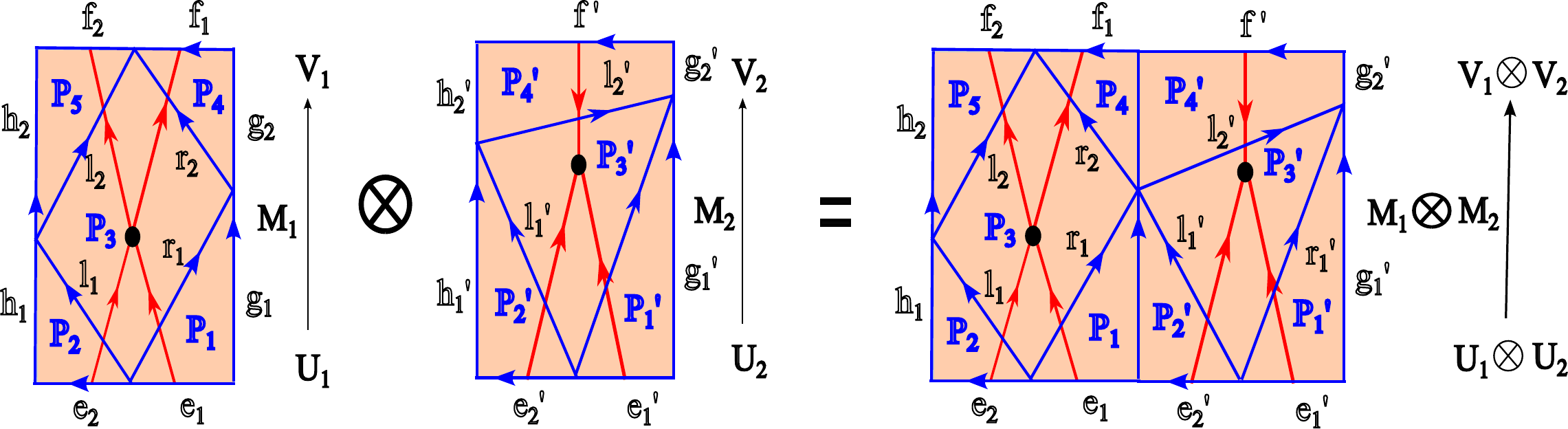}
    \caption{One can subdivide [~\cite{kirillov2012piecewise}, Section 6, Section 7] to get a cell decomposition on $M_1 \otimes M_2$ that restricts to the given cell-decomposition on both $M_1$ and $M_2$}
    \label{fig:TFT_Prop-2}
\end{figure}

\begin{prop} \label{Prop:fusion}
    Let $M_1: U_1 \to V_1$ and $M_2: U_2 \to V_2$ be two bordism restricted to respective polygons. With the definition of $T_o^{\textit{cw}}(M)$ as in \cref{Prop-main}, for the fusion, $$M_1 \otimes M_2 : U_1 \otimes U_2 \to V_1 \otimes V_2 , $$ we get $$T_0^{\textit{cw}}(M_1 \otimes M_2) = T_o^{\textit{cw}}(M_1) \otimes T_o^{\textit{cw}}(M_2).$$
\end{prop}

\begin{proof}

Choose a cell decomposition of $M_1 \otimes M_2$ such that it gives a cell decomposition of both $M_1$ and $M_2$ and such that the same holds for $U$ with respect to $U_1$ and $U_2$, and for $V$ with respect to $V_1$ and $V_2$. This is always possible after suitable refinements of cell decompositions of $M_1$ and $M_2$. See \cref{fig:TFT_Prop-2} for an example. With this, we get the data as follows:

\begin{equation} \label{proof-eqn-1}
    \begin{split}
        T_0^{\textit{cw}}(U) & = \bigotimes_{e \in C_1(U)}R_e \\
        & = \bigotimes_{e \in C_1(U_1)}R_e \bigotimes_{e' \in C_1(U_2)}(R_{e'}) \\
        & = T_0^{\textit{cw}}(U_1) \otimes T_0^{\textit{cw}}(U_2)
    \end{split}
\end{equation}

Similarly, we get $T_0^{\textit{cw}}(V) = T_0^{\textit{cw}}(V_1) \otimes T_0^{\textit{cw}}(V_2)$. Furthermore,

\begin{equation} \label{proof-eqn-2}
  \begin{split}
     \mathscr{P}(M_1 \otimes M_2) & = \bigotimes_{\substack{e \in C_1(M_1 \otimes M_2)\\ e \notin \partial_{in}(M_1 \otimes M_2)}} \mathscr{P}_e \\
     & = \bigotimes_{\substack{e \notin \partial_{in} M_1 \\ e \in C_1(M_1)}} \mathscr{P}_e \bigotimes_{\substack{e' \notin \partial_{in}M_2\\ e' \in C_1(M_2) }} \mathscr{P}_{e'} \\
     & = \mathscr{P}_{M_1} \otimes \mathscr{P}_{M_2}
  \end{split}
\end{equation}

and

\begin{equation} \label{proof-eqn-3}
   \begin{split}
   \mathscr{E}(M_1 \otimes M_2) & = \bigotimes_{P \in C_2(M_1 \otimes M_2)}\mathscr{E}_P\\
   & = \bigotimes_{P \in C_2(M_1)} \mathscr{E}_p \bigotimes_{P' \in C_2(M_2)} \mathscr{E}_{P'}\\
   & = \mathscr{E}(M_1) \otimes \mathscr{E}(M_2)
  \end{split}
\end{equation}

In defining the propagator $\mathscr{P}(M)$ in \cref{P-E-polygons} we kept only one copy of $Q_{(P(e'), e', \mathfrak{O})}$ for an externel edge $e'$ of the polygon $M$. We get two copies of $Q_{(P(e'), e', \mathfrak{O})}$ with opposite signs $\mathfrak{O}$ from this edge in $M_1 \otimes M_2$ - one from each of $M_1$ and $M_2$; as one would expect from an internal edge.

Next, $T_0^{\textit{cw}}(M_1 \otimes M_2)$ is given by the composition

\begin{equation*}
    \begin{aligned}
        T_0^{\textit{cw}}(U_1 \otimes U_2) \xrightarrow{\mathbf{1} \otimes \mathscr{P}(M_1 \otimes M_2)} T_0^{\textit{cw}}(U_1 \otimes U_2) \otimes Q(M_1 \otimes M_2) \otimes T_0^{cw}(V_1 \otimes V_2) \\
        \xrightarrow{E(M_1 \otimes M_2) \otimes \mathbf{1}} T_0^{\textit{cw}}(V_1 \otimes V_2)
    \end{aligned}
\end{equation*}
    
Using \cref{proof-eqn-1}, \cref{proof-eqn-2}, \cref{proof-eqn-3} and $Q(M_1 \otimes M_2) = Q(M_1)\otimes Q(M_2)$ we get
    
\begin{equation} \label{proof-eqn-4}
    \begin{aligned}
        T_0^{\textit{cw}}(U_1) \otimes T_0^{\textit{cw}}(U_2) \xrightarrow{\mathbf{1} \otimes \mathscr{P}(M_1) \otimes \mathscr{P}(M_2)} T_0^{\textit{cw}}(U_1) \otimes T_0^{\textit{cw}}(U_2) \otimes Q(M_1) \otimes Q(M_2) \\ \otimes T_0^{\textit{cw}}(V_1) \otimes T_0^{\textit{cw}}(V_2)
        \xrightarrow{\mathscr{E}(M_1) \otimes \mathscr{E}(M_2) \otimes \mathbf{1}} T_0^{\textit{cw}}(V_1) \otimes T_0^{\textit{cw}}(V_2).
    \end{aligned}
\end{equation}

After arranging, \cref{proof-eqn-4} gives

\begin{equation}
    \begin{aligned}
        T_0^{\textit{cw}}(U_1) \otimes T_0^{\textit{cw}}(U_2) \xrightarrow{\mathbf{1} \otimes \mathscr{P}(M_1) \otimes \mathscr{P}(M_2)} T_0^{\textit{cw}}(U_1) \otimes Q(M_1) \otimes T_0^{\textit{cw}}(V_1) \otimes T_0^{\textit{cw}}(U_2) \\ \otimes  Q(M_2) \otimes  \otimes T_0^{\textit{cw}}(V_2)
        \xrightarrow{\mathscr{E}(M_1) \otimes \mathscr{E}(M_2) \otimes \mathbf{1}} T_0^{\textit{cw}}(V_1) \otimes T_0^{\textit{cw}}(V_2).
    \end{aligned}
\end{equation}    

Now, look at the following composition of maps:

\begin{equation}
    \begin{aligned}
        T_0^{\textit{cw}}(U_1) \otimes T_0^{\textit{cw}}(U_2) \xrightarrow{\mathbf{1} \otimes \mathscr{P}(M_1) \otimes \mathbf{1}} T_0^{\textit{cw}}(U_1) \otimes Q(M_1) \otimes T_0^{\textit{cw}}(V_1) \\ \otimes T_0^{\textit{cw}}(U_2) \xrightarrow{\mathscr{E}(M_1)\otimes \mathbf{1} \otimes \mathbf{1}} T_0^{\textit{cw}}(V_1) \otimes T_0^{\textit{cw}}(U_2) \xrightarrow{\mathbf{1} \otimes \mathbf{1} \otimes \mathscr{P}(M_2)} T_0^{\textit{cw}}(V_1)\\ \otimes T_0^{\textit{cw}}(U_2) \otimes Q(M_2) \otimes T_0^{\textit{cw}}(V_2)
        \xrightarrow{\mathbf{1} \otimes \mathscr{E}(M_2) \otimes \mathbf{1}} T_0^{\textit{cw}}(V_1) \otimes T_0^{\textit{cw}}(V_2).
    \end{aligned}
\end{equation}    
       
This is the map $(\mathbf{1} \otimes T_0^{\textit{cw}}) \circ (T_0^{\textit{cw}} \otimes \mathbf{1})$ on $U_1 \otimes U_2$, but it also equals:

      $$(\mathbf{1} \otimes \mathscr{E}(M_2) \otimes \mathbf{1}) \circ (\mathbf{1} \otimes \mathbf{1} \otimes \mathscr{P}(M_2) ) \circ (\mathscr{E}(M_1) \otimes \mathbf{1} \otimes \mathbf{1}) \circ (\mathbf{1} \otimes \mathscr{P}(M_1) \otimes \mathbf{1}).$$

The two terms in the middle can be interchanged as a consequence of the interchange law in the monoidal category $\text{Vect}_F(\mathbb{C})$. This gives,
  
  $$(\mathbf{1} \otimes \mathscr{E}(M_2) \otimes \mathbf{1}) \circ  (\mathscr{E}(M_1) \otimes \mathbf{1} \otimes \mathbf{1}) \circ (\mathbf{1} \otimes \mathbf{1} \otimes \mathscr{P}(M_2) ) \circ  (\mathbf{1} \otimes \mathscr{P}(M_1) \otimes \mathbf{1}),$$ but that equals to:

  $$(\mathscr{E}(M_1) \otimes \mathscr{E}(M_2) \otimes \mathbf{1}) \circ (\mathbf{1} \otimes \mathscr{P}(M_1) \otimes \mathscr{P}(M_2)).$$

Comparing this with \cref{proof-eqn-3} we get the desired result.
  
\end{proof}

Alternatively, one could first prove, using a suitable PLCW decomposition $T_0^{\textit{cw}}(1 \otimes M) = \mathbf{1} \otimes T_0^{\textit{cw}}(M)$. Then functoriality and interchange law to prove \cref{Prop:fusion}:

\begin{equation*}
    \begin{split}
        T_0^{\textit{cw}}(M_1 \otimes M_2) & = T_0^{\textit{cw}}((M_1 \otimes 1) \circ (1 \otimes M_2)) \\
        & = T_0^{\textit{cw}}((M_1 \otimes 1)) \circ T_0^{\textit{cw}}((1 \otimes M_2)) \hspace{2mm} \\
        & = (T_0^{\textit{cw}}(M_1)\otimes \mathbf{1}) \circ (\mathbf{1} \otimes T_0^{\textit{cw}}(M_2)) \\
        & = T_0^{\textit{cw}}(M_1) \otimes T_0^{\textit{cw}}(M_2) \hspace{2mm}
    \end{split}
\end{equation*}


\section{Finitely presented groups and defects}

As mentioned in the caption for \cref{basic_gons} one can read off all the ingredients of the defect conditions $\mathcal{D}$ of the category $\textit{Bord}_2^{\text{def}}(\mathcal{D})$ by treating all the basic-gons in the category $\textit{Bord}_2^{\text{def, cw}}(\mathcal{D})$ as defect-disks from \cref{local-junctions}. The basic-gons of type (iii) gives the map $\psi_{1,2}$, while the map $\psi_{0,1}$ is obtained from the basic gons of type (iv) as explained in the caption of \cref{basic_gons}. Moreover, if $t(x) = s(x)$ for some $x \in D_1$, then $x$ can not be distinguished from $x^{-1}$ and we can get rid of direction (orientation; note that this condition is trivially satisfied if the set $D_2$ is a singleton.) In other words, the category $\textit{Bord}_2^{\text{def, cw}}(\mathcal{D})$ can be fully specified by specifying all the basic-gons that can appear in this category. Hence, using the forgetful functor in \cref{plcwcd} the category $\textit{Bord}_2^{\textit{def}}(\mathcal{D})$ can also be fully specified in this manner.

\subsection{Surface of defects from a Group presentation} \label{sec:groups}

With that in mind, given a finitely presented group $G$, with the presentation $P_G \coloneqq \langle B_G \mid R_G \rangle$ we define a collection of basic-gons for each (trivial or non-trivial) relations in $R_G$ as shown in \cref{fig:plcwgrp}, and use this to define a set of defect conditions by treating basic-gons as defect disks.

\begin{figure}[ht]
        \centering
        \includegraphics[scale=0.5]{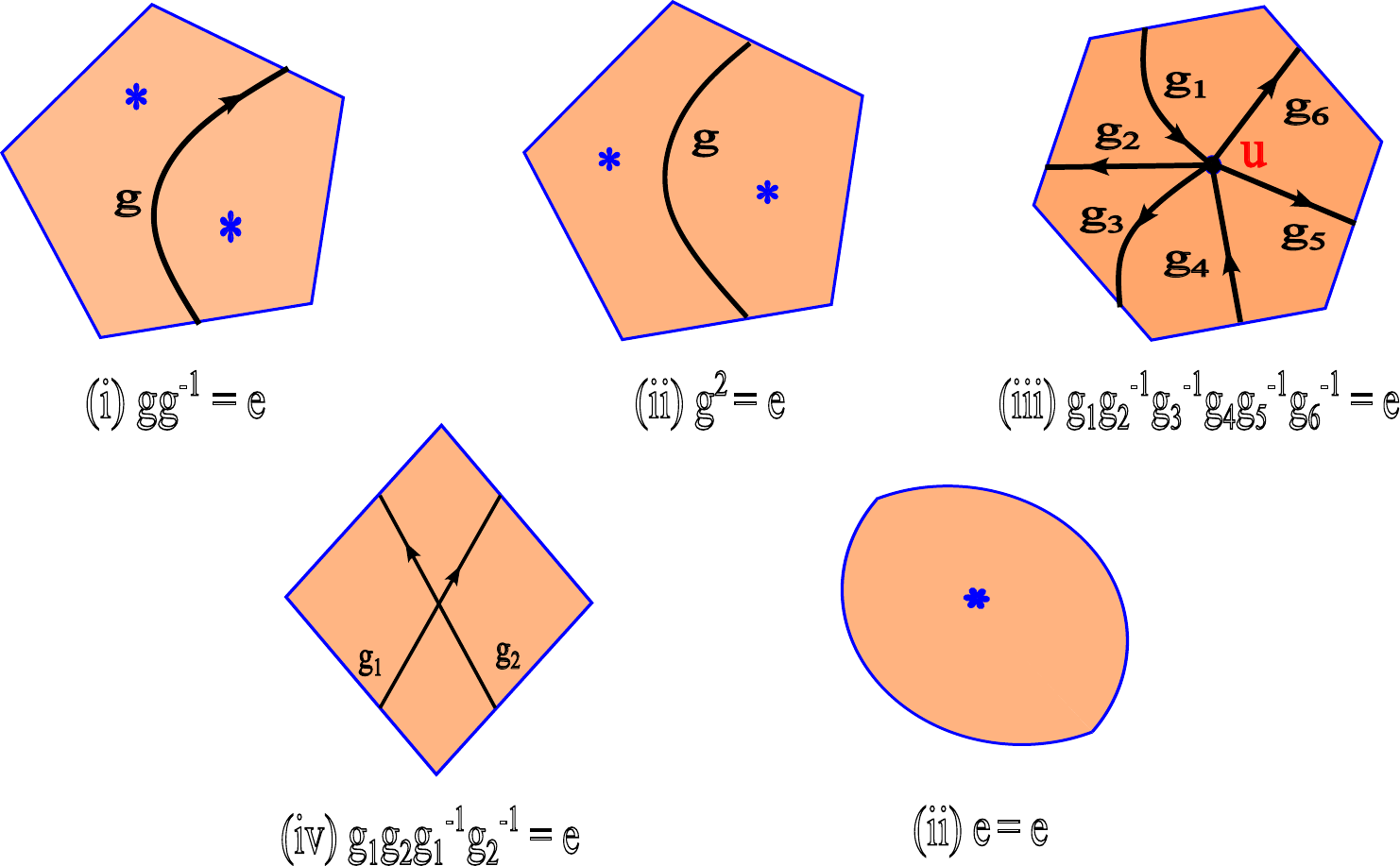}
        \caption{The condition for diagram (ii) is trivially satisfied when $D_2$ is a singleton. Diagram (iv) depicts the commutator relation. There is a special vertex there which is characterised by the property that it is idempotent under the operation of vertical-composition. See the caption below \cref{fig:twist}.}
        \label{fig:plcwgrp}
\end{figure}

\begin{defi} \label{def:group_condns}
    Given a finitely presented group $G$ with the presentation $P_G \coloneqq \langle B_G \mid R_G \rangle$, we define a set of defect conditions $\mathcal{P_G} \coloneqq (D_2, D_1, D_0, \psi_{\{1,2\}}, \psi_{\{0,1\}})$ as follows:
    \begin{enumerate}
    
        \item The set $D_2$ is a singleton, $\{\ast\}$.
        \item The set $D_1$ is the set of generators $B_G$, and
        \item The set $D_0$ is in one-to-one correspondence with the words in the relation $R_G$, i.e., if $\mathit{r} \in R_G$, then $D_0 \coloneqq \{u_{\mathit{r}} \mid \mathit{r} \in R_G\}$.
        \item The map $\psi_{\{1,2\}}: X_1 \to D_2 \times D_2$ is trivial, i.e., for all $g \in B_G, \psi_{\{1,2\}}: g^{\epsilon} \mapsto (\ast, \ast)$.
        \item For each relation $\mathit{r} \in R_G$ such that $\mathit{r} = g_{i_1}^{\epsilon_1} \dots g_{i_m}^{\epsilon_m}$, form the map $$\psi_{\{0,1\}}: u_{\mathit{r}} \mapsto [(g_{i_1}^{\epsilon_1}, \dots , g_{i_m}^{\epsilon_m})], $$ and define $\psi_{\{0,1\}}(u_r^{-1}) \coloneqq \psi_{\{0,1\}}(u_{r^{-1}})$.
        
    \end{enumerate}
\end{defi}

We need the following: 

\begin{prop} \label{prop:plcwgrp}
    \cref{def:group_condns} is well defined.
\end{prop}

\begin{proof}
    We only need to check that the maps $\psi_{1,2}$ and $\psi_{0,1}$ are well-defined and satisfy the orientation consistency conditions of \cref{condns}. Since the set $D_2$ is singleton, the map $\psi_{1,2}$ is well-defined and trivially satisfies the orientation consistency condition. The fact that $\psi_{0,1}$ is well-defined and satisfies the orientation-consistency condition follows from the following easy facts about groups:
    \begin{itemize}
        \item If a word $g_{i_1}^{\epsilon_1} \dots g_{i_n}^{\epsilon_1}$ is in $R_G$ then so is any cyclic permutation and inverse of it.
        \item If $\mathit{r} \in R_G$ is such that $\mathit{r} = g_{i_1}^{\epsilon_1} \dots g_{i_m}^{\epsilon_m}$, then  $\mathit{r}^{-1} = g_{i_m}^{-\epsilon_m} \dots g_{i_1}^{-\epsilon_1}$.
    \end{itemize}

\end{proof}

The basic-gon (iv) in \cref{fig:plcwgrp} can be better understood in terms of the interpretation of the data of $D_0$ given in ~\ref{vertex-hole}. The basic-gon (iv) should be considered a $2$-morphism $\tau: X \otimes Y \to Y \otimes X$ with the property: $\tau \circ \tau = I$. Where '$\circ$' is the vertical composition of $2$-morphisms. Diagrams in \cref{fig:twist} explain it better by drawing an analogy with the identity map. See also \cref{sec:word-problem_theory}, \cref{sec:word_problem_future}, and \cref{sec:ribbon_graph}.

\begin{figure}
    \centering
    \includegraphics[scale=0.7]{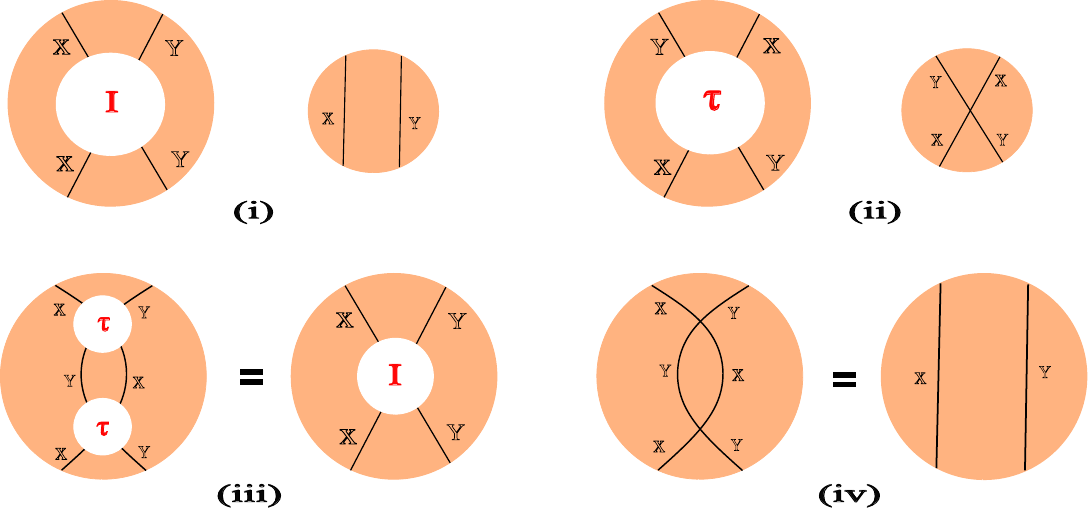}
    \caption{Diagram (ii) represent a $2$-morphism $\tau: X \otimes Y \to Y \otimes X$ in the same manner as Diagram (i) represent the identity $2$-morphism $I: X \otimes Y \to X \otimes Y$. Diagram (iv) shows the interpretation of Diagram (iii) in terms of the filled-in discs analogy of $2$-morphism mentioned in the paragraph below ~\ref{vertex-hole}.}
    \label{fig:twist}
\end{figure}

\begin{rema}
    It is not difficult to notice that \cref{def:group_condns}, (5) can also be flipped, and the tuple on the right-hand side of $\psi_{\{0,1\}}(u^{\epsilon})$ in \cref{condns} (2) can be written as a word.
\end{rema}

\begin{rema}
    $D_2$ does not have to be a singleton. Any set $D_2$ for which the map $\psi_{1,2}: X_1 \to D_2 \times D_2$ is well-defined and satisfies orientation consistency condition can be taken as $D_2$. However, we do not pursue it in this manuscript.
\end{rema}

\begin{defi} \label{associated-def}

    Given a group $G$ and a presentation $P_G \coloneqq \langle B_G \mid R_G \rangle$ we define the category $\textit{Bord}_2^{\text{def, cw}}(\mathcal{P_G})$ (and $\textit{Bord}_2^{\text{def}}(\mathcal{P_G})$) as follows:
    \begin{enumerate}
        \item as a category it is $\textit{Bord}_2^{\text{def}}(\mathcal{D})$ where the defect condition$\mathcal{D}$ is $$\mathcal{P_G} \coloneqq (D_2, D_1, D_0, \psi_{\{1,2\}}, \psi_{\{0,1\}})$$ from \cref{def:group_condns}.
        \item The basic-gons corresponds to words in $R_G$ as in \cref{fig:plcwgrp}.
    \end{enumerate}
    
\end{defi}

In simple language, the category $\textit{Bord}_2^{\text{def, cw}}(\mathcal{P_G})$ has morphisms as surfaces with defects with a PLCW decomposition such that each generalized cell looks like one of the basic-gons in \cref{fig:plcwgrp}, and the category   $\textit{Bord}_2^{\text{def}}(\mathcal{P_G})$ is obtained from the category $\textit{Bord}_2^{\text{def, cw}}(\mathcal{P_G})$ by using the forgetful functor $F$ from the \cref{plcwcd}.

\begin{rema}
    The mathematical object we have discussed in \cref{def:group_condns} and \cref{associated-def} is not as new as it may seem. To see this, let us recall that the string diagram is dual (in the sense of a graph) to the usual pasting diagram of a category [cf ~\cite{barrett2024graycategoriesdualsdiagrams}, Section 2]. Then we see that the pasting diagram for the $2$-category of defect conditions $\mathcal{P_G}$ consists of a single vertex (since $D_2$ is a singleton), loops labeled by the generators of $B_G$, and, finally, the dual of a diagram of type (ii) \cref{fig:defect_orient} is polygon with edges labeled by relations in a cyclic order, which is just a topological disk to be attached to the loops labeled by generators according to relations in $R_G$. This is [~\cite{hatcher2002algebraic}, Corollary 1.28] and is related to the concept of Eilenberg-MacLane space [~\cite{hatcher2002algebraic}, Section 1.B, Section 4.2.]
\end{rema}

.

\begin{exam} \label{exam:coloring}
    Let $K_4 $ be the \textit{Klein four-group} with the presentation $ P_{K_4} \coloneqq \langle a, b, c \mid a^2 = b^2 = c^2 = abc = 1 \rangle$. The basic-gons for the category $\textit{Bord}_2^{\text{def, cw}}(\mathcal{P}_{K_4})$ are shown in \cref{fig:colored}.
  
     \begin{figure}
        
        \centering
        \includegraphics[scale=0.5]{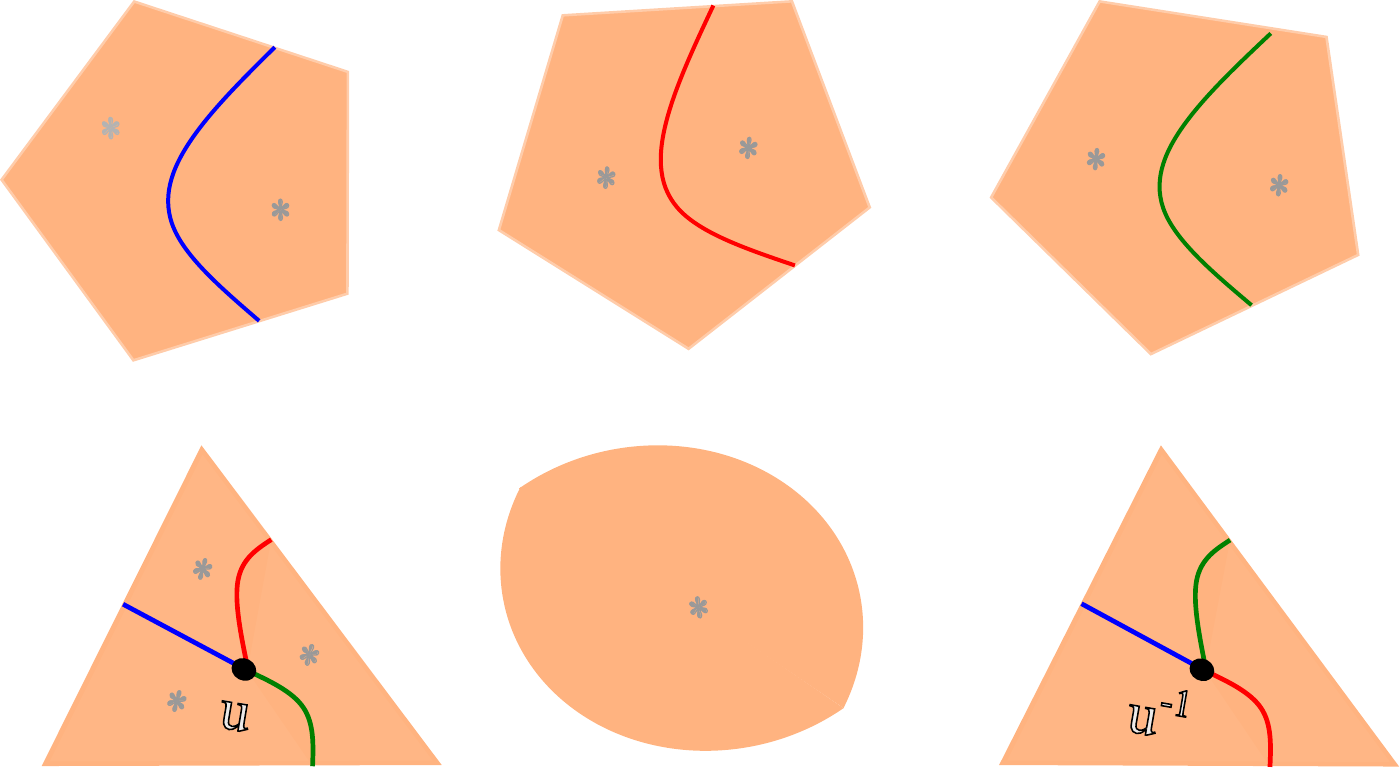}
        \caption{We have denoted $a, b, c$ by colors blue, red, green respectively.}
        \label{fig:colored}

     \end{figure}

Although, we have not defined \textit{coloring} yet (we will do it in the next section) relying on pictures \cref{fig:colored} for now, we note that a surface with defect $\hat{\Sigma}$ in $\textbf{Mor}(\textit{Bord}_2^{\text{def, cw}}(\mathcal{P}_{K_4}))$ is precisely a pair $(\Sigma, \Gamma)$ where $\Gamma$ is a trivalent $3$-edge colorable graph embedded in $\Sigma$. Here by a coloring of an edge $\hat{e}$ of $\Gamma$ we mean the image of $\hat{e}$ under $d$ in $B_{K_4}$, which is the set $\{a, b, c\}$.

\end{exam}    

\begin{exam}
    Consider the symmetric group $S_n$ with the presentation 

   \[
        S_N = \left\langle \tau_1, \dots , \tau_{n-1} \Bigg|
        \begin{array}{lr}
        \tau_i\tau_j = \tau_j\tau_i \hspace{3mm} \mid i - j \mid > 1\\
        \tau_i\tau_{i+1}\tau_i \hspace{2.5mm} = \hspace{2.5mm} \tau_{i+1}\tau_{i}\tau_{i+1} \\
          \hspace{10mm}\tau_i^2 \hspace{1mm} = \hspace{1mm} 1
        \end{array}\right\rangle
  \]    
    
    The basic-gons for the category $\textit{Bord}_2^{\text{def, cw}}(\mathcal{S}_4)$ are given in \cref{fig:hexagonal}.

    \begin{figure}
        \centering
        \includegraphics[scale=0.5]{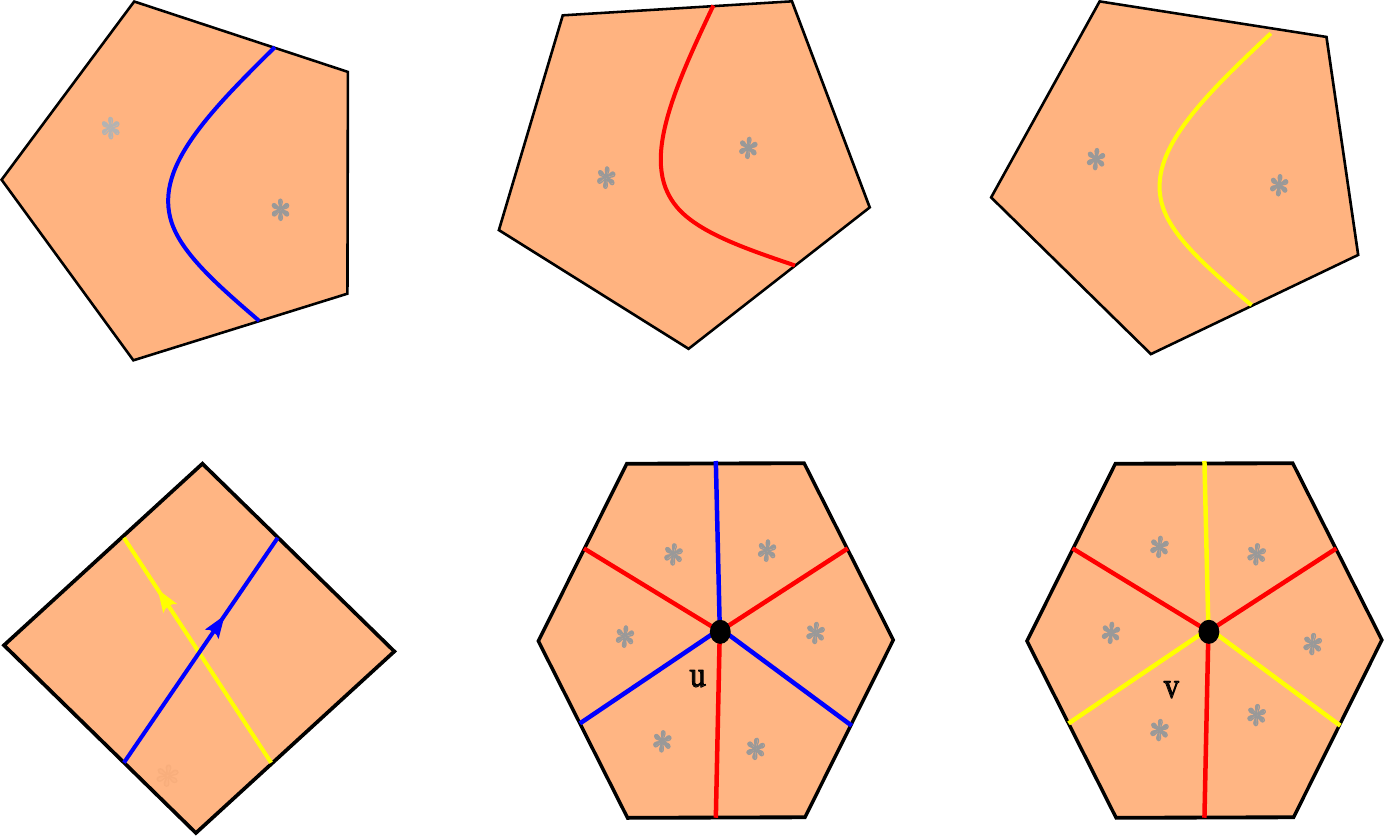}
        \caption{Baisc-gons in  $\textit{Bord}_2^{\text{def, cw}}(\mathcal{S}_4)$. Each generator has been represented by a different color: $\tau_1$ by blue, $\tau_2$ by red, and $\tau_3$ by yellow.}
        \label{fig:hexagonal}
    \end{figure}

The class $\textbf{Mor}(\textit{Bord}_2^{\text{def, cw}}(\mathcal{S}_N))$ consists of $N$-graphs with only hexagonal vertices. It is worth mentioning that a general $N$-graph also has trivalent vertices. Note that $N$-graphs were introduced in ~\cite{casals2023legendrian} where they were used to construct Legendrian surfaces in the first jet space of the underlying surface of the $N$-graph. We do not make this construction explicit here.

\end{exam}

\begin{defi} \label{def-aleph}

    Given $D_2 = \{\ast\}$ and $D_1 = \{e\}$, both singletons, we define a category $\textit{Bord}_2^{\textit{def, cw}}(\mathcal{D}^{\aleph})$ with the property that it has a basic-gon of type-(iii) with $n$ lines for every $n \geq 2$. For a fixed $n \in \mathbb{Z}_{+}$, define the subcategory $\textit{Bord}_2^{\textit{def, cw}}(\mathcal{D}^{\mathbf{n}})$ with the property that the only basic-gon of type-(iii) it has are those with $n $ $1$-strata. Similarly, the category $\textit{Bord}_2^{\textit{def,cw}}(\mathcal{D}_{+}^{\mathbf{n}})$ is the category with $n$-regular undirected graphs.
\end{defi}

\begin{defi} \label{defn:bleach-cat}

    The \textit{forgetful functor} $$\Pi^{cw} : \textit{Bord}_2^{\textit{def, cw}}(\mathcal{P}_G) \to \textit{Bord}_2^{\textit{def, cw}}(\mathcal{D}^{\aleph})$$ is defined as follows:
    \begin{itemize}
        \item on \textbf{objects} it changes the labels on a (disjoint union of) circles from $g^{\epsilon} \in B_{G}$ to $e^{\epsilon}$, and
        \item on \textbf{morphism} it is defined via its action on basic-gons, where it changes a label of $1$-strata from $g^{\epsilon} \in B_{G}$ to $e^{\epsilon}$.
    \end{itemize}
    
\end{defi}

Again, referring to pictures, the map $\Pi^{cw}$ can be thought of as \textit{bleaching}: all the colors on $1$-dimensional strata are forgotten, i.e., get replaced by $e$ without forgetting the signs.

Given two objects $O_1$ and $O_2$ in $\textit{Bord}_2^{\textit{def,cw}}(\mathcal{P}_G)$, the assignment $\Sigma \mapsto \Pi^{cw}(\Sigma)$ induces a function  

 \begin{equation} \label{bleach}
     \pi^{cw} : \text{Mor}(\textit{Bord}_2^{\textit{def,cw}}(\mathcal{P}_G))(O_1, O_2) \to  \text{Mor}(\textit{Bord}_2^{\textit{def,cw}}(\mathcal{D}^{\aleph})(\Pi^{cw}(O_1),  \Pi^{cw}(O_2).
\end{equation}

Since a surface $\Sigma$ in the set $\text{Mor}(\textit{Bord}_2^{\textit{def,cw}}(\mathcal{P}_G))$ includes the information about the source and target objects in its boundaries via the cobordism $\begin{tikzcd}[scale cd =0.5] & \hat{\Sigma} \arrow[dr] & \\ O_1 \arrow[ur] &  & O_2 \end{tikzcd}$, we simply write \cref{bleach} as:
\begin{equation} \label{bleach-reduced}
     \pi^{cw} : \text{Mor}(\textit{Bord}_2^{\textit{def,cw}}(\mathcal{P}_G)) \to  \text{Mor}(\textit{Bord}_2^{\textit{def,cw}}(\mathcal{D}^{\aleph}).
\end{equation}

We conclude this subsection with the following remarks:

\begin{rema}

    The category $\textit{Bord}_2^{\textit{def}}(\mathcal{D}^{\aleph})$, which is obtained from $\textit{Bord}_2^{\textit{def, cw}}(\mathcal{D}^{\aleph})$ using the forgetful functor $F$ in \cref{plcwcd}, is the category that uses single defects for both $D_2$, $D_1$ while adjusting $D_0$ accordingly.
    
\end{rema}

\begin{rema} \label{rema:pre-bleach}

    The functor $\Pi^{cw}$ from \cref{defn:bleach-cat} induces a forgetful functor, $$ \Pi: \textit{Bord}_2^{\textit{def}}(\mathcal{P_G}))\to \textit{Bord}_2^{\textit{def}}(\mathcal{D}^{\aleph}))$$ in a canonical way, namely the following diagram commutes on the level of functors.

    \begin{equation} \label{eqn:plcwcd-bleach}
    \begin{tikzcd}
    \textit{Bord}_2^{\text{def, cw}}(\mathcal{P_G}) \arrow[rr, "\hat{\Pi^{cw}}"] \arrow[d, "F"] &  & \textit{Bord}_2^{\text{def,cw}}(\mathcal{D}^{\aleph}) \arrow[d, "F"]\\
     \textit{Bord}_2^{\text{def}}(\mathcal{P_G}) \arrow[rr, "\Pi"] &  & \textit{Bord}_2^{\text{def}}(\mathcal{D}^{\aleph})    
     \end{tikzcd}
    \end{equation}
There is also a function

\begin{equation} 
     \pi : \text{Mor}(\textit{Bord}_2^{\textit{def}}(\mathcal{P}_G)) \to  \text{Mor}(\textit{Bord}_2^{\textit{def}}(\mathcal{D}^{\aleph})
\end{equation}

analogous to the functor $\pi^{cw}$ in \cref{bleach-reduced}

\end{rema}

By abuse of notation, we will refer to both $\pi^{cw}$ and $\pi$ as \textit{bleach}.

\subsection{A special trivial surrounding theory} \label{chi-cw}

This section aims to give an example of a lattice TFT that is a trivial surrounding theory, introduced in \cref{trivial_theory}. In what follows, let $X$ be a $\mathbb{C}$-vector space generated by $a, b$ and $c$. Further let $X^{\ast}$ be the dual vector space with corresponding dual basis $a^{\ast}, b^{\ast}$ and $c^{\ast}$. We assume $X \cong X^{\ast}$ via the induced inner product. A choice of such an $X$ gives meaning to the fact that on an undirected graph, one can choose any direction when calculating a TFT.

\begin{defi} \label{color_direct_TFT}
Let $X$ be the vector space $\mathbb{C}\langle a, b, c \rangle$ as in \cref{Tait_coloring}. We define the trivial surrounding theory
$\chi^{cw}: \textit{Bord}_2^{\textit{def,cw}}(\mathcal{D}^{\mathbf{3}}) \to \text{Vect}_F(\mathbb{C})$, with the properties that
\begin{itemize}
    \item it assigns to a $1$-cell $e$ containing the single defect, the vector space $R_e = X$, and
    \item to a trivalent vertex $u$, the map $\mu: X \otimes X \otimes X \to \mathbb{C}$ as defined in \cref{Tait_coloring}.
\end{itemize}

\begin{rema}
    A consequence of \cref{color_direct_TFT} is that $\chi^{cw}$ assigns to a circle with $n$-defects, the vector space $X^{\otimes n}$.
\end{rema}

\begin{rema} \label{color_TFT}
    Since $X \cong X^{\ast}$, $\chi^{cw}$ passes to a functor,
    \begin{equation}
        \chi^{cw}: \textit{Bord}_2^{\textit{def,cw}}(\mathcal{D}_{+}^{\mathbf{3}}) \to \text{Vect}_F(\mathbb{C}),
    \end{equation}
    which we also write as $\chi^{cw}$ by the abuse of notation. The category $\textit{Bord}_2^{\textit{def,cw}}(\mathcal{D}_{+}^{\mathbf{3}})$ was defined in \cref{def-aleph}. (To compute with it, we choose any orientation for the $1$-strata.)
\end{rema}

\end{defi}

Now, we compute the value of $\chi^{cw}$ on some simple patterns and basic-gons. We begin with $\chi^{cw}(P_0)$ for the polygon $P_0 : U \to V$ as shown below:

\begin{equation} \label{Pattern-1}
    \centering
    \includegraphics[scale=0.5]{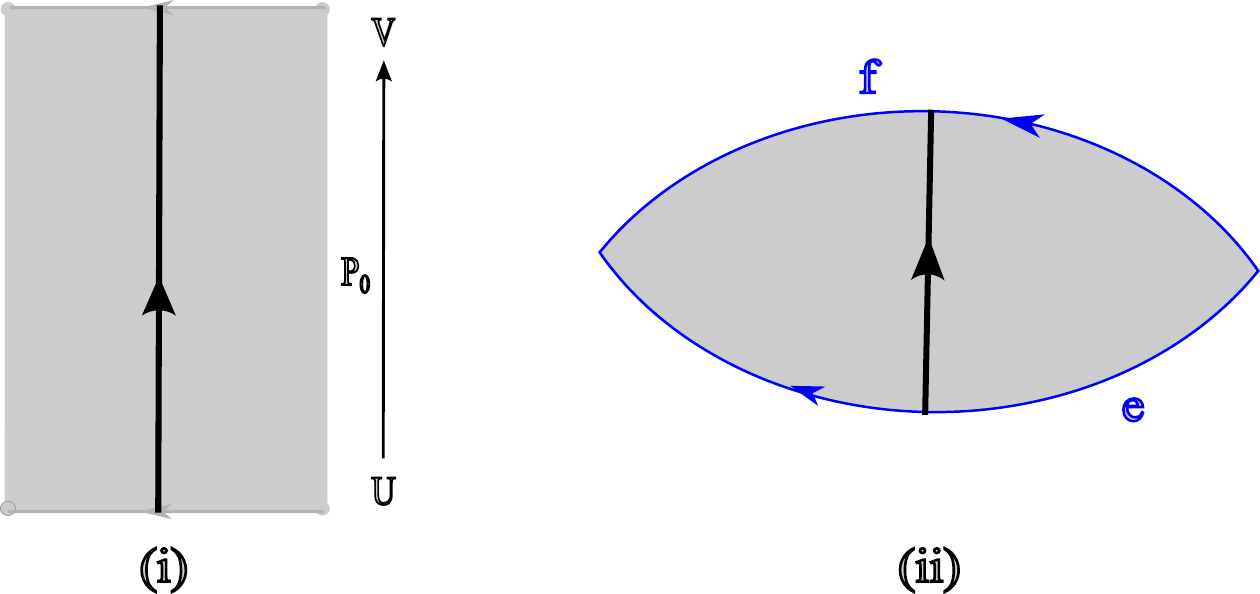}
\end{equation}

In the light of \cref{Thm:simplest}, the cell-decomposition in Diagram (ii) can be used to calculate $\chi^{\textit{cw}}(P_0)$. First, we see that both of $\chi^{\textit{cw}}(U)$ and $ \chi^{\textit{cw}}(V)$ is $X$, which gives $\chi^{\textit{cw}}(P_0): X \to X$. What is this map? This can be calculated using the following composition:
\begin{equation*}
    \chi^{\textit{cw}}(P_0): \chi^{\textit{cw}}(U) \xrightarrow{\mathbf{1} \otimes \mathscr{P}(P_0)} \chi^{\textit{cw}}(U) \otimes Q(P_0) \otimes \chi^{\textit{cw}}(V) \xrightarrow{\mathscr{E}(P_0) \otimes \mathbf{1}} \chi^{\textit{cw}}(V) 
\end{equation*}

We need the data:

\begin{center}
\begin{tabular}{ |c |c| c |}
\hline
 $(P_0, e, \mathfrak{O})$ & $(P_0, e, -)$ & $(P_0, f, +)$ \\ 
\hline
 $Q_{(P_0, e, \mathfrak{O})}$ & $X_e$ & $X_f^{\ast}$ \\ 
\hline 
\end{tabular}
\end{center}

From this we get $Q(P_0) = Q_{(P_0, f, +)}$ and $\mathscr{P(P_0)} = \mathscr{P}_f$ which is the co-pairing map $\mathbb{C} \to X_f^{\ast} \otimes X_f$. This leads to

\begin{equation*}
    \chi^{\textit{cw}}(P_0): X_e \xrightarrow{1 \otimes \mathscr{P}(P_0)} X_e \otimes X_f^{\ast} \otimes X_f \xrightarrow{\mathscr{E}(P_0) \otimes \mathbf{1}} X_f 
\end{equation*}

which is explicitly given by:

\begin{equation*}
    v \xrightarrow{1 \otimes \mathscr{P}(P_0)} v \otimes (a_f^{\ast} \otimes a_f + b_f^{\ast} \otimes b_f + c_f^{\ast} \otimes c_f) \xrightarrow{\mathscr{E}(P_0) \otimes \mathbf{1}} a_f^{\ast}(v)a_f + b_f^{\ast}(v)b_f + c_f^{\ast}(v)c_f 
\end{equation*}

Therefore the map

\[ \chi^{\textit{cw}}(P_0) :
  \begin{cases}
    a \mapsto a\\
    b \mapsto b\\
    c \mapsto c
  \end{cases}
\]

Hence $\chi^{\textit{cw}}(P_0) = \mathbf{1}$.

Next, we consider $P_{\gamma}: U \to V$ shown below:

\begin{equation} \label{Pattern-2}
    \centering
    \includegraphics[scale=0.5]{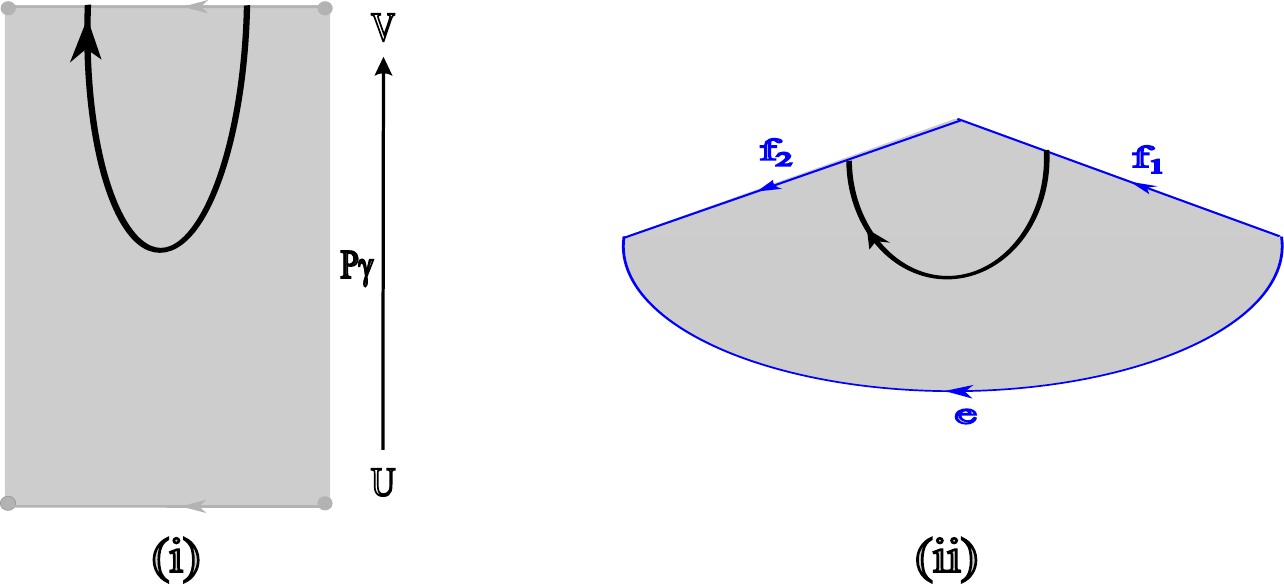}
\end{equation}

Again, we use the cell-decomposition (ii) to calculate $\chi^{\textit{cw}}(P_{\gamma})$. In this case we have $\chi^{\textit{cw}}(U) = \mathbb{C}_e$ and $ \chi^{\textit{cw}}(V) = X_{f_2} \otimes X_{f_1}^{\ast}$, which gives $\chi^{\textit{cw}}(P_{\gamma}): \mathbb{C} \to X \otimes X^{\ast}$. To determine this map, the following data is needed:

\begin{center}
\begin{tabular}{ |c|c|c|c|}
\hline
 $(P_{\gamma}, e, \mathfrak{O})$ & $(P_{\gamma}, e, -)$ & $(P_{\gamma}, f_1, +)$ & $(P_{\gamma}, f_2, +)$ \\ 
\hline
 $Q_{(P_{\gamma}, e, \mathfrak{O})}$ & $\mathbb{C}_e$ & $X_{f_1}$ & $X_{f_2}^{\ast}$ \\ 
\hline 
\end{tabular}
\end{center}

From this we get $Q(P_{\gamma}) = Q_{(P_{\gamma}, f_1, +)} \otimes Q_{(P_{\gamma}, f_2, +)}$ and $\mathscr{P(P_{\gamma})} = \mathscr{P}_{f_1} \otimes \mathscr{P}_{f_2}$ with

$$\begin{matrix}
    \mathscr{P}_{f_1}: & \mathbb{C} \to X_{f_1} \otimes X_{f_1}^{\ast} & \mathscr{P}_{f_2} & : \mathbb{C} \to X_{f_2}^{\ast} \otimes X_{f_2}
\end{matrix}$$

Therefore,

\begin{equation*}
    \chi^{\textit{cw}}(P_{\gamma}): \mathbb{C}_e \xrightarrow{1 \otimes \mathscr{P}(P_{\gamma})} \mathbb{C}_e \otimes X_{f_1} \otimes X_{f_2}^{\ast} \otimes X_{f_2} \otimes X_{f_1}^{\ast} \xrightarrow{\mathscr{E}(P_0) \otimes \mathbf{1}} X_{f_2} \otimes X_{f_1}^{\ast} 
\end{equation*}

among which $\mathbf{1}\otimes \mathscr{P}(P_{\gamma}):$

\begin{equation*}
    1_e \mapsto 1_e \otimes (a_{f_2}^{\ast} \otimes a_{f_2} + b_{f_2}^{\ast} \otimes b_{f_2} + c_{f_2}^{\ast} \otimes c_{f_2}) \otimes (a_{f_1}^{\ast} \otimes a_{f_1} + b_{f_1}^{\ast} \otimes b_{f_1} + c_{f_1}^{\ast} \otimes c_{f_1}) 
\end{equation*}

Thus, the image of $1 \in \mathbb{C}_e$ under $\mathbf{1}\otimes \mathscr{P}(P_{\gamma})$ equals
$$\begin{matrix}

    1_{e} \otimes a_{f_2}^{\ast} \otimes a_{f_2} \otimes a_{f_1} \otimes a_{f_1}^{\ast} + 1_{e} \otimes a_{f_2}^{\ast} \otimes a_{f_2} \otimes b_{f_1} \otimes b_{f_1}^{\ast} + 1_{e} \otimes a_{f_2}^{\ast} \otimes a_{f_2} \otimes c_{f_1} \otimes c_{f_1}^{\ast}\\
    +1_{e} \otimes b_{f_2}^{\ast} \otimes b_{f_2} \otimes a_{f_1} \otimes a_{f_1}^{\ast} + 1_{e} \otimes b_{f_2}^{\ast} \otimes b_{f_2} \otimes b_{f_1} \otimes b_{f_1}^{\ast} + 1_{e} \otimes b_{f_2}^{\ast} \otimes b_{f_2} \otimes c_{f_1} \otimes c_{f_1}^{\ast}\\
    +1_{e} \otimes c_{f_2}^{\ast} \otimes c_{f_2} \otimes a_{f_1} \otimes a_{f_1}^{\ast} + 1_{e} \otimes c_{f_2}^{\ast} \otimes c_{f_2} \otimes b_{f_1} \otimes b_{f_1}^{\ast} + 1_{e} \otimes c_{f_2}^{\ast} \otimes c_{f_2} \otimes c_{f_1} \otimes c_{f_1}^{\ast}    
    
\end{matrix}$$

The action of $\mathscr{E}(P_{\gamma}) \otimes \mathbf{1}$ on these terms is given by

$$\begin{matrix}

    a_{f_2}^{\ast}(1_e a_{f_1}) a_{f_2} \otimes a_{f_1}^{\ast} & + & a_{f_2}^{\ast}(1_e b_{f_1}) a_{f_2} \otimes b_{f_1}^{\ast} & + & a_{f_2}^{\ast}(1_e c_{f_1}) a_{f_2} \otimes c_{f_1}^{\ast}\\ + b_{f_2}^{\ast}(1_e a_{f_1}) a_{f_2} \otimes a_{f_1}^{\ast} & + & b_{f_2}^{\ast}(1_e b_{f_1}) b_{f_2} \otimes b_{f_1}^{\ast} & + & b_{f_2}^{\ast}(1_e c_{f_1}) b_{f_2} \otimes c_{f_1}^{\ast} \\ +c_{f_2}^{\ast}(1_e a_{f_1}) c_{f_2} \otimes a_{f_1}^{\ast} & + & c_{f_2}^{\ast}(1_e b_{f_1}) c_{f_2} \otimes b_{f_1}^{\ast} & + & c_{f_2}^{\ast}(1_e c_{f_1}) a_{f_2} \otimes c_{f_1}^{\ast}
    
\end{matrix}$$

Therefore, the map is

\begin{equation} \label{basic-gon-color-1}
    \chi^{\textit{cw}}(P_{\gamma}) : 1 \mapsto a \otimes a^{\ast} + b \otimes b^{\ast} + c \otimes c^{\ast} ,
\end{equation}

which is the co-evaluation map. 

Now, we turn to the map $P_{\mu}: U \to V$ shown below:

\begin{equation} \label{Pattern-3}
    \centering
    \includegraphics[scale=0.5]{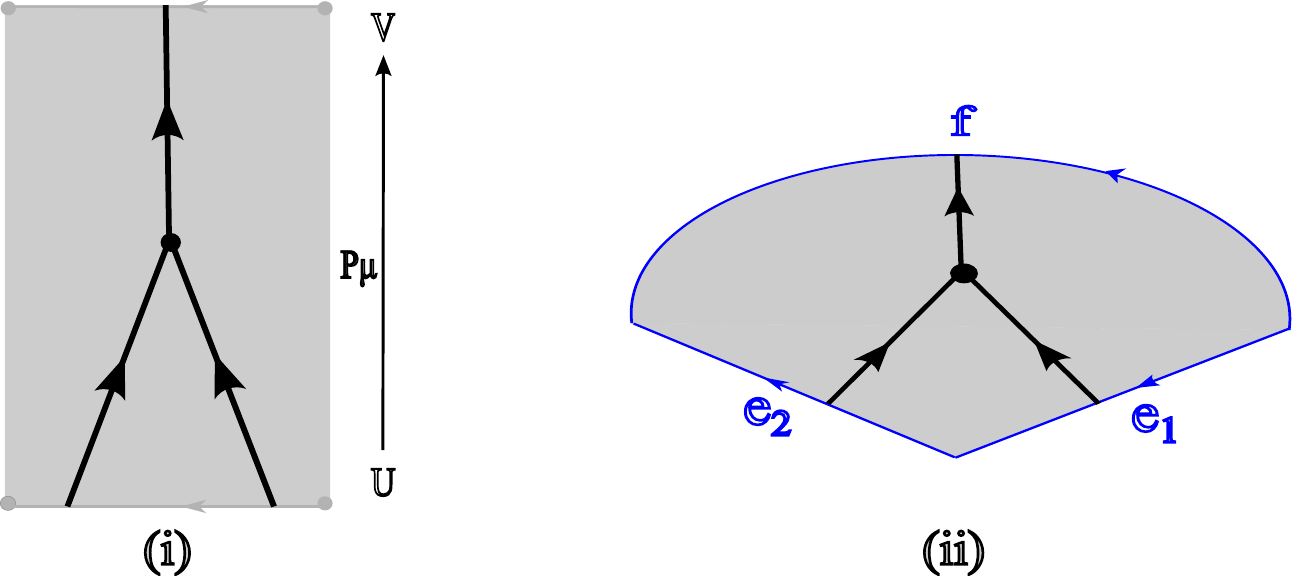}
\end{equation}

Like earlier, we use the cell decomposition as in \ref{Pattern-3}, Diagram (ii), to calculate $\chi^{\textit{cw}}(P_{\mu})$. In this case,

$$\begin{matrix}

    \chi^{\textit{cw}}(U) &= R_{e_1} \otimes R_{e_2} & & , & & \chi^{\textit{cw}}(V) &= R_{f}\\
          & = X_{e_1} \otimes X_{e_2} & &   & & &= X_f
          
\end{matrix}$$
    
Thus, we get the map $$ \chi^{\textit{cw}}(P_{\mu}): X_{e_1} \otimes X_{e_2} \to X_f .$$ The following data is needed to compute this map explicitly:

\begin{center}
\begin{tabular}{ |c|c|c|c|}
\hline
 $(P_{\mu}, e, \mathfrak{O})$ & $(P_{\mu}, e_1, -)$ & $(P_{\mu}, e_2, -)$ & $(P_{\mu}, f, +)$ \\ 
\hline
 $Q_{(P_{\mu}, e, \mathfrak{O})}$ & $X_{e_1}$ & $X_{e_2}$ & $X_{f}^{\ast}$ \\ 
\hline 
\end{tabular}
\end{center}

From this we get $Q(P_{\mu}) = Q_{(P_{\mu}, f, +)}$ and $\mathscr{P(P_{\mu})} = \mathscr{P}_{f}$ defined as
$\mathscr{P}_{f} :  \mathbb{C} \to X_f^{\ast} \otimes X_f$

we get

\begin{equation*}
\chi^{\textit{cw}}(P_{\mu}) : X_{e_1} \otimes X_{e_2} \xrightarrow{1 \otimes \mathscr{P}(P_{\mu})} X_{e_1} \otimes X_{e_2} \otimes X_{f}^{\ast} \otimes X_{f} \xrightarrow{\mathscr{E}(P_{\mu}) \otimes \mathbf{1}} X_{f} , 
\end{equation*}

which, for $x \in X_{e_1}$ and $y \in X_{e_2}$, is given by
\begin{equation} \label{eq-mu}
    \begin{matrix}

    x \otimes y & \mapsto & x \otimes y \otimes (a_{f}^{\ast} \otimes a_{f} + b_{f}^{\ast} \otimes b_{f} + c_{f}^{\ast} \otimes c_{f})\\
    & = & x \otimes y \otimes a_{f}^{\ast} \otimes a_{f} + x\otimes y \otimes b_f^{\ast} \otimes b_{f} + x \otimes y \otimes c_{f}^{\ast} \otimes c_{f} .

   \end{matrix}
\end{equation}

Recall the map $\mu: X \otimes X \otimes X \to \mathbb{C}$ from \cref{Tait_coloring}. The action of ($\mathscr{E}(P_{\mu}) \otimes \mathbf{1} )$ on ~ \ref{eq-mu} is given by

\begin{multline}
    x \otimes y \otimes a_{f}^{\ast} \otimes a_{f} + x\otimes y \otimes b_f^{\ast} \otimes b_{f} + x \otimes y \otimes c_{f}^{\ast} \otimes c_{f} \\
    \mapsto \mu(a_f^{\ast} \otimes y \otimes x)a_f + \mu(b_f^{\ast} \otimes y \otimes x)b_f + \mu(c_f^{\ast} \otimes y \otimes x)c_f
\end{multline}
    
Therefore the action of the map $\chi^{cw}(P_{\mu})$ on the basis elements are given by:

\begin{equation}
\chi^{cw}(P_{\mu}) :
  \begin{cases}
     a \otimes a, b\otimes b, c\otimes c  & \mapsto 0 \\
     a \otimes b, b \otimes a  & \mapsto  c \\
     c \otimes a, a \otimes c  & \mapsto  b \\
     b \otimes c, c \otimes b  & \mapsto  a
  \end{cases}    
\end{equation}

Finally, we want to know the action of the TFT $\chi^{cw}$ on $(P_{\beta}): U \to V$ shown below:

\begin{equation} \label{Pattern-4}
    \centering
    \includegraphics[scale=0.5]{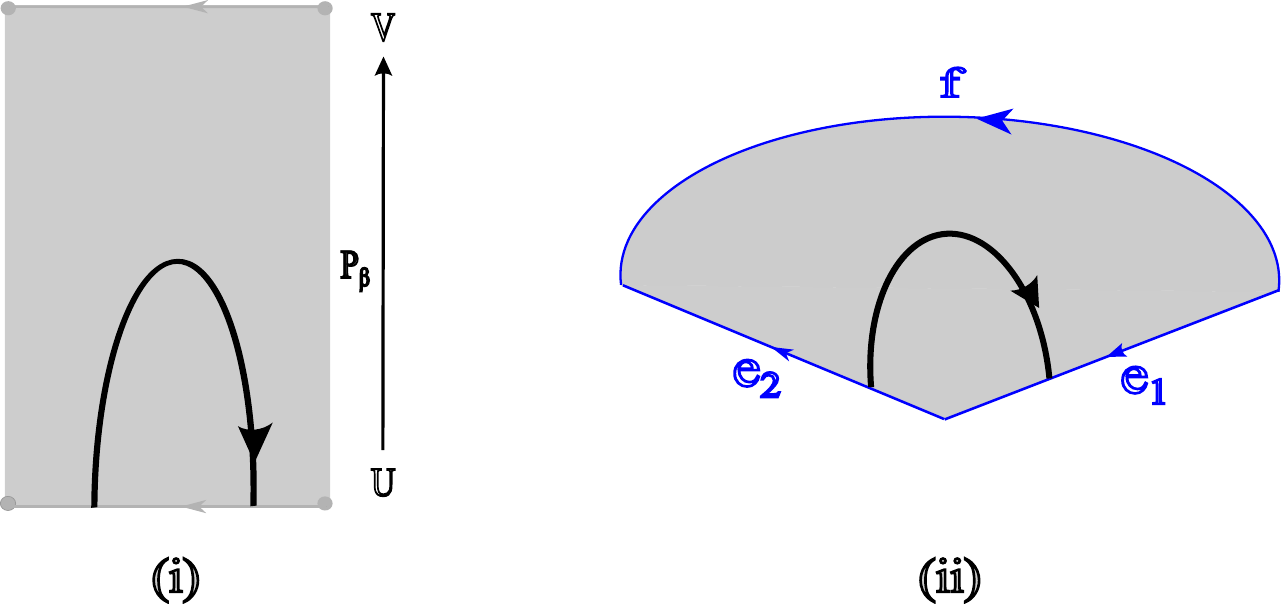}
\end{equation}

We see from ~\ref{Pattern-4}, Diagram (ii), that $C_1(U) = \{e_1, e_2\}$ and $C_1(V) = \{f\}$, which gives $\chi^{cw}(U) \coloneqq R_{e_1} \otimes R_{e_2} = X_{e_1}^{\ast} \otimes X_{e_2}$ and similarly $\chi^{cw}(V) \coloneqq R_f = \mathbb{C}_{f}$. The data for $Q_{(P_{\beta})}$ is given by the table:

\begin{center}
\begin{tabular}{ |c|c|c|c|}
\hline
 $(P_{\beta}, e, \mathfrak{O})$ & $(P_{\beta}, e_1, -)$ & $(P_{\beta}, e_2, -)$ & $(P_{\beta}, f, +)$ \\ 
\hline
 $Q_{(P_{\beta}, e, \mathfrak{O})}$ & $X_{e_1}^{\ast}$ & $X_{e_2}$ & $\mathbb{C}_f$ \\ 
\hline 
\end{tabular}
\end{center}

and $\mathscr{P}(P_{\beta}) = \mathscr{P}_f$, where $\mathscr{P}_f : \mathbb{C} \to \mathbb{C}_f \otimes \mathbb{C}_f$. Therefore, we get the composition:

\begin{equation*}
\chi^{\textit{cw}}(P_{\beta}) : X_{e_1} \otimes X_{e_2}^{\ast} \xrightarrow{1 \otimes \mathscr{P}(P_{\beta})} X_{e_1} \otimes X_{e_2}^{\ast} \otimes \mathbb{C}_{f} \otimes \mathbb{C}_{f} \xrightarrow{\mathscr{E}(P_{\mu}) \otimes \mathbf{1}} \mathbb{C}_{f} 
\end{equation*}

given explicitly by:

\begin{equation}
    x \otimes y^{\ast} \xrightarrow{\mathbf{1} \otimes \mathscr{P}(P_{\beta})} x\otimes y^{\ast} \otimes 1 \otimes 1 \xrightarrow{\mathscr{E}(P_{\beta}) \otimes \mathbf{1}} y^{\ast}(x)
\end{equation}

Thus $\chi^{cw}(P_{\beta})$ is given on the basis elements by

\begin{equation}
\chi^{cw}(P_{\mu}) :
  \begin{cases}
     a \otimes a^{\ast}, b\otimes b^{\ast}, c\otimes c^{\ast}  & \mapsto 1 \\
     
     0  & \text{otherwise} ,
  \end{cases}    
\end{equation}

which is nothing but the evaluation map.

We will return to $\chi^{cw}$ in the next section and interpret the computations of this section in terms of graph coloring.

\section{Applications to graph coloring} \label{sec:graph-coloring}

We met the category $\textit{Bord}_2^{\text{def, cw}}(\mathcal{P}_{K_4})$ in \cref{exam:coloring}, and mentioned that an element of the set $\textbf{Mor}(\textit{Bord}_2^{\text{def, cw}}(\mathcal{P}_{K_4})$ is precisely a pair $(\Sigma, \Gamma)$, where $\Gamma$ is a trivalent, $3$-edge colorable, graph embedded in $\Sigma$. In this section, we give the proper definition of $3$-edge coloring, a $3$-edge colorable graph, and construct a trivial surrounding theory, $$\chi^{cw}:\textit{Bord}_2^{\text{def, cw}}(\mathcal{D}_{+}^{\mathbf{3}})) \to \text{Vect}_F(\mathbb{C}), $$ which counts the number of Tait-coloring of a trivalent planar graph.

\subsection{Admissible edge coloring: a new interpretation.} \label{sec:coloring_process}

Next, we use the tools we have developed so far to redefine the concept of edge coloring of a graph; not necessarily admissible. However, the coloring that we encounter in the subsequent sections is admissible.

Recall the forgetful functor $$ \Pi^{cw}: \textit{Bord}_2^{\textit{def,cw}}(\mathcal{P_G}) \to \textit{Bord}_2^{\textit{def,cw}}(\mathcal{D}^{\aleph})$$ from \cref{defn:bleach-cat}. In the case, when $G = \mathbb{K}_4$, and $\mathcal{P_G}$ is as in \cref{exam:coloring}, the target category is much smaller, and we have, by an abuse of notation, $$\Pi^{cw}: \textit{Bord}_2^{\textit{def,cw}}(\mathcal{P}_{K_4})) \to \textit{Bord}_2^{\textit{def,cw}}(\mathcal{D}_{+}^{\mathbf{3}}))$$ mapping into the subcategory $\textit{Bord}_2^{\textit{def,cw}}(\mathcal{D}_{+}^{\mathbf{3}})$. First, to set up the main results of this manuscript, we give a new definition of coloring by $\mathbb{K}_4$ in terms of the category $\textit{Bord}_2^{\textit{def,cw}}(\mathbb{K}_4)$ and the map $\pi^{cw}$. 

\begin{defi} \label{obj:coloring}

For an object $O$ in $\textit{Bord}_2^{\textit{def,cw}}(\mathcal{D}_{+}^{\mathbf{3}})$, a coloring of $O$ is an object $\hat{O}$ in $\textit{Bord}_2^{\textit{def,cw}}(\mathcal{P}_{K_4})$ such that $\Pi^{cw}(\hat{O}) = O$. The collection of such $\hat{O}$ will be denoted $\Pi^{-1}(O).$
\end{defi}

\begin{defi} \label{defn:coloring}

    Given $(\Sigma, \Gamma)$ in $\text{Mor}(\textit{Bord}_2^{\textit{def,cw}}(\mathcal{D}_{+}^{\mathbf{3}}))$, let $$\pi^{-1}(\Sigma, \Gamma) \coloneqq \{ (\Sigma_i, \Gamma_i, \mathbb{K}_4) \in \text{Mor}(\textit{Bord}_2^{\textit{def,cw}}(\mathcal{P}_{K_4})) \mid \pi^{cw}(\Sigma_i, \Gamma_i, \mathbb{K}_4) = (\Sigma, \Gamma) \hspace{1mm} \forall \hspace{1mm} i \}.$$ A coloring is a map $$ s:  \text{Mor}(\textit{Bord}_2^{\textit{def,cw}}(\mathcal{D}_{+}^{\mathbf{3}})  \to\text{Mor}(\textit{Bord}_2^{\textit{def,cw}}(\mathcal{P}_{K_4}))$$ in the opposite direction of map $\pi^{cw}$ in ~\ref{bleach-reduced} such that $(\pi^{cw} \diamond s) (\Sigma, \Gamma) = (\Sigma, \Gamma)$ if $\pi^{-1}(\Sigma, \Gamma)$ is non-empty, and $s(\Sigma, \Gamma) = \emptyset_2$ (the empty morphism) if $\pi^{-1}(\Sigma, \Gamma)$ is empty. Here, $\diamond$ stands for the composition in the category of sets (also see \cref{sec:additional}.)

    The value of $s$ at a surface with defects $(\Sigma, \Gamma)$ is called a coloring of $(\Sigma, \Gamma)$. A trivalent graph $\Gamma$ embedded in a surface $\Sigma$ is said to be $3$-edge colorable if $s(\Sigma, \Gamma) \neq \emptyset_2$, or equivalently $\pi^{-1}(\Sigma, \Gamma)$ is non-empty.
\end{defi}

Where we have intentionally dropped the notation \textit{cw} when writing $\Pi^{-1}$ or $\pi^{-1}$ to avoid the cumbersomeness of notations.

\begin{rema} \label{rema:section}
    
It follows from the definition of $\Pi^{cw}$ that if $\hat{\Sigma} \in  \text{Mor}(\textit{Bord}_2^{\textit{def,cw}}(\mathcal{P}_{K_4}))$ is such that $\pi^{cw}(\hat{\Sigma}) = (\Sigma, \Gamma)$, then they have identical (isotopic) underlying stratified space, namely given as in \cref{exam-deco} (6). Thus each individual element in $\pi^{-1}(\Sigma, \Gamma)$ is a copy of $(\Sigma, \Gamma)$ as a stratified space. Writing $\pi^{-1}(\Sigma, \Gamma) =\sqcup_{i}(\Sigma_i, \Gamma_i, \mathbb{K}_4)$, we can define the map $\mathfrak{p}: \pi^{-1}((\Sigma, \Gamma)) \to (\Sigma, \Gamma)$ by $\mathfrak{p}(\hat{\Sigma}) = \pi^{cw}(\hat{\Sigma})$, which satisfies $\mathfrak{p} \diamond s = \mathbf{1}$, thus $s$ can be viewed as a section of $\mathfrak{p}: \pi^{-1}((\Sigma, \Gamma)) \to (\Sigma, \Gamma)$.

\end{rema}

\begin{exam} \label{exam:coloring_exam-1}
    We take $O_1 = \emptyset, O_2$ as a single circle $O$ with two $0$-defects, and $(\Sigma, \Gamma) : \phi \to O$ as shown below in \cref{fig:coloring_exam-1} (middle). Note that orientation on the $1$-strata does not matter, consequently, both $0$-defects in the boundary circle with defects, $O_2$, have gotten the same label $e$. The following figure shows some maps $s: (\Sigma, \Gamma) \to \pi^{-1}((\Sigma, \Gamma))$ such that $\mathfrak{p} \diamond s = \mathbf{1}$ 
    \begin{figure}[ht]
        \centering
        \includegraphics[width=0.9\linewidth]{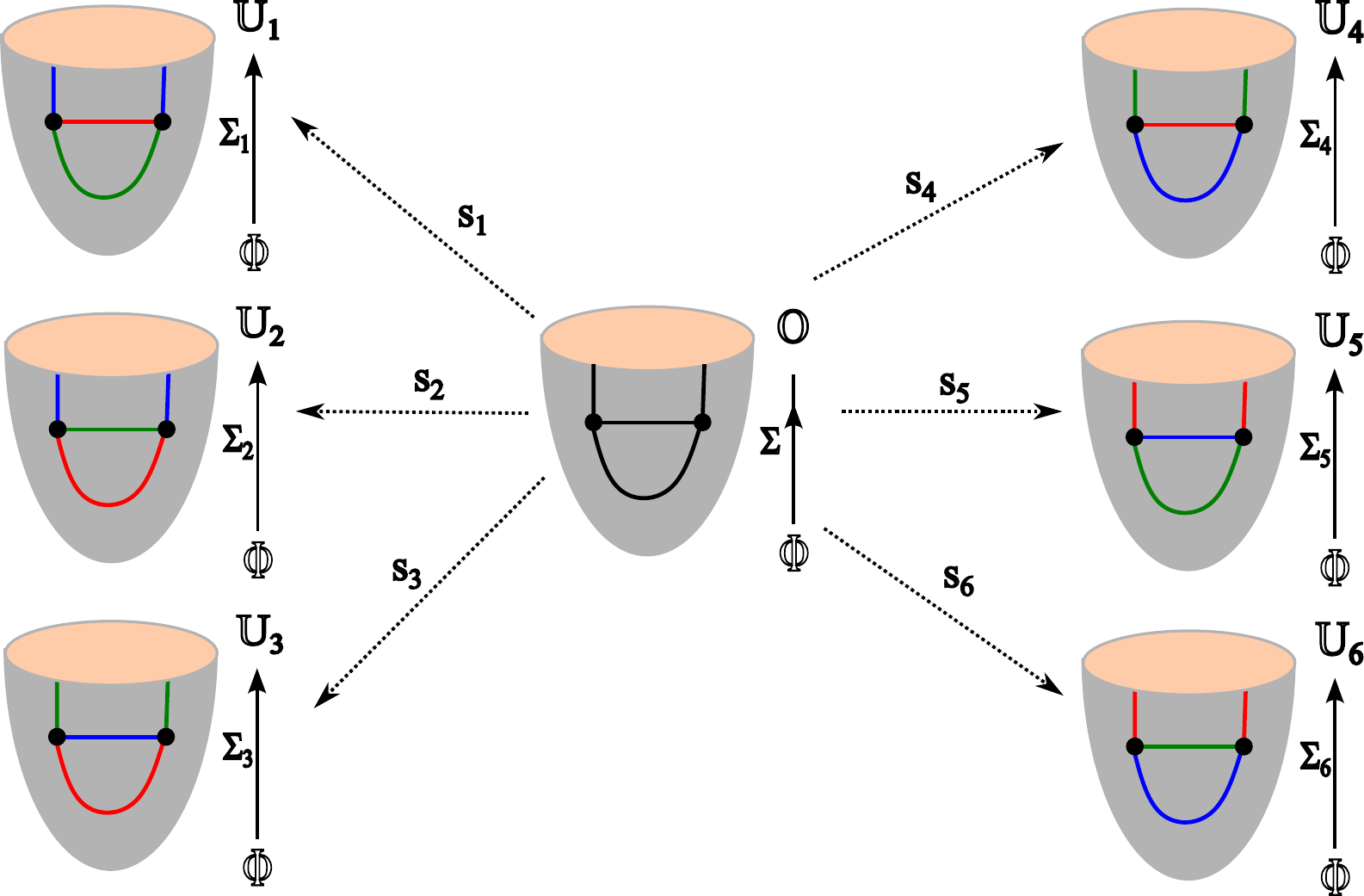}
        \caption{The set $\pi^{-1}((\Sigma, \Gamma))$ from \cref{exam:coloring_exam-1}}
        \label{fig:coloring_exam-1}
    \end{figure}

\end{exam}

Note, that an object with two $0$-defects labeled by two of $a, b$ or $c$, both different, also lies in the collection $\Pi^{-1}(O)$, but they do not appear as an out-boundary of a morphism in the Figure ~\ref{fig:coloring_exam-1}. This is not a coincidence but a consequence of \cref{big-1} below.

Using \cref{defn:coloring}, the question of whether the functor $\Pi^{cw}$ is full, can be rephrased as whether every trivalent graph, embedded in some surface, is $3$-edge colorable. \cref{big-1} gives a necessary condition on a planar trivalent graph to be $3$-edge colorable using the group structure of $\mathbb{K}_4$, and \cref{cor:big-1} answers the question about the fullness of $\Pi^{cw}$ negatively.

Next, we formulate the idea of the \textit{coloring process} for a given trivalent graph $\Gamma$, embedded in a surface $\Sigma$ by considering the pair $(\Sigma, \Gamma)$ in the set $\text{Mor}(\textit{Bord}_2^{\textit{def,cw}}(\mathcal{D}_{+}^{\mathbf{3}}))$. The first definition needed is: 

\begin{defi} \label{def:coloring}

    For a surface with defect  $(\Sigma, \Gamma)$ in $\text{Mor}(\textit{Bord}_2^{\textit{def,cw}}(\mathcal{D}_{+}^{\mathbf{3}}))$, let $P \in C_2(\Sigma)$ be a basic-gon considered as a morphism $(P, \Gamma_P): \emptyset \to \partial P$ , a choice of a $3$-\textit{edge-coloring} localised at $P$ is the value $s_P \coloneqq s((P, \Gamma_P))$ under the function $s$ of \cref{defn:coloring}. Alternatively, it is a choice of a section $s_{P}$ of $ \mathfrak{p}: \pi^{-1}(P) \to P$ in the sense of \cref{rema:section}.
    
\end{defi}

\begin{conv}

    Now onwards, we will suppress $\Gamma_P$ from the notation $(P, \Gamma_P)$ for convenience, and use just $P$ for $(P, \Gamma_P) \in \text{Mor}(\textit{Bord}_2^{\textit{def,cw}}(\mathcal{D}_{+}^{\mathbf{3}}))$.  
    
\end{conv}

\begin{rema}

    It follows from the definition of $\Pi^{cw}$ and the identity $\pi^{cw} \diamond s_{P} = P$ that the underlying stratified spaces of $s(P)$ and $P$ are isotopic. Thus $s(P)$ is isomorphic to one of the defect disks in $\textit{Bord}_2^{\textit{def,cw}}(\mathcal{P}_{K_4})$ 
    as surfaces with defects. The color assigned to the graph $\Gamma_P$ in $P$ is the label, in the defect condition $\mathcal{P}_{K_4}$ (cf \cref{exam:coloring}), that edges of $\Gamma$ get assigned under this $s$.
    
\end{rema}

The following result is a straightforward application of \cref{defn:coloring} and \cref{defn:bleach-cat}:

\begin{prop}

    For basic-gons $Q_1, Q_2$ in $(\Sigma, \Gamma)$ in $\text{Mor}(\textit{Bord}_2^{\textit{def,cw}}(\mathcal{P}_{K_4}))$, the following holds:
    \begin{enumerate}
        \item $\mathfrak{p}(Q_1 \otimes Q_2) = \mathfrak{p}(Q_1) \otimes \mathfrak{p}(Q_2)$, and
        \item $\mathfrak{p}(Q_1 \circ Q_2) = \mathfrak{p}(Q_1) \circ \mathfrak{p}(Q_2)$.
    \end{enumerate}

\end{prop}

By definition, a surface $(\Sigma, \Gamma) \in \text{Mor}(\textit{Bord}_2^{\textit{def,cw}}(\mathcal{D}_{+}^{\mathbf{3}}))$ comes equipped with a PLCW decomposition into cells $C_0(\Sigma), C_1(\Sigma) $ and $ C_2(\Sigma)$ such that each $2$-cell is isomorphic (as a surface with defects) to one of basic-gons as in \cref{basic_gons} (ii), (iii) and (iv).

\begin{conv} \label{fusion_basic-gons}
Referring to \cref{basic_gons} and the caption there, make the following conventions:
    \begin{enumerate}
        \item For two basic-gons $P_i, P_j$ of type (ii) or (iii), we denote by $P_i \otimes P_j$ the gluing of $P_i$ and $P_j$ along $P_{ij} \coloneqq P_i \cap P_j \in C_1(\Sigma)$. Although this seems like an abuse of notation, it corresponds to the fusion or horizontal composition, see ~\ref{fusion}, when considering the defect data of $P_i$ and $P_j$.
        \item For a basic-gon $P_{\mu}$ of type (iii) we use the vertical composition $P_{\mu} \circ (P_{i_1} \otimes \dots \otimes P_{i_k}) $ to denote gluing along the $1$-cell formed by the intersection of $P_{\mu}, P_{i_1}, \dots, P_{i_k}$. Again, this corresponds to the vertical composition of the underlying defect data.
    \end{enumerate}
\end{conv}

We are ready to define a coloring process:

\begin{defi} \label{coloring_process}
Given an un-directed trivalent graph $\Gamma$, embedded in a surface $\Sigma$, a \textit{coloring process} is the following data assigned to the surface with defects \linebreak $(\Sigma, \Gamma) \in \text{Mor}(\textit{Bord}_2^{\textit{def,cw}}(\mathcal{D}_{+}^{\mathbf{3}}))$ :
\begin{enumerate}
    
    \item A coloring $s_P$ for every $P \in C_2(\Sigma)$, as defined in \cref{def:coloring}, 
    \item A coloring $s_{(P_i \otimes P_j)}$ for every fused cells $P_i \otimes \dots \otimes P_j$ given by $s_{(P_i \otimes P_j)} \coloneqq s_{P_i} \otimes s_{P_j}$, and
    \item A coloring $s_{(P_{\mu} \circ P_{\nu})} $ for each vertical composition $P_{\mu} \circ P_{\nu} $ given by $s_{(P_{\mu} \circ P_{\nu})} \coloneqq s_{P_{\mu}} \circ s_{P_{\nu}}$.
    
\end{enumerate}
In short, a coloring process is an assignment of coloring to each $2$-cells and a schema to glue them together to produce a coloring of the entire surface $(\Sigma, \Gamma)$ as \cref{coloring_process}, (2) and (3), facilitate gluing of the coloring of an arbitrary (finite) number of cells by repeated application.
\end{defi}

\begin{rema}
    Compare \cref{coloring_process} with the concept of bundles at the level of $2$-category outlined in \cref{sec:additional}. In particular, to the diagrams in \cref{fig:section_horiz-composition} and \cref{fig:section_vert-composition}.
\end{rema}

\begin{defi} \label{def:coloring-global}
    A surface with defects $(\Sigma, \Gamma) \in \text{Mor}(\textit{Bord}_2^{\textit{def,cw}}(\mathcal{D}_{+}^{\mathbf{3}}))$ admits a $3$-edge coloring or, is $3$-edge colorable if these exist an $s: (\Sigma, \Gamma) \to \pi^{-1}((\Sigma, \Gamma))$ extending all $s_P$ for $P \in C_2(\Sigma)$, i.e. $s\mid_{p} = s_P$, that satisfies the conditions listed in \cref{coloring_process} when restricted to any sub-complex formed by fusing and composing several $2$-cells. 
\end{defi}

 Note that the map $\mathfrak{p}$ coincides with the one defined in \cref{def:coloring} on $2$-cells. Therefore, one can say that the map $s$ in \cref{def:coloring-global} is a global section.  Given a surface $(\Sigma, \Gamma) \in \text{Mor}(\textit{Bord}_2^{\textit{def,cw}}(\mathcal{D}^{\mathbf{3}}))(O_1, O_2)$, local sections, as given by \cref{coloring_process} always exists at every basic-gons. \cref{def:coloring-global} says that the graph $\Gamma$ is $3$-edge colorable if all these local-sections can be patched together to produce a global-section. In that case, such a global section gives a coloring of $\Gamma$.

\begin{theo} \label{big-1}
    Consider the surface with defect $(\mathbb{S}^2, \Gamma, K_4 )$ as an element in the set $\text{Mor}(\textit{Bord}_2^{\textit{def, cw}}(\mathcal{P}_{K_4}))$ Let $\Bar{\mathbb{S}^1}$ be a generic cross-section of $(\mathbb{S}^2, \Gamma, K_4 )$ then the product of defects on $\Bar{\mathbb{S}^1}$ is 1.
\end{theo}

For the rest of this section, we only consider graphs with a single component. Given a graph $\Gamma$, a \textit{bridge} is an edge of $\Gamma$ whose deletion disconnects the graph into two components. (See ~ \cite{bollobas1998modern} for more detail and general, as well as, alternative definitions.) 

Using \cref{def:coloring-global}, we deduce the following famous result from \cref{big-1}, which has been known to people since Tait:

\begin{coro} \label{cor:big-1}
    A planar trivalent graph $\Gamma$ with bridge is not $3$-edge colorable.
\end{coro}

In the language we have developed so far, it means that a pair $(\mathbb{S}^2, \Gamma)$ with $\Gamma$ having a bridge never lies in the image of $\pi^{cw}$. Setting $O_1 = O_2 = \emptyset$ we see that $\pi^{cw}$ is not surjective. Hence $\Pi^{cw}$ is not full.

We prove the corollary first using \cref{big-1}.

\begin{proof}
    If the trivalent graph $\Gamma$ with bridge $e$ is $3$-edge colorable then the edge $e$ gets $a, b$ or $c$ as the color. Since $e$ is a bridge, there exists a generic cross-section of $(\mathbb{S}^2, \Gamma)$ that intersects only $e$. If $\Bar{\mathbb{S}_e^1}$ is this generic cross-section, then it contains a single defect labeled by $a, b$ or $c$, which it inherits from the coloring of $e$. This is a contradiction to ~\ref{big-1}.
\end{proof}

\begin{proof}[Proof of \cref{big-1}]
    Let $\gamma_{xy}$ denote the union of all the edges of the graph $\Gamma$ with color $x$ and $y$. Then all of $\gamma_{ab}, \gamma_{bc}$ and $\gamma_{ac}$ are piecewise linear simple (Jordan) curves embedded in $\mathbb{S}^2$ and thus intersect any generic cross-section even number of times. Let $S_t$ be a generic cross-section and $2n_1, 2n_2$ and $2n_3$ be the number of intersection points of it with $\gamma_{ab}, \gamma_{bc}$ and $\gamma_{ac}$ respectively. Note that it is enough to consider only two of them, say $\gamma_{ab}$ and $\gamma_{bc}$. The contribution from $\gamma_{ab}$ will be of the form $a^kb^{2n_1 - k}$ for some positive integer $k$. The share of $c$ comes from the curve $\gamma_{bc}$ and is equal to $c^{2n_2 - (2n_1 - k)}$. Therefore the product of defects of $S_t$ equals $a^kb^{2n_1 - k}c^{2n_2 - 2n_1 +k}$. This product simplifies to $(ab^{-1})^kc^k$ or $c^{2k}$, which equals $1$.
\end{proof}

\begin{rema} \label{rema:distinct_colors}
    
\cref{big-1} is more general than the classical statement of \cref{cor:big-1}. We return to the comment made below \cref{exam:coloring_exam-1} in connection with it. Now, we see that there can not be a morphism between $\emptyset$ and a single circle labeled with two distinct defects that project to $\hat{\Sigma}$. For, if there were such a morphism, then we could take its involution and glue along the common boundary circle with defects to produce a pair $(\mathbb{S}^2, \Gamma, K_4)$. We see that it contradicts \cref{big-1}.
    
\end{rema}

We conclude this section by deducing the validity of Tait's correspondence.

\begin{coro} \label{coro:Tait}
    Tait's correspondence, as outlined in \cref{sec:intro_Tait}, is a valid procedure.
\end{coro}

\begin{proof}
    After fixing an arbitrary region to begin face coloring, we need to show two things:
    \begin{enumerate}
        \item The face coloring of a region, as described by Tait's correspondence does not depend on the path taken to reach this region, and
        \item each region gets a color.
    \end{enumerate}
    The concept of face-coloring is not defined using \cref{def:coloring} or \cref{coloring_process} (although it can be), therefore, we rely on the classical definition. We use a $G$-set $Y$, with $G = K_4$ as a set of colors available for the face-coloring. To prove (1), assume that the starting region has a color $y \in Y$ and let $\gamma_1$ and $\gamma_2$ be two paths leading to different colors $g_1y$ and $g_2y$. Now, the concatenation of $\gamma_1$ and $\gamma_2$ gives rise to a loop starting and ending at the region colored $y$. This loop is a generic cross-section of $\mathbb{S}^2$ with the product of defects $g_1g_2$, which has to be $1$ by \cref{big-1}. This proves $g_1 = g_2$ (every element is its own inverse in $K_4$, and it is abelian.)
    
    Proof of (2) is a consequence of the fact that $\mathbb{S}^2$ is path-connected.
    
Note that the admissibility of face-coloring, i.e., each face sharing an edge getting a different color, depends on the size of the $K_4$-set $Y$ and the definition of the action. See also the discussion below \cref{4-color}.

\end{proof}

\subsection{Planar trivalent graphs} \label{sec:planar_trivalent_graphs}

Finally, we restrict our attention to un-directed, trivalent, planar graphs. In the language of surface with defects, it is a pair $(\mathbb{S}^2, \Gamma) \in \text{Mor}(\textit{Bord}_2^{\textit{def,cw}}(\mathcal{D}_{+}^{\mathbf{3}}))(\emptyset, \emptyset)$, with admissible decomposition as discussed in \cref{exam-deco} (6). The goal of this section is to address the question of the coloring of such a graph. In other words, whether a given surface with defects $(\mathbb{S}^2, \Gamma)$ lies in the image of $\pi^{cw}$ in \cref{bleach}. We saw in \cref{cor:big-1} that it is not always possible to find a global section $s: (\Sigma, \Gamma) \to \pi^{-1}((\Sigma, \Gamma))$ such that $\mathfrak{p} \circ s = \mathbf{1}$. Note that the cardinality of $\pi^{-1}((\mathbb{S}^2, \Gamma))$ is precisely the number of Tait-coloring of the planar graph $\Gamma$. By definition of $s$, it is also the total number of such global sections $s$.

We begin with an example demonstrating the coloring process for planar trivalent graphs:

\begin{exam} \label{exam:dumbbell}
    We see that for the dumbbell graph below, there are three choices of sections for each of $P_1, P_2$ and $P_4$ and six choices for $P_3$ but no such choice of $s_{P_1}, s_{P_2}, s_{P_3}$ and $s_{P_4}$ extends to a global-section $s$ as this will contradict \cref{big-1}.
    \begin{equation} \label{dumbbell}
        \centering
        \includegraphics[scale=1]{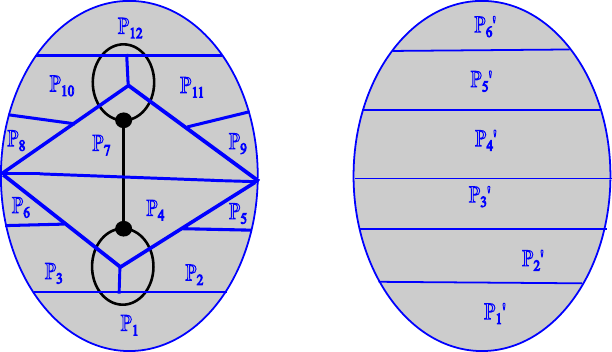}
    \end{equation}
    
\end{exam}

\begin{figure}[ht]
    \centering
    \includegraphics[scale=0.5]{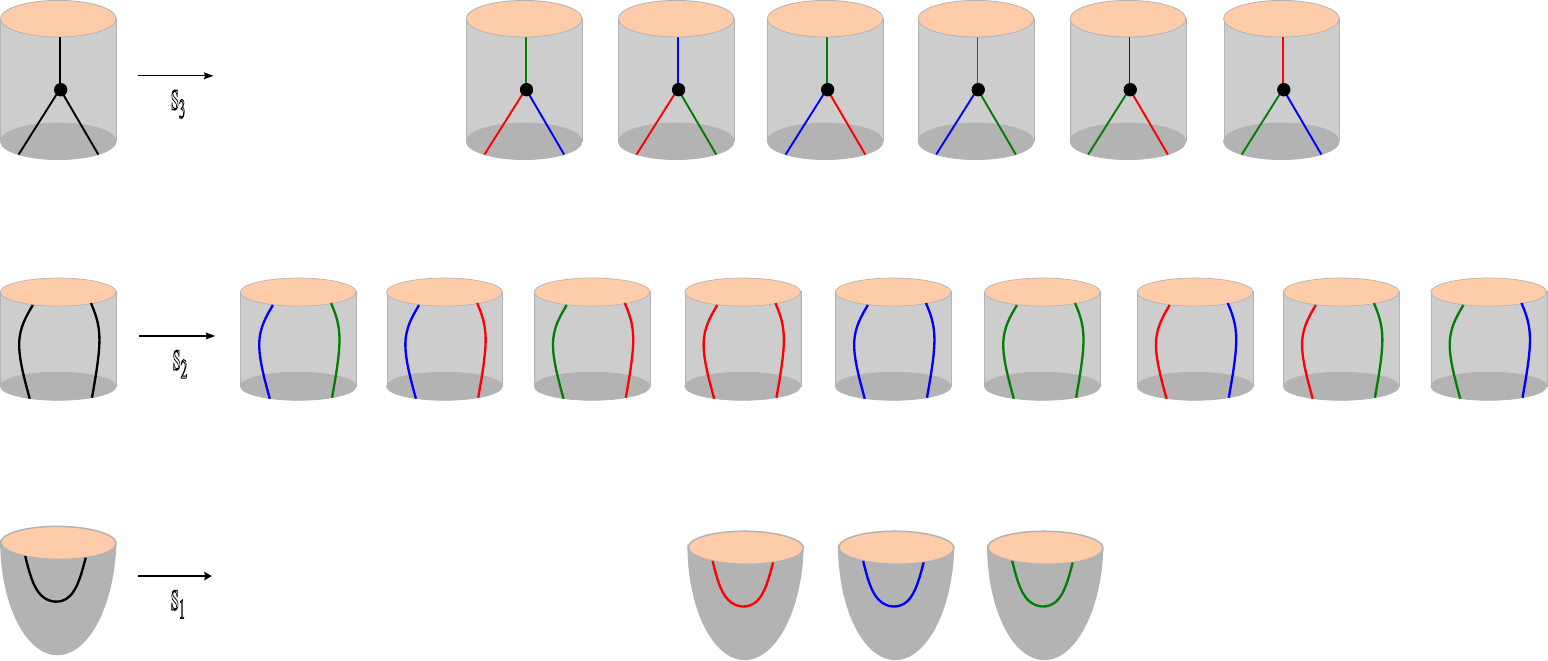}
    \caption{shows the existence of local sections on the southern hemisphere made by fusing $P_1$ and $P'_1$ (bottom), the cylinder made by fusing $P_2, P_3$ and $P'_2$(middle), and finally cylinder made by fusing $P_4, P_5, P_6$ and $P'_3$(top). They do not glue in any manner to produce a section on the southern hemisphere of ~ \ref{dumbbell} that restricts to individual sections. }
    \label{fig:coloring-dumbbell}
\end{figure}
\cref{fig:coloring-dumbbell} demonstrate the coloring process for the dumbbell graph in ~\ref{dumbbell}: extending the sections to $P_3 \otimes P_2 \otimes P'_2$ under horizontal composition by the rule $s_{P_3 \otimes P_2 \otimes P'_2} = s_{P_3} \otimes s_{P_2} \otimes s_{P'_2}$, and to the vertical composition $(P_3 \otimes P_2 \otimes P'_2) \circ (P_1 \otimes P'_1)$ by $(s_{P_3} \otimes s_{P_2} \otimes s_{P'_2}) \circ (s_{P_1} \otimes s_{P'_1})$. Note that for no choice of $s_{P_1}, \dots, s_{P_6}$, these individual sections can be extended to $(P_6 \otimes P_4 \otimes P_5 \otimes P'_3) \circ (P_3 \otimes P_2 \otimes P'_2) \circ (P_1 \otimes P'_1)$.

Note, for a cylinder $C$ in $\text{Mor}(\textit{Bord}_2^{\textit{def,cw}}(\mathcal{D}_{+}^{\mathbf{3}}))(O_1, O_2)$, if $s(C)$ exists then it is a cylinder $\hat{C}$ in the category $\text{Mor}(\textit{Bord}_2^{\textit{def,cw}}(\mathcal{P}_{K_4}))(U_1, U_2)$ for some circle with defects $U_1$ and $U_2$ with the property that $\Pi^{cw}(U_1) = O_1$ and $\Pi^{cw}(U_2) = O_2$. Therefore, if $C_1$ and $C_2$ are two such cylinders in $\text{Mor}(\textit{Bord}_2^{\textit{def,cw}}(\mathcal{D}_{+}^{\mathbf{3}}))$ such that $C_2 \circ C_1$ is defined then $s_{C_1}$ and $s_{C_2}$ extends to a section $s_{C_2 \circ C_1}$ if and only if the composition $s_{C_2} \circ s_{C_1}$ exists in $\textit{Bord}_2^{\textit{def,cw}}(\mathcal{P}_{K_4})$, in which case a section $s_{C_2 \circ C_1}$ is given by the composition $s_{C_2} \circ s_{C_1}$, as suggested by the coloring process. So, we see that it is the vertical composition that dictates the gluing of local sections.

Next, recall the trivial surrounding theory $\chi^{cw}: \textit{Bord}_2^{\textit{def,cw}}(\mathcal{D}^{\mathbf{3}}) \to \text{Vect}_F(\mathbb{C})$ from \cref{chi-cw}. Under the isomorphism $X \cong X^{\ast}$, it is independent of the orientation on the edges of the graph $\Gamma$ and thus we can talk about the correlator of a surface with defects in $\textit{Bord}_2^{\textit{def,cw}}(\mathcal{D}_{+}^{\mathbf{3}})$ by choosing an arbitrary orientation of $1$-strata. Thus we formulate the main result:

\begin{theo} \label{main-2}
    Let $\Gamma$ be a trivalent graph embedded in $\mathbb{S}^2$. Consider the surface with defect $(\mathbb{S}^2, {\Gamma})$ in $\text{Mor}(\textit{Bord}_2^{\textit{def,cw}}(\mathcal{D}_{+}^{\mathbf{3}}))(\emptyset, \emptyset)$. The action of the functor $\chi^{cw}$ on $(\mathbb{S}^2, \Gamma)$ is the assignment
    \begin{equation}
        \begin{split}
            \chi^{cw}(\mathbb{S}^2, \Gamma) &: \mathbb{C} \longrightarrow \mathbb{C}\\
            & \lambda \mapsto \#\text{Tait}(\Gamma) \lambda
        \end{split}
    \end{equation}
    
  In other words the number $\chi^{cw}(\mathbb{S}^2, \Gamma)(1)$ is the number of Tait-coloring of the planar trivalent graph $\Gamma$. 
    
\end{theo}

Proof of \cref{main-2} is split up over the next few pages. The first thing on this line is the \textit{planar trivalent decomposition theorem} stated and proved below:

\begin{theo} \label{prop-deco}
    Every planar trivalent graph, when seen as a surface with defects $(\mathbb{S}^2, \Gamma)\in \text{Mor}(\text{Bord}_2^{\text{def}}(\mathcal{D}_{+}^{\mathbf{3}}))$ can be written as the composite $ \rho_{i_1}\circ \dots \circ \rho_{i_m}$ where each $\rho_{i_j}$ is one of the four patterns shown in the ~\cref{fig:bases}.
\end{theo}

\begin{figure}
    \centering
    \includegraphics[scale=0.6]{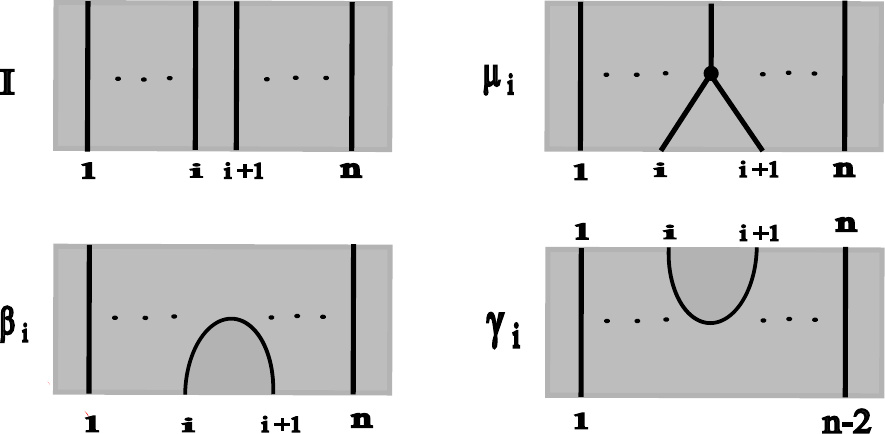}
    \caption{The four patterns $I, \mu_i , \beta_i$, and $\gamma_i$. The value of $n$ can be $2$, in which case $i$ equals $1$ and $i+1$ equals $2$. We have presented only the rectangle part of the cylinder. The part of the cylinder not shown is the region on the sphere without defect.}
    \label{fig:bases}
\end{figure}

\begin{proof}

First, because of a distinguished point, $\{-1\}$, on each generic cross-section, an edge can not wrap around $\mathbb{S}^2$. Next, since $\Gamma$ has only a finite number of vertices, it can be isotoped so that the handle decomposition of $\mathbb{S}^2$ in terms of the cylinder contains at most one vertex. Now, the portion of the cylinder far from this unique vertex is planar and thus generated by $U_i \coloneqq \gamma_i \circ \beta_i$. (See ~\cite{dovsen2003self} or ~\cite{kauffman1990invariant} for a proof of this fact.) On the other hand, the trivalent vertices that do not resemble $\mu_i$ can be obtained from a combination (vertical composition) of $I, \mu_i, \beta_i$, and $\gamma_i$ as shown in ~\cref{proof-2}.

\begin{figure}
    \centering
    \includegraphics[scale=0.6]{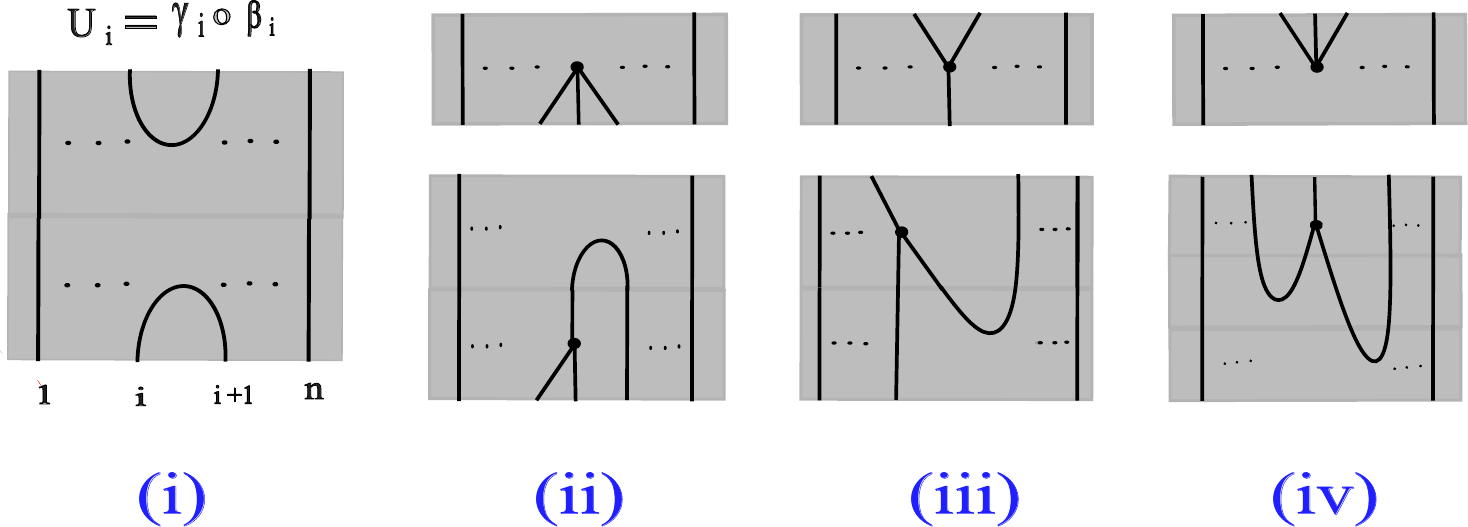}
    \caption{Diagram (i) is a picture of $U_i \coloneqq \gamma_i \circ \beta_i$. Diagrams (ii), (iii), and (iv) show how to write the trivalent vertex above these as a word involving $I, \mu, \gamma$, and $\beta$ from Figure-~\ref{fig:bases}. Where we have dropped the subscript $i$ for notational convenience. }
    \label{proof-2}
\end{figure}

\end{proof}

We also prove the following analogue for the category $\text{Bord}_2^{\textit{def,cw}}(\mathcal{D}_{+}^{\mathbf{3}})$

\begin{prop} \label{prop-deco-plcw}
    Given a planar graph $(\mathbb{S}^2, \Gamma) \in \text{Mor}(\text{Bord}_2^{\textit{def,cw}}(\mathcal{D}_{+}^{\mathbf{3}}))$, there is a PLCW decomposition of it making $\rho_{i_j}$ of \cref{prop-deco}. More precisely, each $\rho_{i_j}$ can be written as $\rho_{i_j} = P^{i_j}_1 \otimes \dots \otimes P^{i_j}_k$ for some $P^{i_j}_1, \dots P^{i_j}_k \in C_2(\mathbb{S}^2)$.
\end{prop}

\begin{proof}
    ~\cref{dumbbell} gives an idea about how to do it. First, choose a height function on $\mathbb{S}^2$ and obtain generic cross-sections containing the cylinders $\rho_{i_j}$. Since each $\rho_{i_j}$ has finitely many $1$-defects, insert a $0$-cell between any two consecutive defects. Referring to \cref{fig:bases}, we see that there is a bijection between all such $0$-cells inserted on either side of rectangles except for those between $i$ and $i+1$. Join these two to form $1$-cells. For $I$, this will immediately give a decomposition into basic-gons. For $\beta_i,\mu_i$, we join the two neighboring $0$-cells of the $0$-cell between $i$ and $i+1$ as shown in the picture below. The pattern $\gamma_i$ is done similarly to $\beta_i$.
    \begin{equation*}
        \centering
        \includegraphics[scale=0.8]{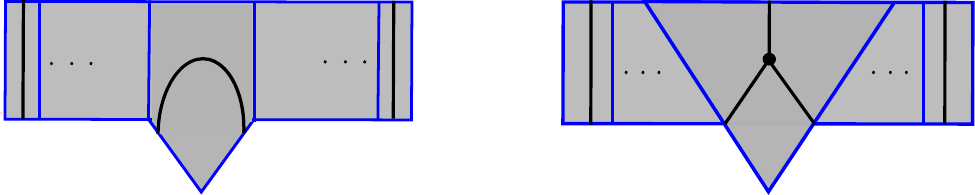}
    \end{equation*}
\end{proof}

Next, we prove the following important property of $3$-edge coloring which is analogous to the \textit{sum over all possible intermediate states} property in QFT. (See \cref{sec:motivation} for various references, and \cref{fig:quantum_decom}, Diagram (i) for a picture.)

\begin{lemma} \label{sum_over_IS}
    Let $(C_{12}, \Gamma_{12}) \in \text{Mor}(\textit{Bord}_2^{\textit{def,cw}}(\mathcal{D}_{+}^{\mathbf{3}}))(O_1, \partial O_2)$ and $O_t$ be a generic cross-section of $(C_{12}, \Gamma_{12})$ that fits into the composite bordism
    $$\begin{tikzcd}
      & (C_{1t}, \Gamma_{1t}) \arrow[dr, "\mathfrak{o}_1"] &  & (C_{t2}, \Gamma_{t2}) \arrow[dr, "\mathfrak{o}_2"] &  \\
     O_1 \arrow[ur, "\iota_1"] &  & O_t \arrow[ur, "\iota_2"] &  & O_2
    \end{tikzcd}$$
    If $\text{Tait}_{\hat{x}, \hat{y}}\Gamma_{xy}$ stands for coloring of the cylinder $(C_{xy}, \Gamma_{xy})$, with given (fixed) colors $\hat{x}$ on the in-boundary $x$ and $\hat{y}$ on the out-boundary $y$, then
    \begin{equation} \label{eqn:IS}
        \#\text{Tait}_{\hat{O}_1, \hat{O}_2}\Gamma_{12} = \sum_{\hat{O}_t}(\#\text{Tait}_{\hat{O}_1, \hat{O}_t}\Gamma_{1t})(\#\text{Tait}_{\hat{O}_t, \hat{O}_2}\Gamma_{12}).
    \end{equation}

\end{lemma}

\begin{proof}
    The lemma says that if colorings $\hat{O}_1$ of $O_1$ and $\hat{O}_2$ of $O_2$ are chosen, then the number of $3$-edge coloring of $\Gamma_{12}$, such that the in-boundary $O_1$ receives the color $\hat{O}_1$ and the out-boundary receives the color $\hat{O}_2$, is the sum of the product of number of $3$-edge coloring of the graph $\Gamma_{1t}$ with in-boundary $\hat{O}_1$ and out-boundary $\hat{O}_t$, and $\Gamma_{t2}$ with in-boundary $\hat{O}_t$ and out-boundary $\hat{O}_2$, counted over all the coloring $\hat{O}_t$ of a given intermediate cross-section $O_t$. We will prove this by establishing equality between two sets:

    \[A = \{ s \mid s \hspace{1mm} \text{is a Tait-coloring of} 
     (C_{12}, \Gamma_{12}) \hspace{1mm} \text{with in-boundary} \hspace{1mm} \hat{O}_1 \text{and out-boundary} \hspace{1mm}\hat{O}_2 \} \]
     
     and

    \[
    \begin{split}
    B = \{s  \mid s \hspace{1mm} \text{is obtained by gluing $s_1$ and $s_2$ along the common boundary} \hspace{1mm}
     \hat{O}_t \text{where} s_1 \\ \text{is a Tait-coloring of} \hspace{1mm} \Gamma_{1t} \text{with the in-boundary} \hspace{1mm} \hat{O}_1
     \text{and} \hspace{1mm} s_2 \text{is a Tait-coloring of}\\  \Gamma_{t2} \hspace{1mm} \text{with out-boundary} \hspace{1mm}  \hat{O}_2\}.    
    \end{split}
    \]

First, $B \subset A$ is obvious. Conversely, if $s \in A$, then $s$ restricts to two sections $s_1$ and $s_2$ that can be glued (composed) along the common boundary, namely the color that $O_t$ receives to give a Tait-coloring of $\Gamma_{12}$, which proves $A \subset B$. \cref{eqn:IS} is then a statement about the cardinality of $A$ (left) and $B$ (right). To find the cardinality of $B$, notice that for a coloring $\hat{O}_t$ of $O_t$, if there are $m$ distinct coloring of $\Gamma_{1t}$ with out-boundary $\hat{O}_t$, and $n$ distinct coloring of $\Gamma_{t2}$ with in-boundary $\hat{O}_t$, then they can be combined in $mn$ ways to give a Tait-coloring of $\Gamma_{12}$. The cardinality of $B$ is obtained by summing over all such coloring $\hat{O}_t$ of $O_t$.

\end{proof}

\cref{sum_over_IS} says that the sum can be taken over arbitrary coloring of $O_t$, that is, it may or may not lead to a Tait-coloring on any of $\Gamma_{1t}$ or $\Gamma_{t2}$. The contribution from a color $\hat{O}_t$, which can not be extended to Tait-coloring of $\Gamma_{12}$, is zero because of the relation $A \subset B$. It also means that for such a coloring either $\#\text{Tait}_{\hat{O}_1, \hat{O}_t}\Gamma_{1t} $ is zero or $\#\text{Tait}_{\hat{O}_t, \hat{O}_2}\Gamma_{12}$ is zero.

Let $\mathcal{B}(V)$ denote the set of bases of the $\mathbb{C}$-vector space $V$. For $V = X^{\otimes n}$, this is the set of colors or states that $\chi^{cw}$ assigns to a circle with $n$-defects. We state the following interpretation of the calculations of $\chi^{cw}(P)$ where $P$ is a polygon as in ~\ref{Pattern-1}, ~\ref{Pattern-2},~\ref{Pattern-3}, and ~\ref{Pattern-4}.  

\begin{prop} \label{basic-gon-chi}

    To the basic-gons of the category $\textit{Bord}_2^{\textit{def,cw}}(\mathcal{D}_{+}^{\mathbf{3}})$, when viewed as a cup $\mathbb{D}_P \in \text{Mor}(\textit{Bord}_2^{\textit{def,cw}}(\mathcal{D}_{+}^{\mathbf{3}}))(\emptyset, \partial \mathbb{D}_P)$, $\chi^{cw}$ assigns a vector $v \in \chi^{cw}(\partial \mathbb{D}_P)$ whose component in the direction of a basis vector $w_i \in \mathcal{D}(\chi^{cw}(\partial \mathbb{D}_P))$ is the number of ways the embedded graph $\Gamma_P$ can be $3$-edge colored so that the out-boundary $\partial \mathbb{D}_P$ receives a color $w_i$.
    
\end{prop}

\begin{proof}
    First, note that patterns ~\ref{Pattern-1} and ~\ref{Pattern-4} are both Pattern $P_{\gamma}$ from ~\ref{Pattern-4} as a basic-gon. So, in this case, the statement of the proposition is verified by \cref{basic-gon-color-1}. (See the map $S_1$ in \cref{fig:coloring-dumbbell}.) For $\mathbb{D}_{\mu}$ it follows from \cref{Prop:fusion} and the decomposition (iii) in the picture below.
    \begin{equation}
        \centering
        \includegraphics[scale=0.3]{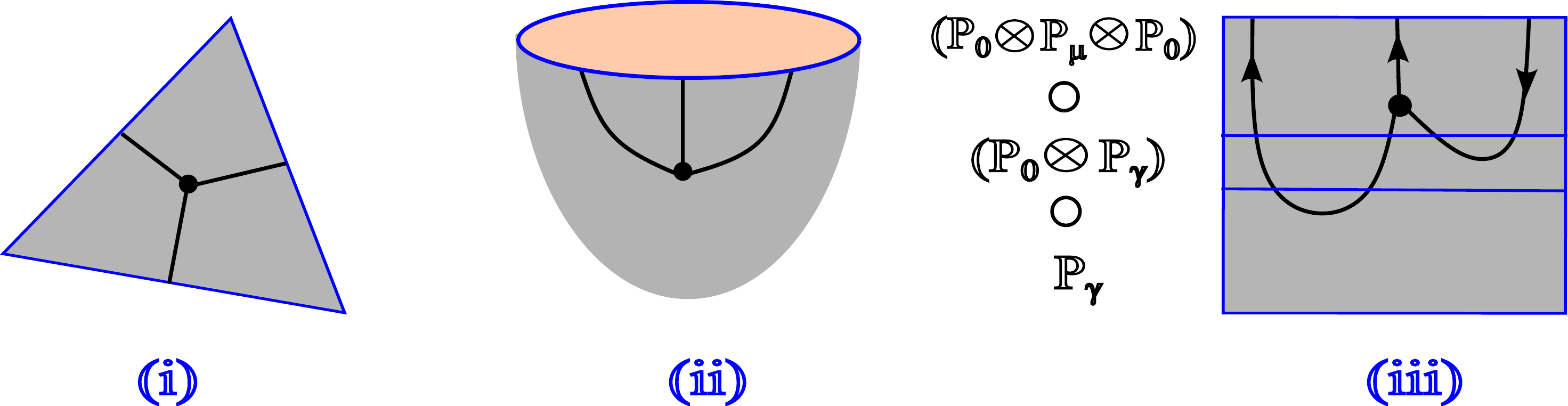}
    \end{equation}
It follows from the vertical composition shown in (iii), \cref{Prop-main}, and functoriality of $\chi^{cw}$ that
$$ \chi^{cw}(\mathbb{D}_{\mu}) = \chi^{cw}(P_0 \otimes P_{\mu} \otimes P_{0}) \circ \chi^{cw}(P_o \otimes P_{\gamma}) \circ x^{cw}(P_{\gamma})
$$ which gives
\begin{equation*}
    \begin{aligned}
        1_{\mathbb{C}} \xrightarrow{\chi^{cw}(P_{\gamma})} a \otimes a + b \otimes b + c \otimes C \xrightarrow{\chi^{cw}(P_0) \otimes \chi^{cw}(P_{\gamma})} a\otimes a \otimes a \otimes a + a \otimes a \otimes b \otimes b\\ 
        + a \otimes a \otimes c \otimes c
         + b \otimes b \otimes a\otimes a + b\otimes b \otimes b \otimes b + b\otimes b \otimes c \otimes c + c \otimes c \otimes a \otimes a
         + \\ c \otimes c \otimes b \otimes b
         + c \otimes c \otimes c \otimes c \xrightarrow{\chi^{cw}(P_0)\otimes \chi^{cw}(P_{\mu}) \otimes \chi^{cw}(P_0)} a \otimes c \otimes b
         + a \otimes b \otimes c + \\
         b \otimes c \otimes a + b \otimes a \otimes c + c \otimes b \otimes a + c \otimes a \otimes b   
    \end{aligned}
\end{equation*}
   but these are the words from the boundary of all the sections $s: \mathbb{D}_{\mu} \to \pi^{-1}(\mathbb{D}_{\mu})$ 
\end{proof}

The following lemma generalises \cref{basic-gon-chi}:

\begin{lemma} \label{cylinder}
    Let $\mathbb{S}_s$ and $\mathbb{S}_t$ be two objects in the category $\textit{Bord}_2^{\textit{def,cw}}(\mathcal{D}_{+}^{\mathbf{3}})$ comprising of single marked circles with $n_s$ and $n_t$ number of markings ($0$-defects) respectively. The TFT $\chi^{cw}: \textit{Bord}_2^{\textit{def,cw}}(\mathcal{D}_{+}^{\mathbf{3}}) \to \text{Vect}_F(\mathbb{C})$ assigns to a cylinder $(C_{st}, \Gamma_{st}): \mathbb{S}_s \to \mathbb{S}_t$ a linear map $\chi^{cw}(C_{st}, \Gamma_{st}): X^{\otimes n_s} \to X^{\otimes n_t}$ that sends a basis vector $v_j \in \mathcal{B}(X^{\otimes n_s})$ to a vector $B \in X^{\otimes n_t}$ such that the component of $B$ in the direction of a vector $w_i \in \mathcal{B}(X^{\otimes n_t})$ is the number of ways $\Gamma_{st}$ can be $3$-edge colored so that the in-boundary $\mathbb{S}_s$ receives the color $v_j$ and the out-boundary $S_t$ receives the color $w_i$.

\end{lemma}

If we denote the linear map $\chi^{cw}(C_{st}, \Gamma_{st}): X^{\otimes n_s} \to X^{\otimes n_t}$ by a $3^{n_t} \times 3^{n_s}$ matrix $A = (a_{ij})$, then with the notation of \cref{sum_over_IS},
\begin{equation} \label{matrix-Tait}
    a_{ij} = \#\text{Tait}_{v_j, w_i}\Gamma_{st}
\end{equation}
for $v_j \in \mathcal{B}(X^{\otimes n_s})$ and $w_i \in \mathcal{B}(X^{\otimes n_t})$. In other words, $a_{ij}$ is the number of ways $\Gamma_{st}$ can be $3$-edge colored so that the in-boundary $\mathbb{S}_s$ receives the color $v_j$ and the out-boundary $\mathbb{S}_t$ receives the color $w_i$. In this language, \cref{eqn:IS} is the familiar matrix product $a_{ij} = \sum_{k}b_{ik}c_{kj}$. Of course, this is expected from the composition of linear maps in \cref{cylinder}.

The following corollary to \cref{cylinder} is immediate:
\begin{coro} \label{cup-cap}
The TFT $\chi^{cw}$ makes the following assignments:
   \begin{enumerate}
       \item For a cup $(D_t, \Gamma_t) : \emptyset \to \mathbb{S}_t$, a vector $w \in X^{\otimes n_t}$ whose component in the direction of $w_i \in \mathcal{B}(X^{\otimes n_t})$ is the number of ways one can $3$-edge color the graph $\Gamma_t$ so that $w_i$ is the color received by the boundary circle $\mathbb{S}_t$.
       \item For a cap $(U_s, \Gamma_s) : \mathbb{S}_s \to \mathbb{C}$, a covector $v$ which evaluates to $\kappa_i$ on $v_i \in \mathcal{B}(X^{\otimes n_s})$ with the property that there are $\kappa_i$ ways to $3$-edge color $\Gamma_s$ so that the in-boundary $\mathbb{S}_s$ receives the color $v_i$.
    \end{enumerate}
    
\end{coro}

Let us understand \cref{cup-cap} through an example.

\begin{exam} \label{exam:cup-cap_decom}
\cref{fig:cup-cap_decom} shows a cup in Diagram (i), which is the same surface with defects $\hat{\Sigma}$ from \cref{exam:coloring_exam-1}. Using a decomposition as in the right of Diagram (i) in \cref{fig:cup-cap_decom}, we see that $\chi^{cw}$ assigns the vector $2a \otimes a + 2b \otimes b + 2 c \otimes c.$ \cref{cup-cap}, (1), then says that there are precisely two $3$-edge coloring of the graph embedded in $\hat{\Sigma}$ such that boundary circle receives the color $a \otimes a$, and the same is true for $b \otimes b$ and $c \otimes c$. This is exactly what we see in \cref{fig:coloring_exam-1}. \cref{cup-cap}, (1), also says that there are zero $3$-edge coloring of the graph in $\hat{\Sigma}$ for each of the boundary colors $a \otimes b, b \otimes a, a \otimes c, c \otimes a, b \otimes c $ and $ c \otimes b$. Compare this with the similar statement in \cref{rema:distinct_colors}. 

In \cref{fig:cup-cap_decom}, Diagram (ii), denotes the surface with defects in the left by $\hat{\Sigma}^{\ast}$. Next, using the decomposition on the right in the same diagram, we see that $\chi^{cw}(\hat{\Sigma}^{\ast})$ maps 

\begin{equation} \label{eqn:cap-1}
 a \otimes a \to a \otimes (b \otimes c + c \otimes b) = a \otimes b \otimes c + a \otimes c \otimes b \to c \times c + b \otimes b \to 2.
\end{equation}

and the same is true for $b \otimes b$ and $c \otimes c$. On the other hand, $\chi^{cw}(\hat{\Sigma}^{\ast})$ maps

\begin{equation} \label{eqn:cap-2}
    a \otimes b \to a \otimes (a \otimes c + c \otimes a) = a \otimes a \otimes c + a \otimes c \otimes a \to 0 + b \otimes a \to 0 + 0 = 0.
\end{equation}
and, the same is true for $b \otimes a, a \otimes c, c \otimes a, b \otimes c$ and $c \otimes b$. This is expected since one can see that flipping a surface does not change the $3$-edge coloring of the graph embedded in it. Noticing this, we see that the value of $\chi^{cw}(\hat{\Sigma}^{\ast})$ that we calculated in \cref{eqn:cap-1} and \cref{eqn:cap-2} is in accordance with \cref{cup-cap}, (2).

\end{exam}

    \begin{figure}
        \centering
        \includegraphics[width=0.9\linewidth]{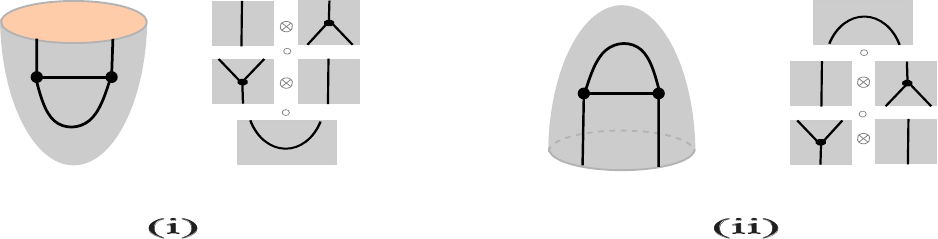}
        \caption{\cref{exam:cup-cap_decom} for details.}
        \label{fig:cup-cap_decom}
    \end{figure}

Finally, we are ready to prove the main result of this paper. First, we give proof of \cref{main-2} from \cref{cylinder}, \cref{sum_over_IS}, and \cref{cup-cap}.

\begin{proof}[Proof of \cref{main-2}]
    For a given surface with defects $(\mathbb{S}^2, \Gamma)$, choose a generic cross-section $\mathbb{S}_t$. By \cref{cup-cap} (1), $\chi^{cw}(D_t, \Gamma_t)$ is a vector of the form $\sum_{i}\lambda_i w_i$ where $\lambda_i$ is the number of ways one can color $\Gamma_t$ so that the cross-section  $\mathbb{S}_t$ gets the color $w_i$. Now, consider the cap $(U_t, \Gamma'_t)$, where $\Gamma'_t$ is the portion of $\Gamma$ embedded in the cap $U_t$. By \cref{cup-cap} (2),  $\chi^{cw}(U_t, \Gamma'_t)$ is a covector that maps $w_i \in \mathcal{B}(X^{\otimes n_t})$ to $\kappa_i \in \mathbb{C}$ with the property that the graph $\Gamma'_t$ can be colored in $\kappa_i$ ways so that $\mathbb{S}_t$ receives a color $w_i$. Composing the two we get:
    \begin{equation}
        \begin{split}
        \chi^{cw}(\mathbb{S}^2, \Gamma)(1) & = \chi^{cw}((U_t, \Gamma'_t) \circ (D_t, \Gamma_t))(1)\\
         & = \chi^{cw}((U_t, \Gamma'_t)) \circ \chi^{cw}((D_t, \Gamma_t))(1)\\
         & = \sum_i \kappa_i \lambda_i.
        \end{split}
    \end{equation}

Which is the number of Tait-coloring of $(\mathbb{S}^2, \Gamma)$ by \cref{sum_over_IS}. So, once we have shown that this number is independent of the choice of the generic cross-section, we are done. For that, let $\mathbb{S}_s$ be another generic cross-section. Without loss of generality, we can assume that it fits into the following composition 

\begin{equation}
   \emptyset \xrightarrow{(D_s, \Gamma_s)} \mathbb{S}_s \xrightarrow{(C_{st}, \Gamma_{st})} \mathbb{S}_t \xrightarrow{(U_t, \Gamma'_t)} \emptyset
\end{equation}

The action of $\chi^{cw}$ on it gives:

\begin{equation}
   \mathbb{C} \xrightarrow{\chi^{cw}((D_s, \Gamma_s))} X^{\otimes n_s} \xrightarrow{\chi^{cw}((C_{st}, \Gamma_{st}))} X^{\otimes n_t} \xrightarrow{\chi^{cw}((U_t, \Gamma'_t))} \mathbb{C}
\end{equation}

Let $\chi^{cw}(\mathbb{S}^2, \Gamma)(1) = \sum_{j}\kappa'_{j} \lambda'_{j}$ along $\mathbb{S}_s$, which means $\chi^{cw}(D_s, \Gamma_s) = \sum_{j}\lambda'_j v_j$ for $v_j \in \mathcal{B}(X^{\otimes n_s})$, and $\chi^{cw}(U_s, \Gamma'_s)$ maps $v_j \in \mathcal{B}(X^{\otimes n_s})$ to $\kappa'_j$. Now, suppose $\chi^{cw}(C_{st}, \Gamma_{st}) = (a_{ij})$ in the same bases $\{v_j\}$ of $X^{\otimes n_s}$ and $\{w_i\}$ of $X^{\otimes n_t}$. The functoriality of $\chi^{cw}$ applied on the identities $(C_{st}, \Gamma_{st}) \circ (D_s, \Gamma_s) = (D_t, \Gamma_t)$ and $(U_t, \Gamma'_t) \circ (C_{st}, \Gamma_{st}) = (U_s, \Gamma'_s)$ gives
$$ \lambda_i = \sum_{j}a_{ij} \lambda'_j \quad , \quad \kappa'_j = \sum_{i} \kappa_i a_{ij}$$ respectively. This leads to
$$
        \sum_{j}\kappa'_j \lambda'_j = \sum_{j}\sum_{i}\kappa_i a_{ij} \lambda'_j
        = \sum_{i}\sum_{j}\kappa_i a_{ij} \lambda'_j
         = \sum_{i}\kappa_{i} \sum_{j}a_{ij} \lambda'_j
         = \sum_{i}\kappa_{i}\lambda_{i}.
$$  

\end{proof}

\begin{proof}[Proof of \cref{cylinder}]
    By \cref{prop-deco} every cylinder $(C_{st}, \Gamma_{st})$ can be written as the composition of basic cylinders as in \cref{fig:bases}. Let $n$ be the number of such compositions, that is, the minimum number of basic cylinders $\rho_{i_j}$ required to make a given cylinder $(C_{st}, \Gamma_{st})$. We prove \cref{cylinder} by induction on $n$. The base case is $n=1$. In this case, $(C_{st}, \Gamma_{st})$ is one of the four basic cylinders in \cref{fig:bases}. Use \cref{prop-deco-plcw} to obtain a cell-decomposition and write each basic cylinder as a horizontal composition of basic-gons. Then the statement of the \cref{cylinder} follows from \cref{Prop:fusion} and \cref{basic-gon-chi}. (Since, other than $(i, i+1)$, everything else is the identity, $(i, i+1)$ is one of the four patterns appearing in \cref{basic-gon-chi}.) Now, for the induction step, assume the statement of \cref{cylinder} is true for all $k < n$. Choose a generic cross-section $\mathbb{S}_o$ of the cylinder $(C_{st}, \Gamma_{st})$ and obtain the composite 
    $$\begin{tikzcd}
      & (C_{so}, \Gamma_{so}) \arrow[dr, "\mathfrak{o}_1"] &  & (C_{ot}, \Gamma{ot}) \arrow[dr, "\mathfrak{o}_2"] &  \\
     \mathbb{S}_s \arrow[ur, "\iota_1"] &  & \mathbb{S}_o \arrow[ur, "\iota_2"] &  & \mathbb{S}_t
    \end{tikzcd}$$   
  Each of the cylinders $(C_{so}, \Gamma_{so})$ and $(C_{ot}, \Gamma_{ot})$ has lengths less than $n$, so the statement of \cref{cylinder} is true for them by the induction hypothesis. By the functoriality of $\chi^{cw}$, we get the composite
\begin{equation}
    X^{\otimes n_s} \xrightarrow{\chi^{cw}(C_{so}, \Gamma_{so})}  X^{\otimes n_o} \xrightarrow{\chi^{cw}(C_{ot}, \Gamma_{ot})} X^{\otimes n_t}
\end{equation}
which equals $\chi^{cw}(C_{st}, \Gamma_{st})$. Let $\mathcal{B}(X^{\otimes n_s}) = \{v_j\}$, $\mathcal{B}(X^{\otimes n_o}) = \{z_k\}$ and  $\mathcal{B}(X^{\otimes n_t}) = \{w_i\}$, and in these bases, the matrices of $\chi^{cw}(C_{st}, \Gamma_{st})$, $\chi^{cw}(C_{so}, \Gamma_{so})$, and $\chi^{cw}(C_{ot}, \Gamma_{ot})$ are given by $A \coloneqq (a_{ij})$, $B \coloneqq (b_{kj})$, and the matrix of $C \coloneqq (c_{ik})$ respectively. By linearity we get

\begin{equation} \label{final_color}
    a_{ij} = \sum_{k}c_{ik}b_{kj}
\end{equation}

 By \cref{cylinder} $c_{ik}$ is the number of coloring of $(C_{ot}, \Gamma_{ot})$ with in-boundary color $z_k$ and out-boundary color $w_i$. Similarly, $b_{kj}$ is the number of coloring of $(C_{so}, \Gamma_{so})$ with in-boundary color $v_j$ and out-boundary color $z_k$. Now, \cref{sum_over_IS} implies that \cref{final_color} is the number of $3$-edge coloring of $(C_{st}, \Gamma_{st})$ with an in-boundary color $v_j$ and out-boundary color $w_i$, but by the definition of the matrix $a_{ij}$, this is the component of $\chi^{cw}(C_{st},\Gamma_{st})(v_j)$ in the direction of $w_i$. 

\end{proof}

\begin{rema} \label{Tait-no-ind}
    We chose a PLCW decomposition to define the Tait-coloring, but \cref{main-2} also shows that the number of Tait-coloring is independent of this choice as the functor $\chi^{cw}$ is. See \cref{plcwcd} and~\cite{davydov2011field}, Section-3.6. 
\end{rema}

One can also show using \textit{Kirillov-moves} (~\cite{kirillov2012piecewise}, Section-6,7) that the definition of a $3$-edge coloring is independent of a choice of a PLCW decomposition, but we only need the number of such coloring, so we are going to be content with \cref{Tait-no-ind}.

We conclude this section with a conjecture, which is a reformulation of the $4$-color theorem in the language developed in this paper:

\begin{conj} \label{4-color}
    If $\Gamma$ is a planar trivalent graph with no bridge, then $$\chi^{cw}(\mathbb{S}^2, \Gamma)(1) \neq 0 .$$
\end{conj}

It is immediate from \cref{cor:big-1} that if $\Gamma$ has a bridge, then $\chi^{cw}(\mathbb{S}^2, \Gamma)(1)$ equals zero. \cref{4-color} is the converse of it. Equivalence with the $4$-color theorem is easily established from the statement of \cref{main-2} and a result due to Tait [cf ~\cite{tait1880note} and ~\cite{baldridge2018cohomology}] that states:

\enquote{  Every planar graph is 4-colorable if and only if every planar bridgeless cubic graph is 3-edge-colorable.  }

This statement can be proved using \cref{coro:Tait} by taking $Y = K_4$ as a $K_4$-set, and a process that is dual to the blowing-up of a graph at vertices. (See also ~\cite{baldridge2023topological}.)

For a one-component graph, a bridge on a planar graph is equivalent to the existence of a generic cross-section with a single defect. Therefore, \cref{4-color} can be reformulated as:

\begin{conj} \label{conj:alternate_4-color}
    If the linear map $\chi^{cw}(\mathbb{S}^2, \Gamma): \mathbb{C} \to \mathbb{C}$ can be written as the composition 
    $$ \mathbb{C} \xrightarrow{\chi^{cw}(\emptyset, \mathbb{S}_t)} X \xrightarrow{\chi^{cw}(\mathbb{S}_t, \emptyset)} \mathbb{C}
    $$
    then it is the zero map.
\end{conj}


\section{Additional topics}

This section is meant to look back and explore some of the key ideas developed in this paper that led to the main result.

\subsection{Bundles at the level of  $n$-category}  \label{sec:additional}

Here, we take another look at \cref{coloring_process}, of the coloring process, and underscore the underlying abstract structure. These should be compared with the ideas developed in ~\cite{baez2005higher} and surveyed in ~\cite{baez2011invitation}. We begin by noticing that the conditions (2) in \cref{coloring_process}, can be given in terms of the commutative pasting diagrams as in \cref{fig:section_horiz-composition} :

\begin{figure}[ht]
    \centering

\tikzset{every picture/.style={line width=0.75pt}} 

\begin{tikzpicture}[x=0.75pt,y=0.75pt,yscale=-1,xscale=1]

\draw    (165.17,372.64) -- (165.17,417.73) ;
\draw [shift={(165.17,419.73)}, rotate = 270] [color={rgb, 255:red, 0; green, 0; blue, 0 }  ][line width=0.75]    (10.93,-3.29) .. controls (6.95,-1.4) and (3.31,-0.3) .. (0,0) .. controls (3.31,0.3) and (6.95,1.4) .. (10.93,3.29)   ;
\draw    (259.79,330.58) -- (369.94,330.58) ;
\draw [shift={(371.94,330.58)}, rotate = 180] [color={rgb, 255:red, 0; green, 0; blue, 0 }  ][line width=0.75]    (10.93,-3.29) .. controls (6.95,-1.4) and (3.31,-0.3) .. (0,0) .. controls (3.31,0.3) and (6.95,1.4) .. (10.93,3.29)   ;
\draw    (262.75,505.87) -- (372.9,505.87) ;
\draw [shift={(374.9,505.87)}, rotate = 180] [color={rgb, 255:red, 0; green, 0; blue, 0 }  ][line width=0.75]    (10.93,-3.29) .. controls (6.95,-1.4) and (3.31,-0.3) .. (0,0) .. controls (3.31,0.3) and (6.95,1.4) .. (10.93,3.29)   ;
\draw    (508.44,372.64) -- (508.44,417.73) ;
\draw [shift={(508.44,419.73)}, rotate = 270] [color={rgb, 255:red, 0; green, 0; blue, 0 }  ][line width=0.75]    (10.93,-3.29) .. controls (6.95,-1.4) and (3.31,-0.3) .. (0,0) .. controls (3.31,0.3) and (6.95,1.4) .. (10.93,3.29)   ;

\draw (99.45,324.27) node [anchor=north west][inner sep=0.75pt]   [align=left] {$\displaystyle \ast $};
\draw (158.42,324.27) node [anchor=north west][inner sep=0.75pt]   [align=left] {$\displaystyle \ast $};
\draw (124.01,323.38) node [anchor=north west][inner sep=0.75pt]   [align=left] {$\displaystyle \Downarrow $};
\draw (184.28,324.86) node [anchor=north west][inner sep=0.75pt]   [align=left] {$\displaystyle \Downarrow $};
\draw (96,498.14) node [anchor=north west][inner sep=0.75pt]   [align=left] {$\displaystyle \ast $};
\draw (219.69,499.14) node [anchor=north west][inner sep=0.75pt]   [align=left] {$\displaystyle \ast $};
\draw (134.57,495.07) node [anchor=north west][inner sep=0.75pt]   [align=left] {$\displaystyle \Downarrow $};
\draw (135.95,322.56) node [anchor=north west][inner sep=0.75pt]   [align=left] {$\displaystyle P_{j}$};
\draw (196.78,323.37) node [anchor=north west][inner sep=0.75pt]   [align=left] {$\displaystyle P_{i}$};
\draw (310.36,311.64) node [anchor=north west][inner sep=0.75pt]   [align=left] {$\displaystyle s$};
\draw (311.74,485.41) node [anchor=north west][inner sep=0.75pt]   [align=left] {$\displaystyle s$};
\draw (217.79,325.72) node [anchor=north west][inner sep=0.75pt]   [align=left] {$\displaystyle \ast $};
\draw (145.41,385.34) node [anchor=north west][inner sep=0.75pt]   [align=left] {$\displaystyle \otimes $};
\draw (149.86,495.3) node [anchor=north west][inner sep=0.75pt]   [align=left] {$\displaystyle P_{i} \otimes P_{j}$};
\draw (399.59,317.04) node [anchor=north west][inner sep=0.75pt]   [align=left] {$\displaystyle s( \ast )$};
\draw (490.65,316.04) node [anchor=north west][inner sep=0.75pt]   [align=left] {$\displaystyle s( \ast )$};
\draw (437.74,318.15) node [anchor=north west][inner sep=0.75pt]   [align=left] {$\displaystyle \Downarrow $};
\draw (529.74,318.62) node [anchor=north west][inner sep=0.75pt]   [align=left] {$\displaystyle \Downarrow $};
\draw (451.01,316.32) node [anchor=north west][inner sep=0.75pt]   [align=left] {$\displaystyle s( P_{j})$};
\draw (543.29,316.14) node [anchor=north west][inner sep=0.75pt]   [align=left] {$\displaystyle s( P_{i})$};
\draw (583.75,318.49) node [anchor=north west][inner sep=0.75pt]   [align=left] {$\displaystyle s( \ast )$};
\draw (486.41,385.34) node [anchor=north west][inner sep=0.75pt]   [align=left] {$\displaystyle \otimes $};
\draw (406.83,497.42) node [anchor=north west][inner sep=0.75pt]   [align=left] {$\displaystyle s( *)$};
\draw (585.59,494.25) node [anchor=north west][inner sep=0.75pt]   [align=left] {$\displaystyle s( *)$};
\draw (466.13,499.24) node [anchor=north west][inner sep=0.75pt]   [align=left] {$\displaystyle \Downarrow $};
\draw (480.04,511.47) node [anchor=north west][inner sep=0.75pt]   [align=left] {$\displaystyle s( P_{i}) \otimes s( P_{j})$};
\draw (517.79,498) node [anchor=north west][inner sep=0.75pt]   [align=left] {=};
\draw (486.44,478) node [anchor=north west][inner sep=0.75pt]   [align=left] {$\displaystyle s( P_{i} \otimes P_{j})$};
\draw    (113.4,320.27) .. controls (129.32,297.7) and (144.54,297.25) .. (159.08,318.92) ;
\draw [shift={(159.97,320.27)}, rotate = 237.26] [color={rgb, 255:red, 0; green, 0; blue, 0 }  ][line width=0.75]    (10.93,-3.29) .. controls (6.95,-1.4) and (3.31,-0.3) .. (0,0) .. controls (3.31,0.3) and (6.95,1.4) .. (10.93,3.29)   ;
\draw    (114.42,346.27) .. controls (129.07,364.91) and (143.46,365.38) .. (157.6,347.67) ;
\draw [shift={(158.68,346.27)}, rotate = 127.15] [color={rgb, 255:red, 0; green, 0; blue, 0 }  ][line width=0.75]    (10.93,-3.29) .. controls (6.95,-1.4) and (3.31,-0.3) .. (0,0) .. controls (3.31,0.3) and (6.95,1.4) .. (10.93,3.29)   ;
\draw    (108.65,494.14) .. controls (145.42,433.67) and (183,433.7) .. (221.39,494.23) ;
\draw [shift={(221.97,495.14)}, rotate = 237.87] [color={rgb, 255:red, 0; green, 0; blue, 0 }  ][line width=0.75]    (10.93,-3.29) .. controls (6.95,-1.4) and (3.31,-0.3) .. (0,0) .. controls (3.31,0.3) and (6.95,1.4) .. (10.93,3.29)   ;
\draw    (108.76,520.14) .. controls (145.04,577.74) and (182.41,578.37) .. (220.89,522) ;
\draw [shift={(221.47,521.14)}, rotate = 124.05] [color={rgb, 255:red, 0; green, 0; blue, 0 }  ][line width=0.75]    (10.93,-3.29) .. controls (6.95,-1.4) and (3.31,-0.3) .. (0,0) .. controls (3.31,0.3) and (6.95,1.4) .. (10.93,3.29)   ;
\draw    (172.62,320.27) .. controls (188.41,298.19) and (203.73,298.21) .. (218.59,320.34) ;
\draw [shift={(219.5,321.72)}, rotate = 237.2] [color={rgb, 255:red, 0; green, 0; blue, 0 }  ][line width=0.75]    (10.93,-3.29) .. controls (6.95,-1.4) and (3.31,-0.3) .. (0,0) .. controls (3.31,0.3) and (6.95,1.4) .. (10.93,3.29)   ;
\draw    (172.72,346.27) .. controls (187.65,366.39) and (202.36,367.35) .. (216.87,349.16) ;
\draw [shift={(217.98,347.72)}, rotate = 127.1] [color={rgb, 255:red, 0; green, 0; blue, 0 }  ][line width=0.75]    (10.93,-3.29) .. controls (6.95,-1.4) and (3.31,-0.3) .. (0,0) .. controls (3.31,0.3) and (6.95,1.4) .. (10.93,3.29)   ;
\draw    (423.88,313.04) .. controls (449.53,284.12) and (474.37,283.37) .. (498.4,310.77) ;
\draw [shift={(499.5,312.04)}, rotate = 229.66] [color={rgb, 255:red, 0; green, 0; blue, 0 }  ][line width=0.75]    (10.93,-3.29) .. controls (6.95,-1.4) and (3.31,-0.3) .. (0,0) .. controls (3.31,0.3) and (6.95,1.4) .. (10.93,3.29)   ;
\draw    (425.25,339.04) .. controls (452.04,365.28) and (476.54,365.35) .. (498.74,339.25) ;
\draw [shift={(499.76,338.04)}, rotate = 129.44] [color={rgb, 255:red, 0; green, 0; blue, 0 }  ][line width=0.75]    (10.93,-3.29) .. controls (6.95,-1.4) and (3.31,-0.3) .. (0,0) .. controls (3.31,0.3) and (6.95,1.4) .. (10.93,3.29)   ;
\draw    (514.54,312.04) .. controls (539.22,281.81) and (564.98,282.14) .. (591.79,313.05) ;
\draw [shift={(593.01,314.49)}, rotate = 229.87] [color={rgb, 255:red, 0; green, 0; blue, 0 }  ][line width=0.75]    (10.93,-3.29) .. controls (6.95,-1.4) and (3.31,-0.3) .. (0,0) .. controls (3.31,0.3) and (6.95,1.4) .. (10.93,3.29)   ;
\draw    (514.87,338.04) .. controls (540.48,366.91) and (565.77,368.13) .. (590.75,341.71) ;
\draw [shift={(591.89,340.49)}, rotate = 132.5] [color={rgb, 255:red, 0; green, 0; blue, 0 }  ][line width=0.75]    (10.93,-3.29) .. controls (6.95,-1.4) and (3.31,-0.3) .. (0,0) .. controls (3.31,0.3) and (6.95,1.4) .. (10.93,3.29)   ;
\draw    (427.53,493.42) .. controls (478.66,434.46) and (532.73,433.11) .. (589.72,489.4) ;
\draw [shift={(590.58,490.25)}, rotate = 224.92] [color={rgb, 255:red, 0; green, 0; blue, 0 }  ][line width=0.75]    (10.93,-3.29) .. controls (6.95,-1.4) and (3.31,-0.3) .. (0,0) .. controls (3.31,0.3) and (6.95,1.4) .. (10.93,3.29)   ;
\draw    (428.13,519.42) .. controls (482.36,578.51) and (536.62,577.76) .. (590.92,517.17) ;
\draw [shift={(591.74,516.25)}, rotate = 131.58] [color={rgb, 255:red, 0; green, 0; blue, 0 }  ][line width=0.75]    (10.93,-3.29) .. controls (6.95,-1.4) and (3.31,-0.3) .. (0,0) .. controls (3.31,0.3) and (6.95,1.4) .. (10.93,3.29)   ;

\end{tikzpicture}

    \caption{The commutativity of this diagram means that $s$ respects the horizontal composition: $s(P_i) \otimes s(P_j) = s(P_i) \otimes s(P_j)$.}
\label{fig:section_horiz-composition}

\end{figure}
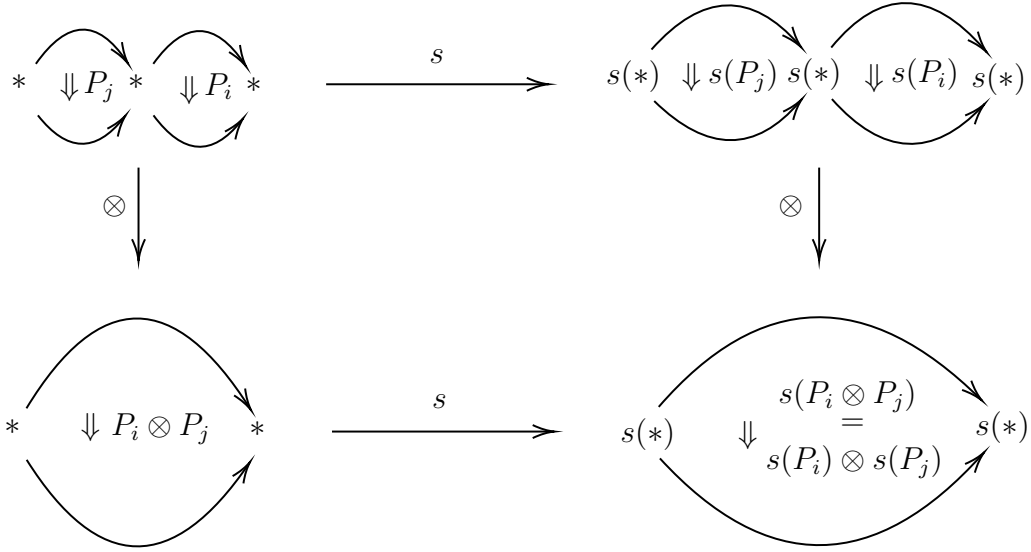

Similarly, the conditions (3) in \cref{coloring_process}, can be given in terms of the commutative pasting diagrams as shown in \cref{fig:section_vert-composition}.

\begin{figure}[ht]

\tikzset{every picture/.style={line width=0.75pt}} 

\begin{tikzpicture}[x=0.75pt,y=0.75pt,yscale=-1,xscale=1]

\draw    (115.43,321.81) -- (216.08,321.81) ;
\draw [shift={(218.08,321.81)}, rotate = 180] [color={rgb, 255:red, 0; green, 0; blue, 0 }  ][line width=0.75]    (10.93,-3.29) .. controls (6.95,-1.4) and (3.31,-0.3) .. (0,0) .. controls (3.31,0.3) and (6.95,1.4) .. (10.93,3.29)   ;
\draw    (439.66,324.34) -- (544.34,322.71) ;
\draw [shift={(546.34,322.68)}, rotate = 179.11] [color={rgb, 255:red, 0; green, 0; blue, 0 }  ][line width=0.75]    (10.93,-3.29) .. controls (6.95,-1.4) and (3.31,-0.3) .. (0,0) .. controls (3.31,0.3) and (6.95,1.4) .. (10.93,3.29)   ;
\draw    (165.93,392.19) -- (165.93,444.39) ;
\draw [shift={(165.93,446.39)}, rotate = 270] [color={rgb, 255:red, 0; green, 0; blue, 0 }  ][line width=0.75]    (10.93,-3.29) .. controls (6.95,-1.4) and (3.31,-0.3) .. (0,0) .. controls (3.31,0.3) and (6.95,1.4) .. (10.93,3.29)   ;
\draw    (256.36,322.85) -- (377.88,322.85) ;
\draw [shift={(379.88,322.85)}, rotate = 180] [color={rgb, 255:red, 0; green, 0; blue, 0 }  ][line width=0.75]    (10.93,-3.29) .. controls (6.95,-1.4) and (3.31,-0.3) .. (0,0) .. controls (3.31,0.3) and (6.95,1.4) .. (10.93,3.29)   ;
\draw    (256.31,515.16) -- (377.83,515.16) ;
\draw [shift={(379.83,515.16)}, rotate = 180] [color={rgb, 255:red, 0; green, 0; blue, 0 }  ][line width=0.75]    (10.93,-3.29) .. controls (6.95,-1.4) and (3.31,-0.3) .. (0,0) .. controls (3.31,0.3) and (6.95,1.4) .. (10.93,3.29)   ;
\draw    (493,392.27) -- (492.91,408.92) -- (492.75,436.89) ;
\draw [shift={(492.74,438.89)}, rotate = 270.32] [color={rgb, 255:red, 0; green, 0; blue, 0 }  ][line width=0.75]    (10.93,-3.29) .. controls (6.95,-1.4) and (3.31,-0.3) .. (0,0) .. controls (3.31,0.3) and (6.95,1.4) .. (10.93,3.29)   ;

\draw (89.66,313.85) node [anchor=north west][inner sep=0.75pt]   [align=left] {$\displaystyle \ast $};
\draw (226,313.85) node [anchor=north west][inner sep=0.75pt]   [align=left] {$\displaystyle \ast $};
\draw (158.51,283.52) node [anchor=north west][inner sep=0.75pt]   [align=left] {$\displaystyle \Downarrow $};
\draw (157.62,339.06) node [anchor=north west][inner sep=0.75pt]   [align=left] {$\displaystyle \Downarrow $};
\draw (91.36,507.16) node [anchor=north west][inner sep=0.75pt]   [align=left] {$\displaystyle \ast $};
\draw (227.69,506.16) node [anchor=north west][inner sep=0.75pt]   [align=left] {$\displaystyle \ast $};
\draw (135.7,505.95) node [anchor=north west][inner sep=0.75pt]   [align=left] {$\displaystyle \Downarrow $};
\draw (402.98,313.85) node [anchor=north west][inner sep=0.75pt]   [align=left] {$\displaystyle s( \ast )$};
\draw (550.03,313.85) node [anchor=north west][inner sep=0.75pt]   [align=left] {$\displaystyle s( \ast )$};
\draw (473.18,282.02) node [anchor=north west][inner sep=0.75pt]   [align=left] {$\displaystyle \Downarrow $};
\draw (474.59,343.59) node [anchor=north west][inner sep=0.75pt]   [align=left] {$\displaystyle \Downarrow $};
\draw (403.76,505.2) node [anchor=north west][inner sep=0.75pt]   [align=left] {$\displaystyle s( \ast )$};
\draw (553.45,505.2) node [anchor=north west][inner sep=0.75pt]   [align=left] {$\displaystyle s( \ast )$};
\draw (447.16,503.65) node [anchor=north west][inner sep=0.75pt]   [align=left] {$\displaystyle \Downarrow $};
\draw (170.8,282.48) node [anchor=north west][inner sep=0.75pt]   [align=left] {$\displaystyle P_{\nu }$};
\draw (169.95,338.18) node [anchor=north west][inner sep=0.75pt]   [align=left] {$\displaystyle P_{\mu }$};
\draw (150.4,506.04) node [anchor=north west][inner sep=0.75pt]   [align=left] {$\displaystyle P_{\mu } \circ P\nu $};
\draw (150.37,406.41) node [anchor=north west][inner sep=0.75pt]   [align=left] {$\displaystyle \circ $};
\draw (477.5,400.57) node [anchor=north west][inner sep=0.75pt]   [align=left] {$\displaystyle \circ $};
\draw (312.62,302.41) node [anchor=north west][inner sep=0.75pt]   [align=left] {$\displaystyle s$};
\draw (310.83,495.69) node [anchor=north west][inner sep=0.75pt]   [align=left] {$\displaystyle s$};
\draw (485.47,279.94) node [anchor=north west][inner sep=0.75pt]   [align=left] {$\displaystyle s( P_{\nu })$};
\draw (486.59,341.51) node [anchor=north west][inner sep=0.75pt]   [align=left] {$\displaystyle s( P_{\mu })$};
\draw (469.91,482.22) node [anchor=north west][inner sep=0.75pt]   [align=left] {$\displaystyle s( P_{\mu } \circ P_{\nu })$};
\draw (493.43,501.81) node [anchor=north west][inner sep=0.75pt]   [align=left] {$\displaystyle =$};
\draw (460,517.5) node [anchor=north west][inner sep=0.75pt]   [align=left] {$\displaystyle s( P_{\mu }) \circ s( P_{\nu })$};
\draw    (103.7,309.85) .. controls (146.95,249.9) and (188.33,249.6) .. (227.86,308.95) ;
\draw [shift={(228.46,309.85)}, rotate = 236.61] [color={rgb, 255:red, 0; green, 0; blue, 0 }  ][line width=0.75]    (10.93,-3.29) .. controls (6.95,-1.4) and (3.31,-0.3) .. (0,0) .. controls (3.31,0.3) and (6.95,1.4) .. (10.93,3.29)   ;
\draw    (104.08,335.85) .. controls (148.25,393.5) and (189.5,393.79) .. (227.82,336.71) ;
\draw [shift={(228.4,335.85)}, rotate = 123.59] [color={rgb, 255:red, 0; green, 0; blue, 0 }  ][line width=0.75]    (10.93,-3.29) .. controls (6.95,-1.4) and (3.31,-0.3) .. (0,0) .. controls (3.31,0.3) and (6.95,1.4) .. (10.93,3.29)   ;
\draw    (104.83,503.16) .. controls (146.37,440.03) and (187.92,439.38) .. (229.46,501.22) ;
\draw [shift={(230.09,502.16)}, rotate = 236.38] [color={rgb, 255:red, 0; green, 0; blue, 0 }  ][line width=0.75]    (10.93,-3.29) .. controls (6.95,-1.4) and (3.31,-0.3) .. (0,0) .. controls (3.31,0.3) and (6.95,1.4) .. (10.93,3.29)   ;
\draw    (105.43,529.16) .. controls (146.25,586.1) and (187.46,586.05) .. (229.04,529.03) ;
\draw [shift={(229.66,528.16)}, rotate = 125.83] [color={rgb, 255:red, 0; green, 0; blue, 0 }  ][line width=0.75]    (10.93,-3.29) .. controls (6.95,-1.4) and (3.31,-0.3) .. (0,0) .. controls (3.31,0.3) and (6.95,1.4) .. (10.93,3.29)   ;
\draw    (425.68,309.85) .. controls (470.68,237.89) and (515.98,237.53) .. (561.58,308.77) ;
\draw [shift={(562.26,309.85)}, rotate = 237.64] [color={rgb, 255:red, 0; green, 0; blue, 0 }  ][line width=0.75]    (10.93,-3.29) .. controls (6.95,-1.4) and (3.31,-0.3) .. (0,0) .. controls (3.31,0.3) and (6.95,1.4) .. (10.93,3.29)   ;
\draw    (426.54,335.85) .. controls (473.49,400.45) and (518.42,400.77) .. (561.33,336.81) ;
\draw [shift={(561.97,335.85)}, rotate = 123.58] [color={rgb, 255:red, 0; green, 0; blue, 0 }  ][line width=0.75]    (10.93,-3.29) .. controls (6.95,-1.4) and (3.31,-0.3) .. (0,0) .. controls (3.31,0.3) and (6.95,1.4) .. (10.93,3.29)   ;
\draw    (426.44,501.2) .. controls (471.74,428.44) and (517.9,428.07) .. (564.88,500.11) ;
\draw [shift={(565.59,501.2)}, rotate = 237.15] [color={rgb, 255:red, 0; green, 0; blue, 0 }  ][line width=0.75]    (10.93,-3.29) .. controls (6.95,-1.4) and (3.31,-0.3) .. (0,0) .. controls (3.31,0.3) and (6.95,1.4) .. (10.93,3.29)   ;
\draw    (426.67,527.2) .. controls (472.49,597.72) and (518.54,598.07) .. (564.8,528.26) ;
\draw [shift={(565.49,527.2)}, rotate = 123.27] [color={rgb, 255:red, 0; green, 0; blue, 0 }  ][line width=0.75]    (10.93,-3.29) .. controls (6.95,-1.4) and (3.31,-0.3) .. (0,0) .. controls (3.31,0.3) and (6.95,1.4) .. (10.93,3.29)   ;

\end{tikzpicture}

    \caption{The commutativity of this diagram means that $s: \mathfrak{E} \to \mathfrak{B}$ respects the vertical composition: $s(P_{\mu} \circ P_{\nu}) = s(P_{\mu}) \circ s(P_{\nu})$.}
    \label{fig:section_vert-composition} 
\end{figure}
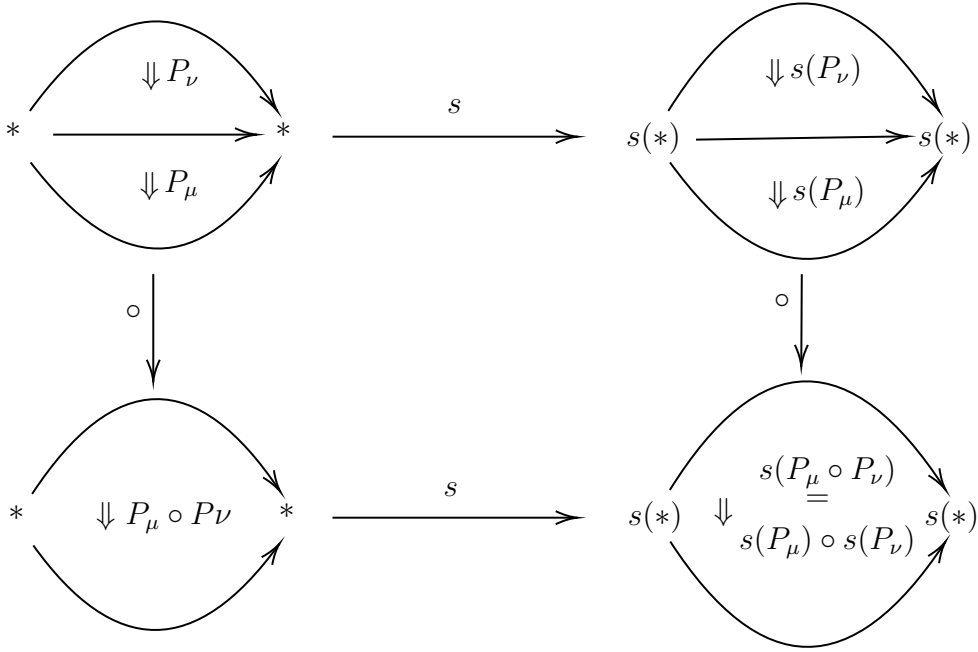

The same holds for $\mathfrak{p}$ as well. This suggests that it may be possible to understand both $\mathfrak{p}$ and $s$ as $2$-functors. It is not difficult to see that from the perspective of diagrammatical calculus of (higher) categories, we have added an extra \textit{normal} direction and a composition along it, giving a three-dimensional diagram. This is an important observation since the totality of $2$-category is a $3$-category (cf ~\cite{leinster2001survey}). Definition of $\mathfrak{p}$ as in \cref{rema:section} owes to the fact that while defining objects in the $1$-category $\textit{Bord}_2^{def}(\mathcal{D})$ as a disjoint union of several circles with defects we also introduced many copies of $\{\ast\}$ at level-$0$ (\cref{fig:level-0}). This suggests the following generalizations.

We begin with the convention that the term bundle also takes account of the term covering. Next, to define a (fiber) bundle, one can first define what a trivial bundle is, and then, with a concept of locality available, one can impose the triviality condition only locally. By analogy, we can think of a bundle as consisting of two data: triviality and locality, and proceed to define what a trivial bundle is in the category \textbf{Set} instead of \textbf{Top}. 

\begin{defi} \label{defi:set_bundle}

In the category $\textbf{Set}$, a bundle is a surjective function, \begin{tikzcd}
E \arrow[r, "\pi"] & B \end{tikzcd}, such that for every $b \in B$, there exists a set $F_b$ called the \textit{fiber} at $b$ such that the set $\pi^{-1}(b)$ is in bijection with the set $\{b\} \times F_b$.

A section of \begin{tikzcd}
E \arrow[r, "\pi"] & B \end{tikzcd} is a function  \begin{tikzcd}
B \arrow[r, "s"] & E \end{tikzcd} satisfying $\pi \circ s = 1$.

\end{defi}

Next, we recall the definition of $n$-category given inductively in terms of enrichment. 

\begin{defi}[~\cite{cheng2004higher}, Definition 2-n] \label{defi:strict_n-cat}
A strict $n$-category consists of:
\begin{enumerate}
    \item A set $\mathcal{C}_0$ whose elements are called objects.
    \item For any two objects $\alpha, \beta \in \mathcal{C}_0$, a $(n-1)$ category $\mathcal{C}(\alpha, \beta)$.
    \item For three objects $\alpha, \beta, \gamma \in C_0$, an $(n-1)$-functor $$ \mathcal{C}(\alpha, \beta) \otimes \mathcal{C}(\beta, \gamma) \to \mathcal{C}(\alpha, \gamma) $$
    \item For every object $\alpha \in \mathcal{C}_0$, an $(n-1)$-functor $$ 1 \to \mathcal{C}(\alpha, \alpha) $$ where $1$ is the terminal $(n-1)$-category,
\end{enumerate}
    satisfying associativity, unit, and interchange axioms.
\end{defi}

Now, we can define a bundle inductively at the level of $n$-category.

\begin{defi}
    A \textit{bundle at the level of} $n$-\textit{category} is an $n$-functor\begin{tikzcd}
\mathcal{E}_n \arrow[r, "\mathfrak{P}_n"] & \mathcal{B}_n  , \end{tikzcd} built inductively:
\begin{itemize}
    \item (\textbf{Base case:}) \begin{tikzcd}
\mathcal{E}_0 \arrow[r, "\mathfrak{P}_0"] & \mathcal{B}_0 \end{tikzcd} on objects is given as $\pi$ in \cref{defi:set_bundle}.
     \item (\textbf{Induction step:}) Assuming we know how to build a bundle at $(n-1)$ category, for $p, q \in \mathcal{E}_0$, it is given by a family of bundles at the level of $(n-1)$ category \begin{tikzcd}
\mathcal{E}(\alpha, \beta) \arrow[r, "\mathfrak{P}_{n-1}^{\alpha\beta}"] & \mathcal{B}(\mathfrak{p}_0(\alpha), \mathfrak{p}_0(\beta) ) \end{tikzcd} such that for three objects $\alpha, \beta, \gamma \in \mathcal{E}_0$, there is a natural isomorphism $\mathfrak{P}_{n-1}^{\alpha\beta} \otimes \mathfrak{P}_{n-1}^{\beta\gamma} \to \mathfrak{P}_{n-1}^{\alpha\gamma}$.
\item For an object $\alpha \in \mathcal{E}_0$, $\mathfrak{P}_{n-1}^{\alpha\alpha}$ maps $1_{\alpha}$, the image of $1$ in $\mathcal{C}(\alpha, \alpha)$ under the $(n-1)$ functor of \cref{defi:strict_n-cat}, (4), to $1_{\mathfrak{p}_0(\alpha)}.$

\end{itemize}

\end{defi}

By definition of an $n$-functor in enriched category theory, $\mathfrak{p}_n$ respects associativity, units, and the interchange law. Next, recursively construct bundles at the level of $n$-category for $n=1$ and $n=2$.

\begin{cons} \label{cons: bundle_1-cat}
Here, we build a bundle at the level of $1$-category. So let, \begin{tikzcd}
\mathcal{E}_1 \arrow[r, "\mathfrak{P}_1"] & \mathcal{B}_1 \end{tikzcd} be such a bundle. For two objects $p,q \in \mathcal{E}_0 $, $ \mathcal{E}_1(p,q)$ is a set. Therefore, the family \begin{tikzcd}
\mathcal{E}_1(p, q) \arrow[r, "\mathfrak{P}_{1}^{pq}"] & \mathcal{B}(\mathfrak{p}_0(p), \mathfrak{p}_0(q) ) \end{tikzcd}
is a collection of functions, each given by \cref{defi:set_bundle}.
\end{cons}

\begin{exam}
\cref{fig:bundle_1-cat} shows an example of a bundle at the level of $1$-category. The function $\mathfrak{p}_0$ maps $p_1, p_2, p_3 \in \mathcal{E}_0$ to $a \in \mathcal{B}_0$, and $q \in \mathcal{E}_0$ to $b \in \mathcal{B}_0$. The function $\mathfrak{p}_0^{p_1q}$ maps $l_1, l_2 \in \mathcal{E}_1(p_1,q)$ to $x \in \mathcal{B}_1(a, b)$, the function $\mathfrak{p}_0^{p_2q}$ maps $n \in \mathcal{E}_1(p_2,q)$ to $x \in \mathcal{B}_1(a, b)$, and the function $\mathfrak{p}_0^{p_3q}$ maps $m \in \mathcal{E}_1(p_3,q)$ to $x \in \mathcal{B}_1(a, b)$. Note that the fiber of $\mathfrak{p}_0$ over $a$ differs from that over $b$.

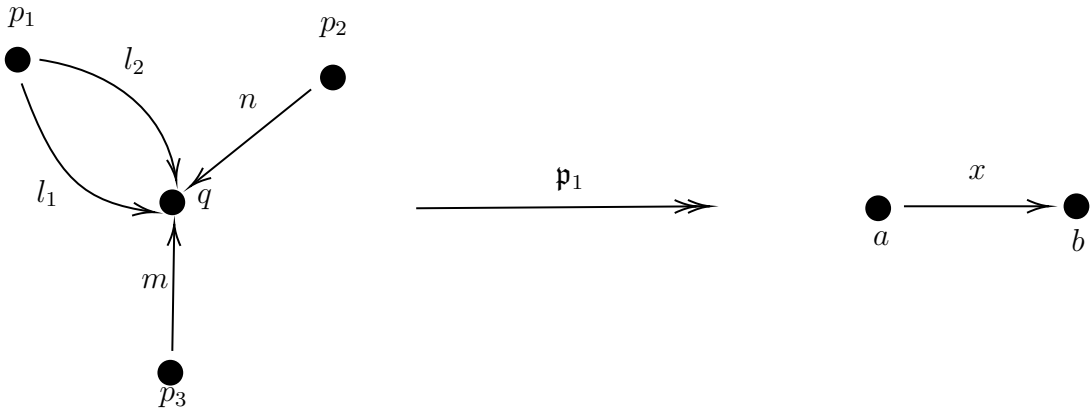
\begin{figure}[ht]
    \centering

\tikzset{every picture/.style={line width=0.75pt}} 

\begin{tikzpicture}[x=0.75pt,y=0.75pt,yscale=-1,xscale=1]

\draw    (53,317) .. controls (86.49,321.93) and (116.1,340.43) .. (121.76,376.35) ;
\draw [shift={(122,378)}, rotate = 262.3] [color={rgb, 255:red, 0; green, 0; blue, 0 }  ][line width=0.75]    (10.93,-3.29) .. controls (6.95,-1.4) and (3.31,-0.3) .. (0,0) .. controls (3.31,0.3) and (6.95,1.4) .. (10.93,3.29)   ;
\draw    (44,329) .. controls (61.73,378.25) and (75.58,388.69) .. (109.44,393.77) ;
\draw [shift={(111,394)}, rotate = 188.13] [color={rgb, 255:red, 0; green, 0; blue, 0 }  ][line width=0.75]    (10.93,-3.29) .. controls (6.95,-1.4) and (3.31,-0.3) .. (0,0) .. controls (3.31,0.3) and (6.95,1.4) .. (10.93,3.29)   ;
\draw    (190,332) -- (130.56,379.75) ;
\draw [shift={(129,381)}, rotate = 321.23] [color={rgb, 255:red, 0; green, 0; blue, 0 }  ][line width=0.75]    (10.93,-3.29) .. controls (6.95,-1.4) and (3.31,-0.3) .. (0,0) .. controls (3.31,0.3) and (6.95,1.4) .. (10.93,3.29)   ;
\draw    (120,464) -- (120.97,402) ;
\draw [shift={(121,400)}, rotate = 90.9] [color={rgb, 255:red, 0; green, 0; blue, 0 }  ][line width=0.75]    (10.93,-3.29) .. controls (6.95,-1.4) and (3.31,-0.3) .. (0,0) .. controls (3.31,0.3) and (6.95,1.4) .. (10.93,3.29)   ;
\draw    (489,391) -- (560,391) ;
\draw [shift={(562,391)}, rotate = 180] [color={rgb, 255:red, 0; green, 0; blue, 0 }  ][line width=0.75]    (10.93,-3.29) .. controls (6.95,-1.4) and (3.31,-0.3) .. (0,0) .. controls (3.31,0.3) and (6.95,1.4) .. (10.93,3.29)   ;
\draw    (243,392) -- (391,391) ;
\draw [shift={(391,391)}, rotate = 179.61] [color={rgb, 255:red, 0; green, 0; blue, 0 }  ][line width=0.75]    (17.64,-3.29) .. controls (13.66,-1.4) and (10.02,-0.3) .. (6.71,0) .. controls (10.02,0.3) and (13.66,1.4) .. (17.64,3.29)(10.93,-3.29) .. controls (6.95,-1.4) and (3.31,-0.3) .. (0,0) .. controls (3.31,0.3) and (6.95,1.4) .. (10.93,3.29)   ;
\draw  [fill={rgb, 255:red, 0; green, 0; blue, 0 }  ,fill opacity=1 ] (36,317) .. controls (36,313.69) and (38.69,311) .. (42,311) .. controls (45.31,311) and (48,313.69) .. (48,317) .. controls (48,320.31) and (45.31,323) .. (42,323) .. controls (38.69,323) and (36,320.31) .. (36,317) -- cycle ;
\draw  [fill={rgb, 255:red, 0; green, 0; blue, 0 }  ,fill opacity=1 ] (114,389) .. controls (114,385.69) and (116.69,383) .. (120,383) .. controls (123.31,383) and (126,385.69) .. (126,389) .. controls (126,392.31) and (123.31,395) .. (120,395) .. controls (116.69,395) and (114,392.31) .. (114,389) -- cycle ;
\draw  [fill={rgb, 255:red, 0; green, 0; blue, 0 }  ,fill opacity=1 ] (195,326) .. controls (195,322.69) and (197.69,320) .. (201,320) .. controls (204.31,320) and (207,322.69) .. (207,326) .. controls (207,329.31) and (204.31,332) .. (201,332) .. controls (197.69,332) and (195,329.31) .. (195,326) -- cycle ;
\draw  [fill={rgb, 255:red, 0; green, 0; blue, 0 }  ,fill opacity=1 ] (113,475) .. controls (113,471.69) and (115.69,469) .. (119,469) .. controls (122.31,469) and (125,471.69) .. (125,475) .. controls (125,478.31) and (122.31,481) .. (119,481) .. controls (115.69,481) and (113,478.31) .. (113,475) -- cycle ;
\draw  [fill={rgb, 255:red, 0; green, 0; blue, 0 }  ,fill opacity=1 ] (470,392) .. controls (470,388.69) and (472.69,386) .. (476,386) .. controls (479.31,386) and (482,388.69) .. (482,392) .. controls (482,395.31) and (479.31,398) .. (476,398) .. controls (472.69,398) and (470,395.31) .. (470,392) -- cycle ;
\draw  [fill={rgb, 255:red, 0; green, 0; blue, 0 }  ,fill opacity=1 ] (570,391) .. controls (570,387.69) and (572.69,385) .. (576,385) .. controls (579.31,385) and (582,387.69) .. (582,391) .. controls (582,394.31) and (579.31,397) .. (576,397) .. controls (572.69,397) and (570,394.31) .. (570,391) -- cycle ;

\draw (36,289) node [anchor=north west][inner sep=0.75pt]   [align=left] {$\displaystyle p_{1}$};
\draw (193,294) node [anchor=north west][inner sep=0.75pt]   [align=left] {$\displaystyle p_{2}$};
\draw (112,481) node [anchor=north west][inner sep=0.75pt]   [align=left] {$\displaystyle p_{3}$};
\draw (131,381) node [anchor=north west][inner sep=0.75pt]   [align=left] {$\displaystyle q$};
\draw (50,376) node [anchor=north west][inner sep=0.75pt]   [align=left] {$\displaystyle l_{1}$};
\draw (94,309) node [anchor=north west][inner sep=0.75pt]   [align=left] {$\displaystyle l_{2}$};
\draw (103,424) node [anchor=north west][inner sep=0.75pt]   [align=left] {$\displaystyle m$};
\draw (152,333) node [anchor=north west][inner sep=0.75pt]   [align=left] {$\displaystyle n$};
\draw (472,402) node [anchor=north west][inner sep=0.75pt]   [align=left] {$\displaystyle a$};
\draw (572,401) node [anchor=north west][inner sep=0.75pt]   [align=left] {$\displaystyle b$};
\draw (520,369) node [anchor=north west][inner sep=0.75pt]   [align=left] {$\displaystyle x$};
\draw (312,370) node [anchor=north west][inner sep=0.75pt]   [align=left] {$\displaystyle \mathfrak{p}_{1}$};

\end{tikzpicture}

    \caption{Diagram on the left represents a category $\mathcal{E}_1$ with only four objects and only non-trivial morphisms are shown. The diagram on the right is of category $\mathcal{B}_1$ with only two objects.}
    \label{fig:bundle_1-cat}
\end{figure}
    
\end{exam}

\begin{cons} \label{cons:bundle_2-cat}
Now, with bundles at the level of $1$-category in hands, we build a bundle at the level of $2$-category. So let, \begin{tikzcd}
\mathcal{E}_2 \arrow[r, "\mathfrak{P}_2"] & \mathcal{B}_2 \end{tikzcd} be such a bundle. For two objects $\alpha,\beta \in \mathcal{E}_0$, $\mathcal{E}_2(\alpha,\beta)$ is a $1$-category. Therefore, the family \begin{tikzcd}
\mathcal{E}(a, b) \arrow[r, "\mathfrak{P}_{1}^{ab}"] & \mathcal{B}(\mathfrak{p}_0(a), \mathfrak{p}_0(b) ) \end{tikzcd}
is a collection of $1$-functors, each given by \cref{cons: bundle_1-cat}.
\end{cons}

\begin{defi}
    A section of a bundle at the level of $n$-category \begin{tikzcd}
\mathcal{E}_n \arrow[r, "\mathfrak{P}_n"] & \mathcal{B}_n \end{tikzcd} is another $n$-functor \begin{tikzcd}
\mathcal{B}_n \arrow[r, "\mathfrak{s}_n"] & \mathcal{E}_n \end{tikzcd}, satisfying $\mathfrak{p}_n \diamond \mathfrak{s}_n = \mathbf{1}$, where $\diamond$ stands for the composition of $n$-functors, and $\mathbf{1}$ for the identity $n$-functor.
\end{defi}

We conclude this section with the conjecture that there is an extended TFT approach to graph coloring where the coloring processing is given by sections of bundles at the level of $2$-category as defined in \cref{cons:bundle_2-cat}.

\subsection{Word Problem and defect-cobordism} \label{sec:word-problem_theory}

The ideas developed in \cref{sec:groups}, in particular \cref{associated-def}, allow us to interpret the word problem as a co-bordism problem in the category $\textit{Bord}_2^{def}(\mathcal{P}_G)$. Note that to read a word as an element of the group $G$, we must fix an orientation and a base-point on the defect-circle. First, objects in the category $\textit{Bord}_2^{def}(\mathcal{P}_G)$ come with orientation, this is not an issue. Next, one can choose the distinguished point \{-1\} as a base point. Note, that a different choice of base-point changes the word by a conjugate, which is not necessarily bad since functions like trace are known to be independent of conjugation. 

However, for this manuscript, we fix the base-point as the distinguished point $\{-1\}$, and all the circles in the category $\textit{Bord}_2^{def}(\mathcal{P}_G)$ is given the standard orientation of the plane, We employ the following conventions:

\begin{conv} \label{conv:word_reading}
    \begin{enumerate}
        \item A single circle with defects, labeled by the generators of $P_G$ represents a word read from $\{-1\}$ to $\{-1\}$ following the standard orientation of the circle. (See the caption for \cref{fig:word-problem}, Diagram (i).)
        \item Recall, an object in $\textit{Bord}_2^{def}(\mathcal{P}_G)$ is a disjoint union of circles. Each of them represents a word as described in the previous point. To multiply them as elements of the group $G$, we use a pair of pants, as shown in \cref{fig:word-problem}.
        \item A single circle with defects reading the word $w$ will be denoted by $[w]$, i.e., while $w$ is an element of the group $G$, $[w]$ is an object in the category $\textit{Bord}_2^{def}(\mathcal{P}_G)$.
    \end{enumerate}
\end{conv}

\begin{rema} \label{rema:co-pants_factorization}

A consequence of \cref{conv:word_reading}, (2), is that co-pants can be used for the factorization of a word.
    
\end{rema}

\begin{rema}

This is worth noticing that \cref{defn:def-cat}, (2), also allows a permutation of $\hat{\mathbb{S}}^1$ components of a disjoint union as a morphism, which gives us a symmetric structure. In the light of \cref{conv:word_reading}, (2), we see that a transposition like $[w_1] \sqcup [w_2] \to [w_2] \sqcup [w_1]$ amounts to $[w_1w_2] = [w_2w_1]$, which is not true always. Therefore, to discuss the word problem we have to sacrifice the symmetric structure. Therefore, in the context of the word problem, we assume that the category $\textit{Bord}_2^{def}(\mathcal{P}_G)$ is not symmetric unless $G$ is abelian.

\end{rema}

\begin{figure}[ht]
    \centering
    \includegraphics[width=0.98\linewidth]{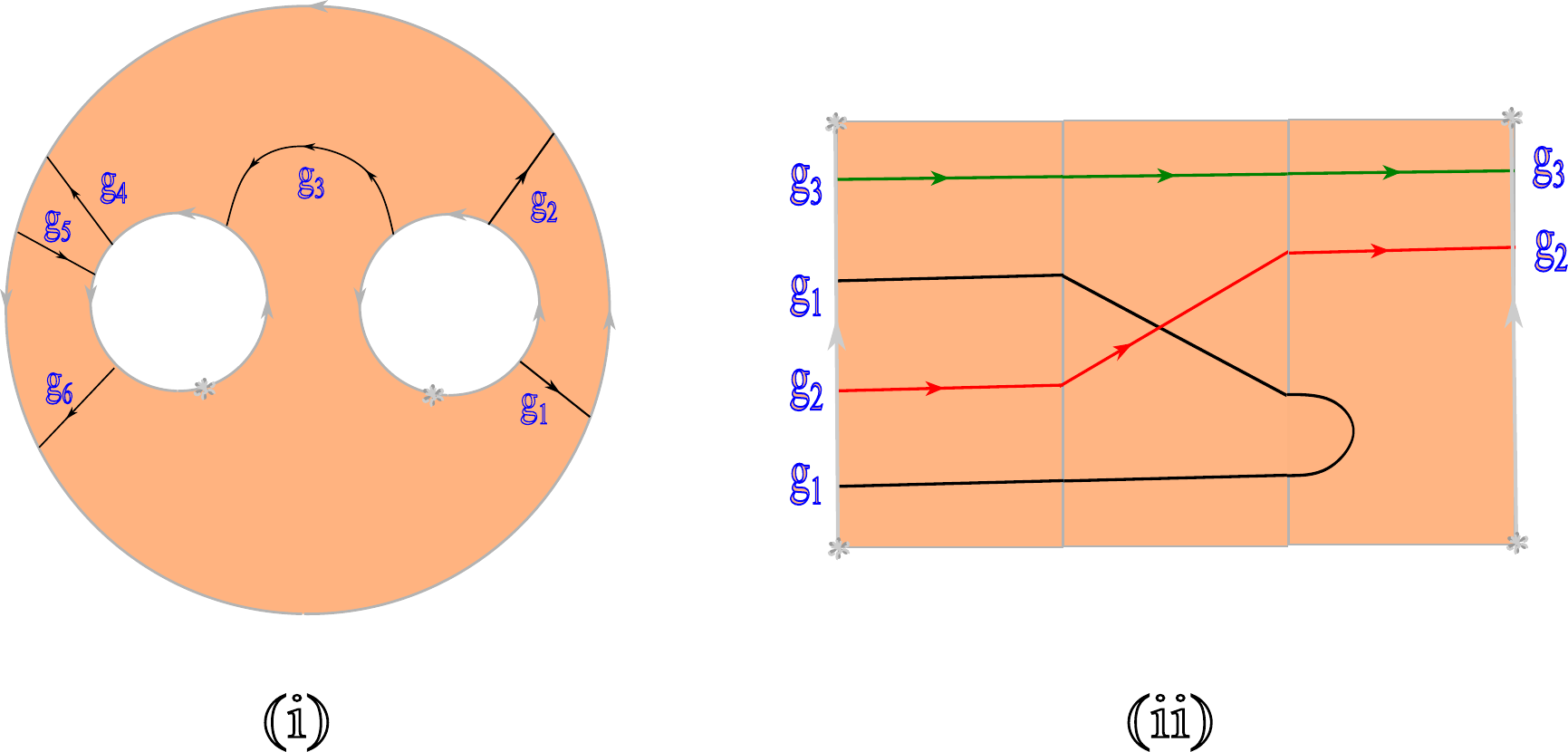}
    \caption{The figure (i) shows an example of multiplying two group elements when interpreting them as an object in the category $\textit{Bord}_2^{def}(\mathcal{P}_G)$. In this case, we see $(g_1g_2g_3)(g_3^{-1}g_4g_5^{-1}g_6) = g_1g_2g_4g_5^{-1}g_6$. Note that the induced orientation on the two inner circles is opposite to the standard orientation, but we always follow the standard orientation - the orientation that we chose on defect-circles as objects in the category $\textit{Bord}_2^{def}(\mathcal{P}_G)$. Diagram (ii) shows the word relation discussed in \cref{exam:word-problem}. Here, we have located the point $\{-1\}$ by placing the only element of $D_2$ directly over it, so its placement is specified by $\ast$. We follow this convention throughout this section. }
    \label{fig:word-problem}
\end{figure}

\begin{exam} \label{exam:word-problem}
    Let $G = \langle g_1, g_2, g_3 \mid g_1^2, g_1g_2g_1^{-1}g_2^{-1} \rangle$. \cref{fig:word-problem} shows the bordism verion of the relation $g_1g_2g_1g_3 =g_2g_3$.
\end{exam}


\cref{fig:word-problem} shows a planar picture for the bordism version of the word-problem mentioned in \cref{exam:word-problem}. However, a field-theoretic description can be given when one interprets a word problem as a cobordism problem. First, we need to explore the point \cref{local-junctions}, (3), in more detail. Defect-disks, as in \cref{fig:defect_orient}, (i) and (ii), can be thought of both as a cup, i.e., a bordism from the $\emptyset$, and also as a cap, i.e., a bordism to the $\emptyset$. This is demonstrated in \cref{fig:cup-cap_detail} where the defect-disk (iii) in \cref{fig:plcwgrp} is considered.  

\begin{figure}
    \centering
    \includegraphics[width=0.9 \linewidth]{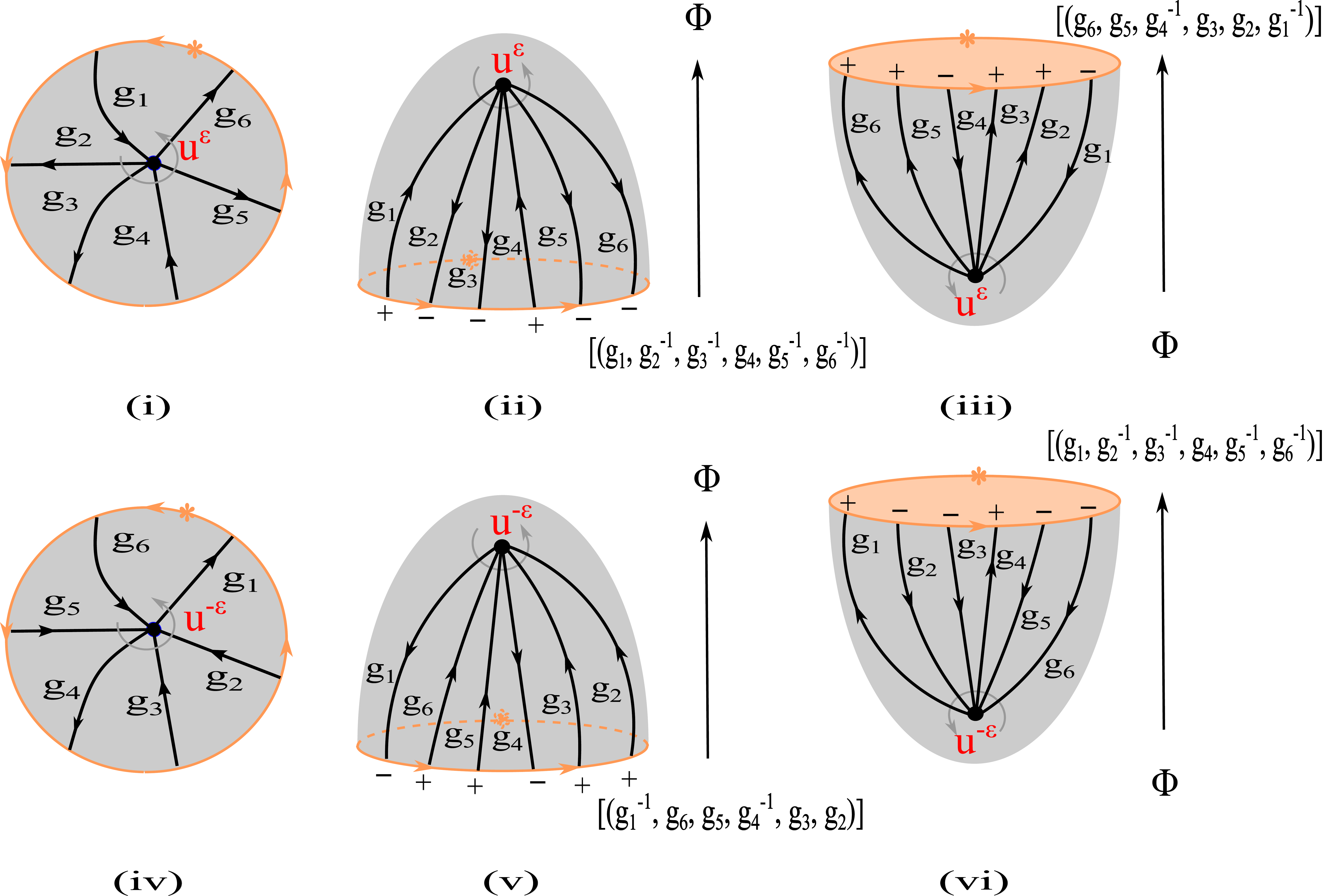}
    \caption{Diagram (i) shows the defect-disk for $\psi_{\{0,1\}}(u^{\epsilon}) = [(g_1, g_2^{-1}, g_3^{-1}, g_4, g_5^{-1}, g_6^{-1})]$. referring to the convention in the caption of \cref{Bord_def}, we see that if this defect-disk is treated as a cap then the boundary circle is an in-boundary and the corresponding word is in the class $[(g_1, g_2^{-1}, g_3^{-1}, g_4, g_5^{-1}, g_6^{-1})]$ as in Diagram (ii)). However, if it is seen as a cup then the boundary circle is the out-boundary and the word determined by the standard orientation is in the class $[(g_6, g_5, g_4^{-1}, g_3, g_2, g_1^{-1})]$ as in Diagram (iii)). Similarly, Diagram (iv), shows a defect-disk for the map $\psi_{\{0,1\}}(u^{-\epsilon}) = [(g_1, g_2^{-1}, g_3^{-1}, g_4, g_5^{-1}, g_6^{-1})]$. The cap version of this map is shown in Diagram (v), the word determined by the boundary circle is in the class $[(g_6, g_5, g_4^{-1}, g_3, g_2, g_1^{-1})]$. While Diagram (vi) shows the cup with the word in the boundary class $[(g_1, g_2^{-1}, g_3^{-1}, g_4, g_5^{-1}, g_6^{-1})]$.}
    \label{fig:cup-cap_detail}
\end{figure}

Note in \cref{fig:cup-cap_detail}, the cap (cup) for $\psi_{\{0,1\}}(u^{\epsilon})$ has identical boundary with the cup (cap) for $\psi_{\{0,1\}}(u^{-\epsilon})$, which facilitate their gluing as mentioned in \cref{local-junctions}, (3). One can also think of involution as reflecting in the horizontal plane. This will change the orientation of the surface, and also the left and right. Consequently, one needs to flip the direction of arrows on the $1$-strata to be consistent with the map $\psi_{1,2}$. This is the relation that we see in \cref{fig:cup-cap_detail} between Diagram (ii) and (vi), and between Diagram (iii) and (v).


Let us look at the following example:

\begin{exam} \label{exam:big_word-problem}
    Consider the group $G$, given by the presentation $P_G = \langle a, b, c, d \mid ab^{-1}ac \rangle $. In this group, the word $aca$ equals $b$, as this can be seen by the following chain of transformations:
    \begin{multline} \label{multline:word_problem}
        aca \longrightarrow (ac)(c^{-1}a^{-1}ba^{-1})a \longrightarrow (a(cc^{-1})a^{-1})b(a^{-1}a) \longrightarrow b
    \end{multline}
    
\end{exam}

We take a moment to decipher \cref{multline:word_problem}. The process of inserting the relation $c^{-1}a^{-1}ba^{-1}$ between $a$ and $c$ can be thought of as the creation of the word $c^{-1}a^{-1}ba^{-1}$, and is indeed given by the creation operator, the cup bordism, when one interprets this as a cobordism problem in $\textit{Bord}_2^{def}(\mathcal{P}_G)$. \cref{fig:word_bordism} shows \cref{multline:word_problem} as a cobordism between $[aca]$ and $[b]$ in $\textit{Bord}_2^{def}(\mathcal{P}_G)$.

\begin{figure}
    \centering
    \includegraphics[width=420pt, height=500pt]{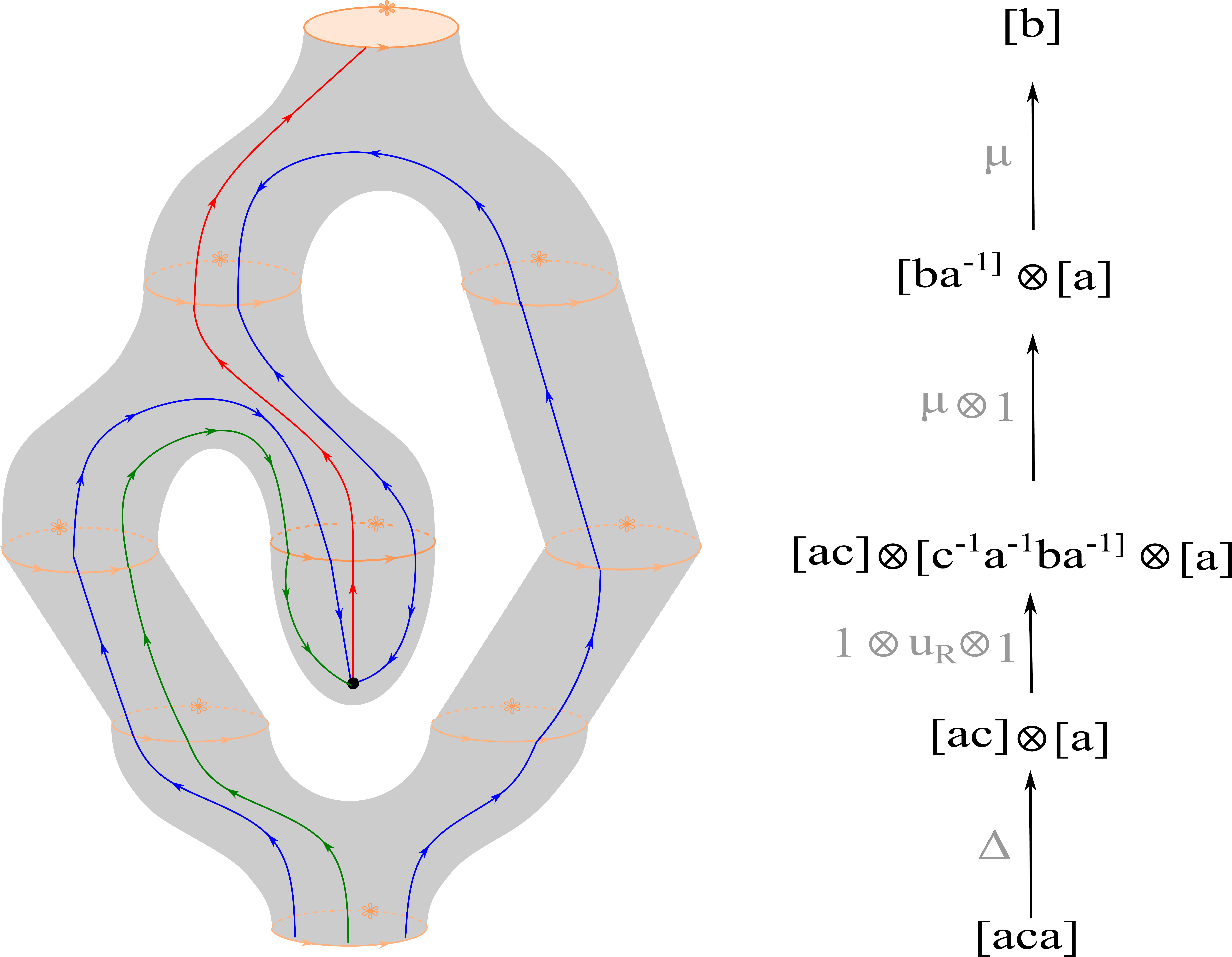}
    \caption{The cobordism version of \cref{multline:word_problem}. Note, in Group Theory, the trivial relation $gg^{-1}$ is always understood as one always wants a reduced word. In the case of surfaces with defects, this is just the continuity relation: if a curve enters a small region, it must leave. See \cref{exam:big_word-problem} for details of this diagram.}
    \label{fig:word_bordism}
\end{figure}

Generally, when solving a word problem, one often replaces a word (say) $w$ with another word $w'$. This can be done provided $w = w'$, but that amounts to $e = w^{-1}w' = w'w^{-1}$. Therefore, this process is the same as inserting the relation $w^{-1}w'$ to the right of $w$ or $w'w^{-1}$ to the left of $w$. Therefore, at each stage, solving a word problem can be thought of as inserting a word given by some relation or removing a word that determines a relation. Therefore, we state the following:

\begin{theo} \label{thm:word_cobordism}

     Given a group $G$ and a presentation $\mathcal{P}_G$, form the category \linebreak
    $\textit{Bord}_2^{def}(\mathcal{P}_G)$. If two words $w_1$ and $w_2$ represent the same element of the group $G$, then any two circles with defects, $[w_1]$ and $[w_2]$, are cobordant in the category $\textit{Bord}_2^{def}(\mathcal{P}_G)$. 
    
\end{theo}

We believe that the converse of \cref{thm:word_cobordism} is also true, i.e., if two circles with defects representing words $w_1$ and $w_2$ are cobordant in $\textit{Bord}_2^{def}(\mathcal{P}_G)$ then the two words represent the same element of the group $G$. A version of Morse Theory can be used to prove this, provided some questions related to the category $\textit{Bord}_2^{def}(\mathcal{P}_G)$ are addressed (see \cref{sec:word_problem_future}.)

The following example demonstrates an obstruction in solving a word problem in terms of tools developed in this manuscript.

\begin{exam} \label{exam:planar_word-problem}

    It is clear that in the group $K_4$, the words $c$ and $aa$ are not the same. Let us see this as a cobordism problem. We claim that there is no planar cobordism between them in the category $\textit{Bord}_2^{def}(\mathcal{K}_4)$, i.e., $c$ and $aa$ are not cobordant via a cylinder. Suppose there exists such a cobordism, then glue the defect-disk in $\textit{Bord}_2^{def}(\mathcal{P}_G)$ (\cref{fig:colored}, top left) to the boundary circle $aa$ as in \cref{fig:planar_word-problem}, Diagram (i). Take the involution of this composite surface and glue it along the common boundary (see \cref{fig:planar_word-problem}, Diagram (ii)). This gives us a $3$-edge colored planar trivalent graph with a generic cross-section containing a single defect $c$ (see \cref{fig:planar_word-problem}, Diagram (iii). This is a contradiction to \cref{big-1}.
    
\end{exam}

\begin{figure}
    \centering
    \includegraphics[width=0.8\linewidth]{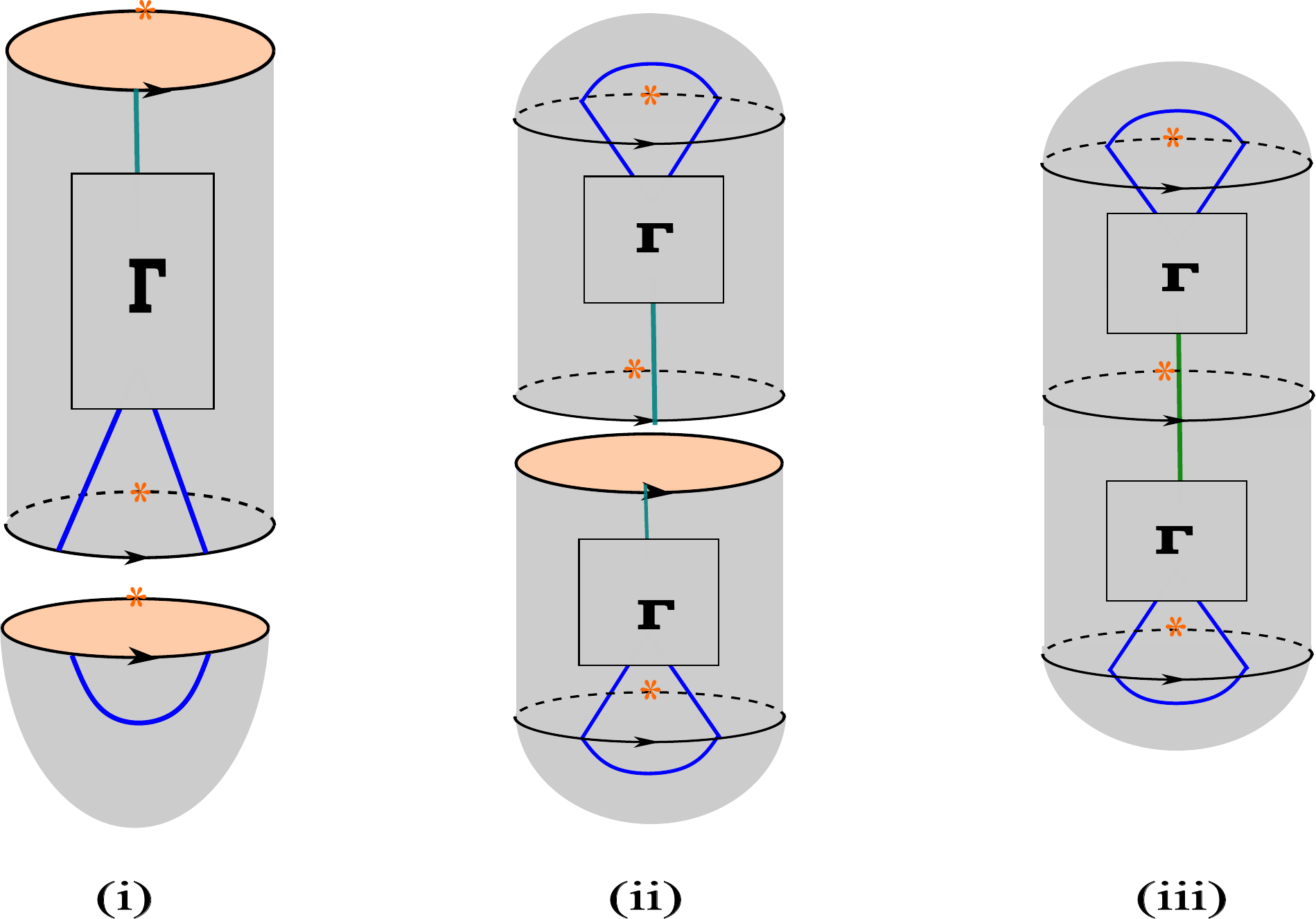}
    \caption{See \cref{exam:planar_word-problem} for details. }
    \label{fig:planar_word-problem}
\end{figure}

The construction used in \cref{exam:planar_word-problem} can be generalized if there exists a functor $T: \textit{Bord}_2^{def}(\mathcal{P}_G) \to \mathcal{C}$ to monoidal category. For demonstration, we can take $\mathcal{C}$ to be $\text{Vect}_{\mathbb{K}}.$ Let, $w$ be a word, that is equal to the identity as an element of $G$, then there exists a cup $\hat{M}$ as in \cref{fig:bordism_future}, Diagram (i), with the out-boundary $[w]$. If $\hat{M}^{\ast}$ denotes the involution of $\hat{M}$ (the cap in \cref{fig:bordism_future}), then it can be glued along the common-boundary $[w]$ to produce a closed surface (with defects) $\hat{M} \cup_{[w]} \hat{M}^{\ast}$ in $\textit{Bord}_2^{def}(\mathcal{P}_G)$. Thus $T$ assigns a scalar to  $\hat{M} \cup_{[w]} \hat{M}^{\ast}$. This scalar denoted $T(\hat{M})$, is a characteristic of the bordism $\hat{M}: \emptyset \to [w]$. Thus, for a fixed word $w$, every bordism $\hat{M}: \emptyset \to [w]$ gives rise to a scalar $T(\hat{M})$. In general, two such scalars are not expected to be equal. \cref{fig:bordism_future} demonstrate this for a specific group given by a presentation.

\begin{figure}
    \centering
    \includegraphics[width=0.98\linewidth]{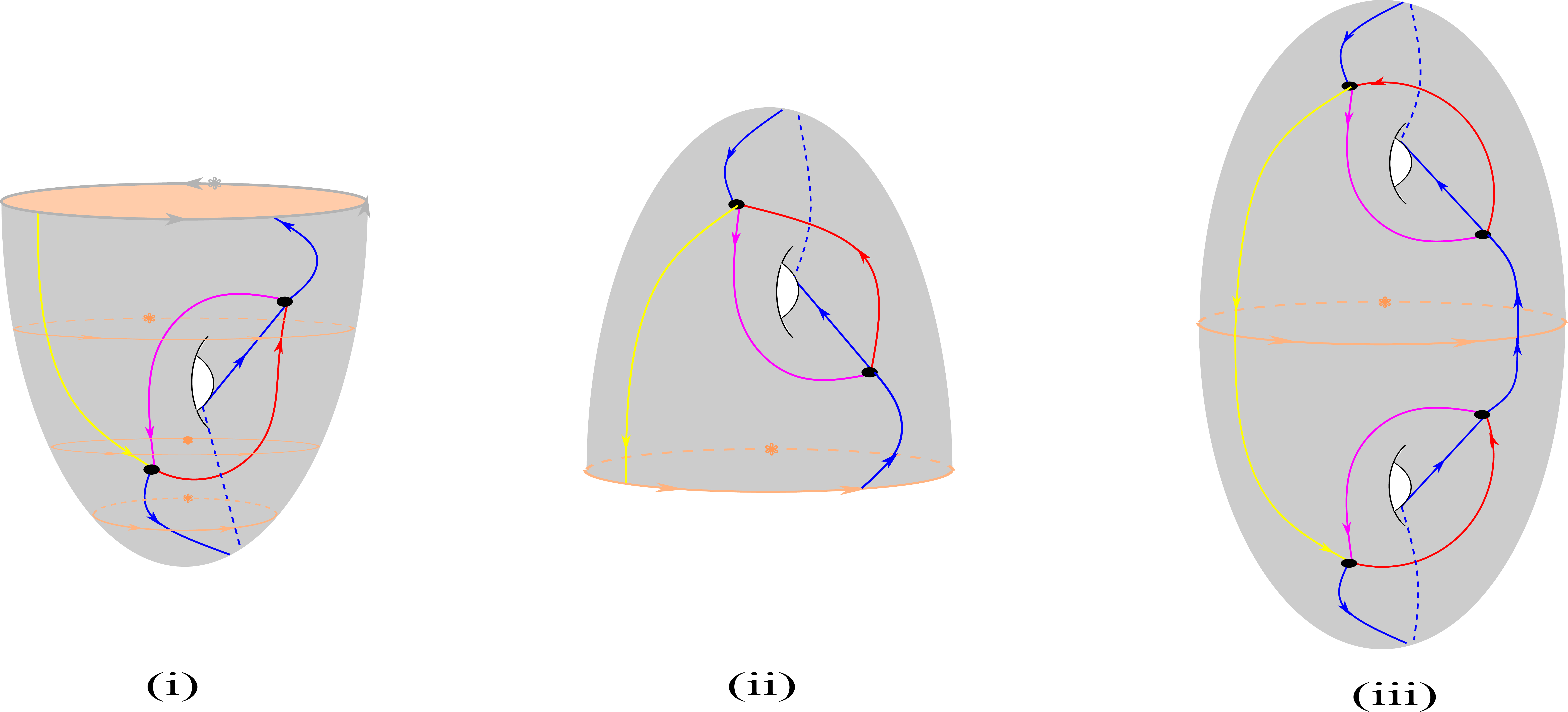}
    \caption{ Let $G$ be a group, given by the presentation $\langle a, b, d, e \mid a^{-1}b^{-1}ed, aba^{-1}b^{-1}, aba^{-1}e^{-1} \rangle$. With $a$ in blue, $b$ in red, $d$ in yellow, and $e$ in pink, Diagram (i) shows a bordism $\hat{H}:\emptyset \to (d^{-1}a)$ as described in \cref{eqn:bordism_future}. Diagram (ii) is the involution, showing a bordism $\hat{H}^{\ast}: (d^{-1}a) \to \emptyset $. Diagram (iii) is the closed surface obtained by gluing $\hat{H}$ and $\hat{H}^{\ast}$ along the common boundary $(d^{-1}a)$.}
    \label{fig:bordism_future}
\end{figure}

\begin{multline}  \label{eqn:bordism_future}
     e \xrightarrow[]{(a^{-1}a)} a^{-1}a \xrightarrow[]{ (d^{-1}e^{-1}ba)(a^{-1}a} d^{-1}e^{-1}ba \\
     \xrightarrow[]{(d^{-1}b^{-1})(aba^{-1}b^{-1})(ab)}  d^{-1}e^{-1}ab \xrightarrow[]{(d^{-1}e^{-1}ab)(b^{-1}a^{-1}ea)}  d^{-1}a
\end{multline}

One can check that target words at each arrow in \cref{eqn:bordism_future} can be read by taking generic cross sections of \cref{fig:bordism_future}, Diagram (i), beginning from the bottom.

We conclude this section with the following dictionary:

\begin{table}
\centering
\begin{tabular}{ | m{200pt}| m{200pt} | } 
  \hline
  \textbf{Word Problem for finitely presented groups} & \textbf{cobordism Problem with defects} \\
  \hline
  Words. & Circle with defects. \\
  \hline
  Multiplication in the group & A pair of pants with defects. \\
  \hline
  Factorization & Co-pants with defects.\\
  \hline
  Inserting a word from a relation & Creation (cup-bordism) \\
  \hline
  Removing a word from a relation & Annihilation (cap-bordism) \\
  \hline
  When two words represent the same elements of the group. & When two circles with defects are $\text{cobordant}^{\ast \ast}$. \\ 
  \hline
\end{tabular}
\caption{The asterisk in the last row highlights the fact that this correspondence is only one way: left to right. The converse, right to left, should also be true but not covered in this project.}
\label{table:word_problem}
\end{table}

Where $\ast \ast$ shows that they are equivalent but their theory is not fully dealt with in this manuscript. See \cref{sec:word_problem_future}.


\section{Future direction and possible connections} \label{sec:future_dirn}

We conclude this manuscript with a list of possible projects that further explore the mathematics developed here. Some of the things that appear here are not rigorous and are mere speculations of the authors. We begin with an invitation for an alternate proof of \cref{main-2}:


\subsection{A generalization of the universal construction.} \label{sec:2-defect_Universal}

We recall the definition of the universal construction of topological theories from ~\cite{khovanov2020universal} [section-1]. Let $\text{Bord}_n$ be the category of oriented $n$-dimensional bordism, i.e., its objects are oriented closed $(n-1)$ dimensional manifolds and a morphism between two objects $N_1$ and $N_2$ is given by $n$-dimensional surface $M$ (up to diffeomorphism relative to the boundary) such that $\partial M = (- N_1) \sqcup N_1$. Suppose that there exists an invariant $\alpha$ of closed $n$-dimensional manifolds that assigns each such manifold $M$, an element $\alpha(M)$ in a commutative ring $R$ with unity and involution $\kappa : R \to R$, such that:

\begin{enumerate}
    \item $\alpha(M_1) = \alpha(M_2)$ whenever $M_1$ and $M_2$ are diffeomorphic.
    \item $\alpha(\emptyset_n) = 1$,
    \item $\alpha(M_1 \sqcup M_2) = \alpha(M_1)\alpha(M_2)$, and
    \item $\alpha(-M) = \kappa(\alpha(M))$
\end{enumerate}

Note that, a partition function of a topological field theory satisfies all these properties. So, it is not unnatural to ask if such an $\alpha$ can be a partition function of a topological field theory. Universal construction aims at producing a topological field theory whose partition function is $\alpha$ by defining the state space of a closed, oriented $(n-1)$-manifold $N$ as follows:

\begin{itemize}
    \item First, define a free $R$-module $\textit{Fr}(N)$ with the basis $[M]$, for all oriented $n$-manifold $M$ such that $\partial M = N$. Here, $[M]$ is just a formal symbol associated to $M$.
    \item Next, use $\alpha$ to obtain a $\kappa$-semilinear pairing $\langle , \rangle : \textit{Fr}(N) \times \textit{Fr}(N) \to R$ by $$ \langle M_1, M_2 \rangle \coloneqq \alpha(-M_1 \cup_N M_2) $$ for $\partial M_1 = \partial M_2 = N$, and extend by (semi)linearity: $$ \langle \sum a_i[M_i], \sum b_j[M'_j] \rangle = \sum \kappa(a_i)b_j \langle M_i, M'_j \rangle.$$
    \item Finally, define the state space $\alpha(N)$ as the quotient space of $\textit{Fr}(N)$ by the kernel of the semilinear form $( , )$, i.e., $\alpha(N) \coloneqq \textit{Fr}(N)/\text{ker} \langle , \rangle$.
    \item A cobordism $M: N_0 \to N_1$ induces a map $\alpha(M): \alpha(N_0) \to \alpha(N_1)$, which takes a generator $[M_0]$ of $\alpha(N_0)$ to $[MM_0]$ in $\alpha(N_1)$.

\end{itemize}

The idea is that if such a TFT exists then it assigns a vector to any $n$-manifold with the out-boundary $N$ (i.e. $M: \emptyset \to N$ as a morphism in $\text{Bord}_n$.) The state space $\alpha(N)$ has a basis as all these vectors. (But, since $\alpha$ takes value in a ring rather than in a field, this has to be modified.)

After this introduction, note that if we choose $\alpha$ to be the number of Tait-coloring of a planar trivalent graph, then it satisfies the property $\alpha(\Gamma_1 \sqcup \Gamma_2) = \alpha(\Gamma_1) \alpha(\Gamma_2)$. So, a natural question to ask is can one generalize this construction for defect TFTs where $\alpha$ assigns a state to a defect-interval (and not to just defect-circles)in such a way that the partition function of the theory for planar graphs is given by $\alpha$? We emphasize that there is already a version for 2-extended TFT given in ~\cite{khovanov2002functor} and next, it may be helpful to consider the correspondence between defects and extended TFT in
[~\cite{kapustin2010topological}, section-2.3]. Another reason this approach is appealing is the following general version of \cref{main-2}:

\begin{conj} \label{conj:gen-surface}
Let $\Gamma$ be a trivalent graph embedded in a surface $\Sigma$. Consider the surface with defect $(\Sigma, {\Gamma})$ in $\text{Mor}(\textit{Bord}_2^{\textit{def,cw}}(\mathcal{D}_{+}^{\mathbf{3}}))(\emptyset, \emptyset)$. The action of the functor $\chi^{cw}$ on $(\Sigma, \Gamma)$ is the assignment,
    \begin{equation}
        \begin{split}
            \chi^{cw}(\Sigma, \Gamma) &: \mathbb{C} \longrightarrow \mathbb{C}\\
            & \lambda \mapsto \#\text{Tait}(\Gamma) \lambda .
        \end{split}
    \end{equation}
    
\end{conj}

In other words, the number $\chi^{cw}(\Sigma, \Gamma)(1)$ is the number of Tait-coloring of the trivalent graph $\Gamma$, embedded in an arbitrary surface $\Sigma$. However, \cref{main-2} has been proved only for the planar case. This makes sense as edge-colorability is intrinsic to the graph and does not depend on the surface it is embedded in \cref{sec:motivation}. Extending the universal construction and proving that it is naturally isomorphic to $\chi^{cw}$ will prove \cref{conj:gen-surface}. However, a proof of \cref{conj:gen-surface} does not have to utilize this idea.


\subsection{Foam evaluation and trivial surrounding 3-d TFT.} \label{sec:3-defect_Universal}

It has been shown in \cite{carqueville20203} that the $3$-dimensional case of defect TFTs is related to Gray categories with duals in the same way as the $2$-dimensional defect TFTs are related to pivotal $2$-categories. A gray category, being a $3$-category, possesses three compositions that can be used to glue (fuse) diagrams in three independent directions to form a foam, which we want to see as a $3$-dimensional manifold with defects with $D_3 = \{\ast\}$. In connection with ~\cite{khovanov2021foam}, we see that a (pre)foam also allows its $2$-dimensional starts to accommodate dots, which plays a role in the definition of foam evaluation. In analogy with the approach taken here to define a surface with defects, it is natural to seek a definition of (pre)foam that uses  \textit{defect-balls} as a local model. Then, one can ask:
\enquote{Does there exist a trivial surrounding $3$-dimensional lattice TFT (suitably using $D_1, D_0$ and $\psi_{0,2}$)
which is connected to the foam evaluation?} 
To construct a lattice TFT, it may be useful to look at ideas developed in \cite{kapustin2014topological}.

Other than lattice TFT construction, the universal construction can be another way to produce TFT with defects in dimension three; especially, when a variation of the universal construction for foams already exists, and was used in [\cite{khovanov2004sl}] to categorify Kuperberg invariant of closed $A_2$ webs [~\cite{kuperberg1996spiders}]. This suggests that the universal construction can be extended to the category $\text{Bord}_3^{def}(\mathcal{D})$ (the case of  $\text{Bord}_2^{def}(\mathcal{D})$ has been discussed in \cref{sec:2-defect_Universal}.) Furthermore, producing a defect TFT in dimension three can be employed to produce invariants of $3$-manifolds via orbifold construction of defect TFTs [\cite{carqueville2023orbifolds}, \cite{carqueville2018orbifolds}, \cite{carqueville2019orbifolds}].


\subsection{Word problem, n-deformation, and Andrew-Curtis.} \label{sec:word_problem_future}

We begin from where we left in \cref{sec:word-problem_theory}. The first problem that comes when attempting to prove the converse of \cref{thm:word_cobordism} is the fate of the distinguished point $\{-1\}$ under bordism. Even if we only want to work with conjugacy classes, we would not want $1$-defects to wrap around the surfaces uncontrollably. It is not clear whether \cref{distinguished} is sufficient; one may want to impose more conditions in terms of the topology of the underlying surface. A possible remedy can be given in terms of \cref{fig:plcwgrp}, (iv) while staying in the planar world. Indeed, we see this in \cref{fig:bordism_future} that the commutator $aba^{-1}b^{-1}$ correspond to a handle up to stabilization (cf ~\cite{kauffman2021virtual}). Next, in \cref{fig:word-problem}, (i), we combined $g_3$ from both input words. This corresponds to the simple fact that we want a reduced word at the end. It is not clear whether one should allow the insertion of a relation of the kind $ww^{-1}$ when factoring a word. It is not difficult to see that a co-pants that creates a word like $ww^{-1}$ between two words $w_1$ and $w_2$, can be given as a composition of co-pants without any word like $ww^{-1}$, (so, the diagram is always progressive, i.e., the projection map to the time-coordinate is injective when restricted to each $1$-defects) a cup with the word $ww^{-1}$, and a pair of pants. That is true for a pair of pants too, and one can check that the pair of pants in \cref{fig:word-problem}, (i), can be given in terms of progressive pants and co-pants, and a cap. This is just a bordism version of the fact that one needs to first multiply and then reduce to get to a reduced word. It appears that cups, caps, progressive pants, and progressive co-pants form all the elementary cobordisms in $\textit{Bord}_2^{def}(\mathcal{P}_G)$. Therefore, using Morse Theory to deduce that a morphism in $\textit{Bord}_2^{def}(\mathcal{P}_G)$ can be written in terms of these elementary cobordisms would prove the converse of \cref{thm:word_cobordism}. Note that the converse of \cref{thm:word_cobordism} will allow us to prove the following generalized version of \cref{big-1}:

\enquote{If an object is in $\textit{Bord}_2^{def}(\mathcal{P}_G)$ is such that the word $w$ corresponding to it is not equal to the identity $e \in G$, then $[w]$ is not cobordant to $\emptyset$.}


In addition to those in connection with the word problem, there are other questions related to the category$\textit{Bord}_2^{def}(\mathcal{G})$. For instance, when the group $G$ is finite, one can also define the category $\textit{Bord}_2^{def}(\mathcal{G})$, which has $D_2 = \{\ast\}, D_1 = G$, and the set $D_0$ is in one-to-one correspondence with words whose product is the identity. The defect disks, in the distinguished neighborhood of $0$-strata labeled by elements of $D_0$, are modeled on \cref{fig:plcwgrp}, (iii), with these words. Now, we see issues regarding the category $\textit{Bord}_2^{def}(\mathcal{P}_G)$ that need to be addressed. For example, it is not unreasonable to enquire what happens to the category $\textit{Bord}_2^{def}(\mathcal{P}_G)$ under Tietze transformations (cf ~\cite{geoghegan2008fundamental}, Section 3.1), and how is this category related to the category$\textit{Bord}_2^{def}(\mathcal{G})$ when $G$ is a finite group? Taking inspiration from ~\cite{henry2022tietze}, it is reasonable to introduce some notion of weak equivalence between categories  $\textit{Bord}_2^{def}(\mathcal{P}_G)$ and $\textit{Bord}_2^{def}(\mathcal{Q}_G)$ for two presentations $P_G, Q_G$ of the same group $G$. However, interpreting a vertex as a $2$-morphism, this has to be a weak-equivalence of $2$-categories, which means that we step into the world of $3$-categories (cf ~\cite{leinster2001survey}, Introduction.) It is also interesting to note that when $G$ is a finite group, a construction similar to ours appears as $3$-defect conditions in ~\cite{carqueville20203}, Section 4.3 with $D_3 = \{\ast\}$.) However, it should be emphasized that $D_2$ and $D_0$ do not have to be trivial, and a general such theory will appear elsewhere. 

Finally, we note that the concept of $n$-deformations as introduced in ~\cite{wright1975group}, and the Andrew-Curtis Conjecture are based on the subject of a group presentation. So, it is not strange to seek how they are related to the category $\textit{Bord}_2^{def}(\mathcal{P}_G)$. More precisely, if we put them on the left-hand side of \cref{table:word_problem}, then what will go into the right-hand side in columns against them?


\subsection{N-graphs, N-triangulation, and orbifold construction.} The central ideas of this manuscript were discovered by the author when he was working on the ribbon graph formulation of N-graphs (cf ~\cite{casals2023legendrian}) in connection with Section-3 of ~\cite{treumann2016cubic}. The plan was to connect this work with special Legendrians in $\mathbb{S}^5$ (see ~\cite{wang2002compact}). However, one would need to allow to glue more than just disks to work with Legendrian weaves. There could be ways to do this using ideas developed in ~\cite{barazer2021cutting}, but for now, we follow this manuscript that is relevant to $N$-graphs. In our setup, an $N$-graph with only hexagonal vertices lives in the set $\text{Mor}(\textit{Bord}_2^{def, cw}(\mathcal{P}_{S_n}))$ where

   \begin{equation} \label{eqn:N-graph}
        S_n = \left\langle \tau_1, \dots , \tau_{n-1} \Bigg|
        \begin{array}{lr}
        \tau_i\tau_j = \tau_j\tau_i \hspace{3mm} \mid i - j \mid > 1\\
        \tau_i\tau_{i+1}\tau_i \hspace{2.5mm} = \hspace{2.5mm} \tau_{i+1}\tau_{i}\tau_{i+1} \\
          \hspace{10mm}\tau_i^2 \hspace{1mm} = \hspace{1mm} 1
        \end{array}\right\rangle .
  \end{equation}    

  It is unclear how to incorporate trivalent vertices. An $N$-graph which also has trivalent vertices sits inside a larger category. This category also allows basic-gons of type-III where all the $1$-strata are labeled by the same generator $\tau_j$ in \cref{eqn:N-graph}. However, looking at the connection explored in Appendix-A of ~\cite{casals2023legendrian}, there might be other indirect ways to accommodate trivalent vertices since the subject of Soergel bimodules is strongly linked to the Coxeter system (cf ~\cite{libedinsky2019gentle} and ~\cite{elias2020introduction}.)

  From the perspective of orbifold construction of a defect TFT (~\cite{carqueville2018orbifolds}, ~\cite{carqueville2019orbifolds}, ~\cite{carqueville2023orbifolds}), we note that an $N$-graph is dual to $N$-triangulation (cf ~\cite{fock2006moduli}, Section 1.15 and also ~\cite{casals2023legendrian}, Section 3) in the same way the dual to a triangulation is a single trivalent graph. The orbifold construction of defect TFT utilizes the dual of Pachner moves (~\cite{carqueville2023orbifolds}, Section 4). The analog of Pachner moves for $N$-triangulation is related to cluster mutation (~\cite{schrader2019cluster}.) Therefore it may be worthwhile to explore the subject of defect TFT for such defects, with appropriate $D_2$.


\subsection{Ribbon graph and graph connection.} \label{sec:ribbon_graph}

  The theory of the bordism category  $\textit{Bord}_2^{def}(\mathcal{P}_G)$ for a finitely presented group bears an uncanny resemblance with the subject of graph-connection (see, for instance, ~\cite{bourque2023flat}. Therefore, it is not unnatural to seek relationships between two subjects.

  Another closely related subject is that of ribbon-graphs. A ribbon-graph, like a surface with defects, is given by a cyclic ordering at each vertices, which dictates how disks should be glued to produce a surface (cf. ~\cite{mulase1998ribbon}). This provides a way to study defects in a non-planar surface, which can be very useful in studying the word problem. Here, we view the pair $(\Sigma, \Gamma)$, for a ribbon graph $\Gamma$ as a surface with defects, where $\Sigma$ is the surface determined by $\Gamma$. It is worth mentioning that this approach does not cover the entire theory of non-planar defects since there are surfaces with defects that do not come from a ribbon-graph. Moreover, after blowing up (cf. ~\cite{baldridge2023topological}) one can work exclusively with trivalent ribbon graphs, which allows one to work with twists as in ~\cite{baldridge2023topological} instead of the usual cyclic ordering at vertices; ~\cref{fig:ribbon_torus} gives an example of this. Note that for a trivalent graph $\Gamma$ the pair $(\Sigma, \Gamma)$ is a surface with defects. Considering a ribbon graph allows the ideas developed in this manuscript to extend beyond the plane, but their theory will appear elsewhere.

  \begin{figure}[ht]
      \centering
      \includegraphics[scale=0.2]{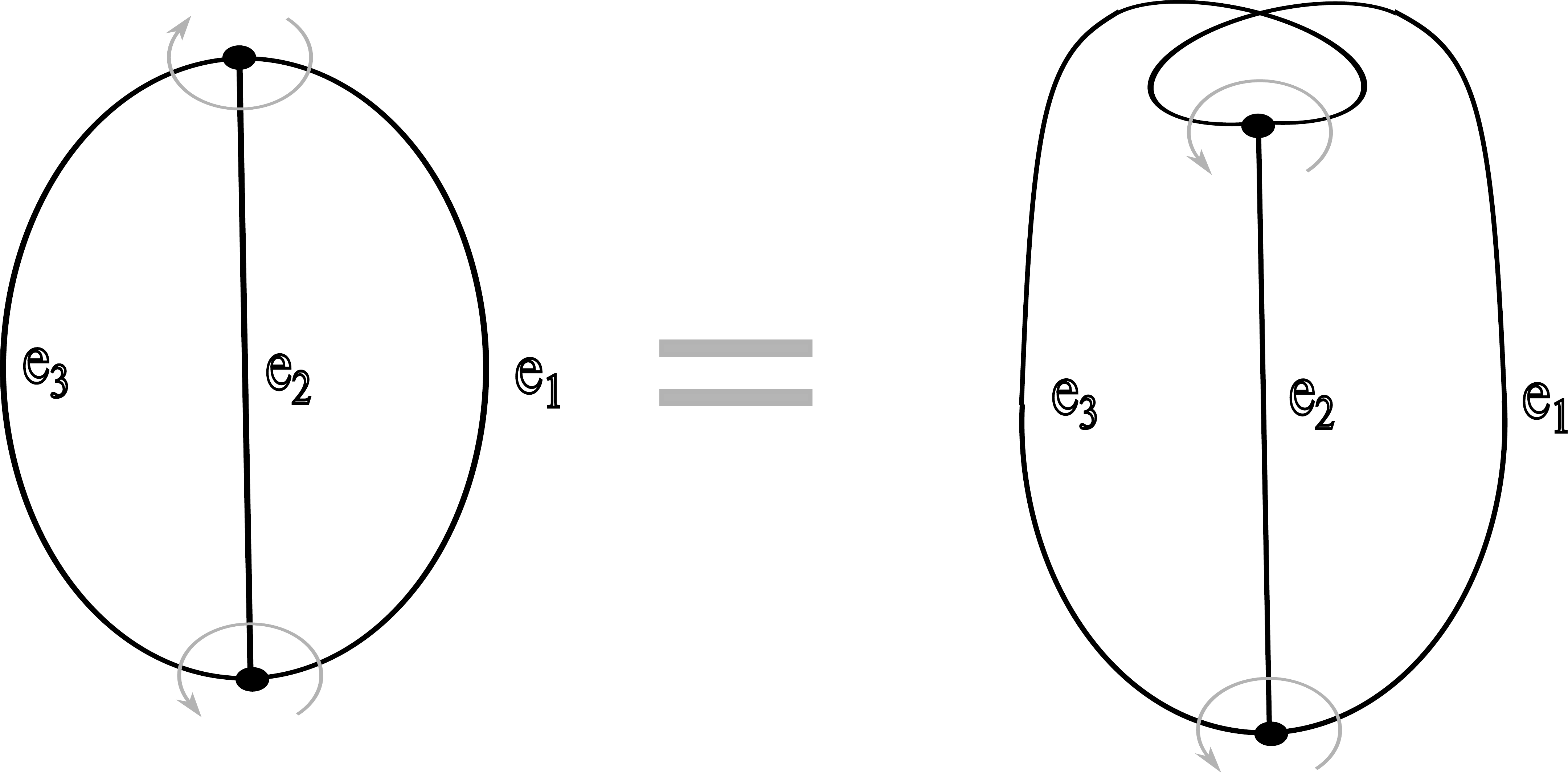}
      \caption{The ribbon graph for Torus in two pictures. Twist allows to keep the orientation fixed on each vertex.}
      \label{fig:ribbon_torus}
  \end{figure}


\subsection{Reinterpretation and obstruction:}

Finally, it will be interesting to reinterpret the ideas developed here in connection with $(\infty, n)$ categories as in [~\cite{lurie2008classification}, Section-4.3] and in connection with constructible sheaf [~\cite{freed2022topological}, section-2.4, 2.5]
and topological symmetry in QFT as in ~\cite{freed2022topological} and ~\cite{freed2024introduction}. It may shed some light on the obstruction of the coloring process. Since, \cref{4-color} can be stated as the obstruction for a planar trivalent graph to admit a global section to $\textit{Bord}_2^{def, cw}(\mathcal{P}_{K_4})$ are solely given by its bridges. Finally, from the higher bundle interpretation as outlined in \cref{sec:additional}, incorporating the idea of $2$-connection developed in ~\cite{baez2005higher}, may also be very useful in studying the obstruction.



\bibliographystyle{alpha}
\bibliography{refs}

\end{document}